\documentclass[a4paper,12pt]{amsart}

\usepackage{arrayjobx}
\usepackage{fullpage}
\usepackage{amsmath,amssymb,graphicx,psfrag}
\usepackage[T1]{fontenc}
\usepackage{amsmath,amssymb,ifthen}
\usepackage{color}
\usepackage{tikz}
\usetikzlibrary{arrows,chains,matrix,positioning,scopes}
\usetikzlibrary{decorations.markings,calc}

\usepackage{subfigure}
\usepackage{pgfplots}
\usepackage{pgfplotstable}
\usetikzlibrary{plotmarks}

\newcommand{\norm}[3][]{#1\|#2#1\|_{#3}}
\newcommand{\set}[3][\big]{#1\{#2\,:\,#3#1\}}

\def\reff#1#2{\!\stackrel{\eqref{#1}}{#2}\!}

\newcounter{statement}
\newenvironment{statement}[2][!]{%
\vskip3mm
\hrule
\hrule
\hrule
\vskip1mm
\noindent%
\refstepcounter{statement}%
\bf#2~\thestatement%
\ifthenelse{\equal{#1}{!}}{.\ }{~(#1).\ }%
\it%
}{%
\vskip1mm
\hrule
\hrule
\hrule
\vskip2mm
}

\newenvironment{theorem}[1][!]{\begin{statement}[#1]{Theorem}}{\end{statement}}
\newenvironment{lemma}[1][!]{\begin{statement}[#1]{Lemma}}{\end{statement}}

\newenvironment{remark}[1][!]{\begin{statement}[#1]{Remark}}{\end{statement}}
\newenvironment{algorithm}[1][!]{\begin{statement}[#1]{Algorithm}}{\end{statement}}

\renewcommand{\subsection}[1]{%
 \vskip2mm
 \refstepcounter{subsection}%
 {\bf\arabic{section}.\arabic{subsection}.~#1.~}%
}

\usepackage{fancyhdr}
\advance\footskip0.4cm
\advance\textheight-0.4cm
\calclayout
\newcommand*\patchAmsMathEnvironmentForLineno[1]{%
  \expandafter\let\csname old#1\expandafter\endcsname\csname #1\endcsname
  \expandafter\let\csname oldend#1\expandafter\endcsname\csname end#1\endcsname
  \renewenvironment{#1}%
     {\linenomath\csname old#1\endcsname}%
     {\csname oldend#1\endcsname\endlinenomath}}%
\newcommand*\patchBothAmsMathEnvironmentsForLineno[1]{%
  \patchAmsMathEnvironmentForLineno{#1}%
  \patchAmsMathEnvironmentForLineno{#1*}}%
\AtBeginDocument{%
\patchBothAmsMathEnvironmentsForLineno{equation}%
\patchBothAmsMathEnvironmentsForLineno{align}%
\patchBothAmsMathEnvironmentsForLineno{flalign}%
\patchBothAmsMathEnvironmentsForLineno{alignat}%
\patchBothAmsMathEnvironmentsForLineno{gather}%
\patchBothAmsMathEnvironmentsForLineno{multline}%
}
\usepackage[mathlines]{lineno}

\title{Optimal convergence behavior of adaptive FEM \\ driven by simple $\boldsymbol{({h}-{h}/2)}$-type error estimators}

\author{Christoph Erath}

\address{TU Darmstadt, Department of Mathematics, Dolivostra\ss{}e 15, 64293 Darmstadt, Germany}
\email{Erath@mathematik.tu-darmstadt.de}

\author{Gregor Gantner}
\author{Dirk Praetorius}

\address{TU Wien, Institute for Analysis and Scientific Computing, Wiedner Hauptstr.~8--10/E101/4, 1040 Wien, Austria}
\email{Gregor.Gantner@asc.tuwien.ac.at \quad\rm (corresponding author)}
\email{Dirk.Praetorius@asc.tuwien.ac.at}


\date{\today}

\begin{document}

\begin{abstract}
For some Poisson-type model problem, we prove that adaptive FEM  driven by the $(h-h/2)$-type error estimators from [Ferraz-Leite, Ortner, Praetorius, \emph{Numer.\ Math.} 116 (2010)] leads to convergence with optimal algebraic convergence rates. Besides the implementational simplicity, another striking feature of these estimators is that they can provide guaranteed lower bounds for the energy error with known efficiency constant $1$.
\end{abstract}
\maketitle

\section{Introduction}

\noindent 
Let $\Omega \subset \mathbb{R}^d$ with $d\ge2$ be a bounded Lipschitz domain with polyhedral boundary $\Gamma := \partial\Omega$.
Given $f\in L^2(\Omega)$, let $u \in H^1_0(\Omega)$ be the unique weak solution 
\begin{align}\label{eq:strong}
 -{\rm div}(\boldsymbol{A} \nabla u) = f
 \text{ in } \Omega
 \quad\text{subject to Dirichlet boundary conditions}\quad
 u = 0
 \text{ on } \Gamma,
\end{align}
where $\boldsymbol{A}:\Omega\to\mathbb{R}^{d\times d}$ is piecewise constant on some initial conforming triangulation $\mathcal{T}_0$ and maps into the space of symmetric positive definite matrices. 

Based on a conforming simplicial triangulation $\mathcal{T}_\ell$, we consider the $H^1$-conforming FE space of $\mathcal{T}_\ell$-piecewise polynomials of degree $p \ge 1$. Let $u_\ell$ be the corresponding FEM solution. Throughout, the index $\ell \in \mathbb{N}_0 := \{0,1,2,\dots\}$ denotes the step of the adaptive algorithm. 
Due to singularities of the (unknown) exact solution, uniform mesh-refinement usually leads a suboptimal convergence behavior of the energy norm error $\norm{\boldsymbol{A}^{1/2}\nabla( u-U_\ell)}{\Omega}$, where $\norm{\cdot}{\Omega} := \norm{\cdot}{L^2(\Omega)}$.
However, the appropriate grading of the triangulation $\mathcal{T}_\ell$ has the potential to lead to the optimal convergence rate ${\mathcal{O}}((\#\mathcal{T}_\ell)^{-p/d})$ with respect to the number of elements $\#\mathcal{T}_\ell$. Such a mesh-grading can automatically be generated by adaptive mesh-refining algorithms of the type
\begin{align}\label{eq:semr}
\boxed{\rm~SOLVE~} 
\longrightarrow \boxed{\rm~ESTIMATE~}
\longrightarrow \boxed{\rm~MARK~}
\longrightarrow \boxed{\rm~REFINE~}
\end{align}
In the last two decades, the mathematical understanding of adaptive algorithms has matured. Starting with the first convergence results in~\cite{doerfler,mns}, it is meanwhile known that the adaptive algorithm, driven by the canonical residual error estimator, leads to linear convergence with optimal algebraic rates; see, e.g.,~\cite{stevenson07,ckns,ffp14}. The same result holds for any estimator, which is \emph{locally equivalent} to the residual error estimator~\cite{ks,axioms}, where the analysis strongly exploits this local equivalence. Examples for locally equivalent estimators include hierarchical error estimators, averaging estimators, and equilibrated fluxes.

\def\qsat{q_{\rm sat}}
The current work considers $(h-h/2)$-type error estimators which are \emph{only globally, but not locally equivalent} to residual error estimators. This error estimation strategy is a well-known technique; see~\cite{hnw87} for ordinary differential equations and the works of Bank~\cite{bank85,bank93,bank96} or the monograph~\cite[Chapter~5]{aoAposteriori} in the context of FEM. 
Let $\widehat{\mathcal{T}}_\ell$ be the uniform refinement of $\mathcal{T}_\ell$. Let $\widehat u_\ell$ be the corresponding FE solution. 
The natural $(h-h/2)$-error estimator 
\begin{align}\label{intro:hh2estimator}
 \widetilde\mu_\ell := \norm{\boldsymbol{A}^{1/2} \nabla (\widehat u_\ell - u_\ell)}{\Omega}
\end{align}
is a computable quantity which can be used to estimate the error $\norm{\boldsymbol{A}^{1/2} \nabla (u - u_\ell)}{\Omega}$. According to the Galerkin orthogonality, it holds that
\begin{align}\label{intro:pythagoras}
 \norm{\boldsymbol{A}^{1/2} \nabla (u - \widehat u_\ell)}{\Omega}^2
 + \widetilde\mu_\ell^{\,2}
 = \norm{\boldsymbol{A}^{1/2} \nabla (u - u_\ell)}{\Omega}^2.
\end{align}
From this, it is easy to see that
\begin{align}\label{intro:efficiency+reliability}
 \widetilde\mu_\ell \le \norm{\boldsymbol{A}^{1/2} \nabla (u - u_\ell)}{\Omega}
 \le (1-q_{\rm sat}^2)^{-1/2}
\, \widetilde\mu_\ell.
\end{align}
The upper bound requires and is even equivalent to the so-called \emph{saturation assumption}
\begin{align}\label{intro:saturation_assumption}
 \norm{\boldsymbol{A}^{1/2} \nabla (u \!-\! \widehat u_\ell)}{\Omega}
 \le \qsat \, \norm{\boldsymbol{A}^{1/2} \nabla (u \!-\! u_\ell)}{\Omega}
 \quad \text{with some uniform} \quad 
  0 < \qsat < 1.
\end{align}
We remark that \eqref{intro:saturation_assumption} dates back to the early work~\cite{bank85}, but may fail to hold in general~\cite{bek96,dn02} and is essentially equivalent to asymptotic behavior of the FEM; see the discussion in~\cite[Section~5.2]{fp2008} and Remark~\ref{rem:estimator} below. However, under certain assumptions on the polynomial degree $p$ and/or the mesh-refinement (e.g., $d=2$ with bisec5-refinement or $d=2$ with $p\ge2$ and bisec3-refinement), one can rigorously prove that
\begin{align}\label{intro:reliability}
 \big( \widetilde\mu_\ell^{\,2} + {\rm osc}_\ell(f)^2 \big)^{1/2} 
 \le \big( \norm{\boldsymbol{A}^{1/2} \nabla (u - u_\ell)}{\Omega}^2 + {\rm osc}_\ell(f)^2 \big)^{1/2}
 \le C_{\rm rel} \, \big( \widetilde\mu_\ell^{\,2} + {\rm osc}_\ell(f)^2 \big)^{1/2},
\end{align}
where ${\rm osc}_\ell(f)$ denote the data oscillations; see Theorem~\ref{thm:abstract} below, where we extend an idea from~\cite{doerfler,mns}. 
We stress that the counter examples from \cite{bek96,dn02} show that \eqref{intro:reliability} requires the inner node property (bisec5-refinement for $d=2$), if $p=1$.
Having to compute $\widehat u_\ell$, it is not attractive to compute the less accurate $u_\ell$; cf.~\eqref{intro:pythagoras}. In this work, we thus consider variants of the $h-h/2$ error estimator from~\cite{fop}, which avoid this computation, e.g.,
\begin{align}\label{eq:intro}
 \eta_\ell = \big( \norm{(1 - \pi_\ell) \boldsymbol{A}^{1/2} \nabla \widehat u_\ell}{\Omega}^2 + {\rm osc}_\ell(f)^2 \big)^{1/2},
\end{align}
where $\pi_\ell$ is the $\mathcal{T}_\ell$-elementwise $L^2$-projection onto polynomials of degree $p-1$ (see~\eqref{table} below for further variants). We prove that
\begin{align}\label{intro2:efficiency+reliability}
 \eta_\ell \le \big( \norm{\boldsymbol{A}^{1/2} \nabla (u - u_\ell)}{\Omega}^2 + {\rm osc}_\ell(f)^2 \big)^{1/2}
 \le C_{\rm rel}C_{\rm hh2} \, \eta_\ell.
\end{align}
\def\Clin{C_{\rm lin}}%
\def\Copt{C_{\rm opt}}%
\def\qlin{q_{\rm lin}}%
It is thus a particular strength of this approach that $\eta_\ell$ is a computable guaranteed lower bound for the \emph{total error} even with known constant $1$. Using this estimator (or one of its variants~\eqref{table}) in the adaptive algorithm (see Algorithm~\ref{algorithm} for the precise statement), we prove that the error estimator (or equivalently: the total error) is linearly convergent with optimal algebraic rates, i.e., \begin{align}\label{intro:lin}
 \eta_{\ell+n} \le \Clin \qlin^n \, \eta_\ell
 \quad \text{for all } \ell, n \in \mathbb{N}_0
\end{align}
and, for all possible algebraic rates $s > 0$, 
\begin{align}\label{intro:opt}
 \eta_\ell \le \Copt \, (\#\mathcal{T}_\ell)^{-s}
\end{align}
with certain constants $\Clin, \Copt > 0$ and $0 < \qlin < 1$.
 \emph{Possible algebraic rates} are, as usually, characterized in terms of certain approximation classes which are the same as those for residual error estimators. In explicit terms, the simple $(h-h/2)$-type error estimators thus yield the same optimal convergence behavior as the residual error estimators, even though these two types of estimators are not locally equivalent.

{\bf Outline.}\quad
In Section~\ref{section:main_result}, we collect the mathematical framework to formally state our main results. To this end, we formulate the precise assumptions on the conforming triangulations and the mesh-refinement (Section~\ref{section:triangulations}), define the employed FEM spaces (Section~\ref{section:fem}), introduce the considered $(h-h/2)$-type error estimators (Section~\ref{section:hh2}) and the corresponding adaptive algorithm (Algorithm~\ref{algorithm} as a precise specification of~\eqref{eq:semr}), and formulate the main result (Theorem~\ref{thm:abstract} which gives the formal statement of~\eqref{intro2:efficiency+reliability} as well as~\eqref{intro:lin}--\eqref{intro:opt}). For the proof of Theorem~\ref{thm:abstract}, we rely on certain properties of the residual error estimator. These are collected and proved in Section~\ref{section:residual}, where we slightly improve the discrete reliability estimate from~\cite{stevenson07,ckns} as well as the discrete efficiency estimate from~\cite{doerfler,mns}. The proof of Theorem~\ref{thm:abstract} is given in Section~\ref{section:proof}. Finally, we underline the theoretical findings by some numerical experiments in Section~\ref{section:numerics}.

\bigskip

{\bf General notation.}\quad
Throughout, we write $a \lesssim b$ to abbreviate $a \le  Cb$ with some generic constant $C> 0$, which is clear from the context. Moreover, $a \simeq b$ abbreviates $a \lesssim b \lesssim a$. Mesh-related quantities have the same index, e.g., $u_\star$ is the FEM solution corresponding to the triangulation $\mathcal{T}_\star$, and $\mathcal{E}_\bullet$ is the set of facets of the triangulation $\mathcal{T}_\bullet$.
Throughout, we make the following convention: If $\mathcal{T}_\star$ is a triangulation and 
$\alpha_\star(T,\cdot) \in \mathbb{R}$ is defined for all $T \in \mathcal{T}_\star$, then
\begin{align}\label{eq:convention}
 \alpha_\star(\cdot) := \alpha_\star(\mathcal{T}_\star,\cdot),
 \quad \text{where} \quad
\alpha_\star({\mathcal{U}}_\star,\cdot)^2 := \sum_{T \in {\mathcal{U}}_\star} \alpha_\star(T,\cdot)^2
 \quad \text{for all } {\mathcal{U}}_\star \subseteq \mathcal{T}_\star.
\end{align}
Finally, $\norm{\cdot}{\omega}^2 := \int_\omega (\cdot)^2\,dx$ abbreviates the $L^2$-norm over a measurable set $\omega$ (with respect to either the $d$-dimensional Lebesgue measure or the $(d-1)$-dimensional surface measure). 

\begin{figure}[t]
 \centering
 \includegraphics[width=35mm]{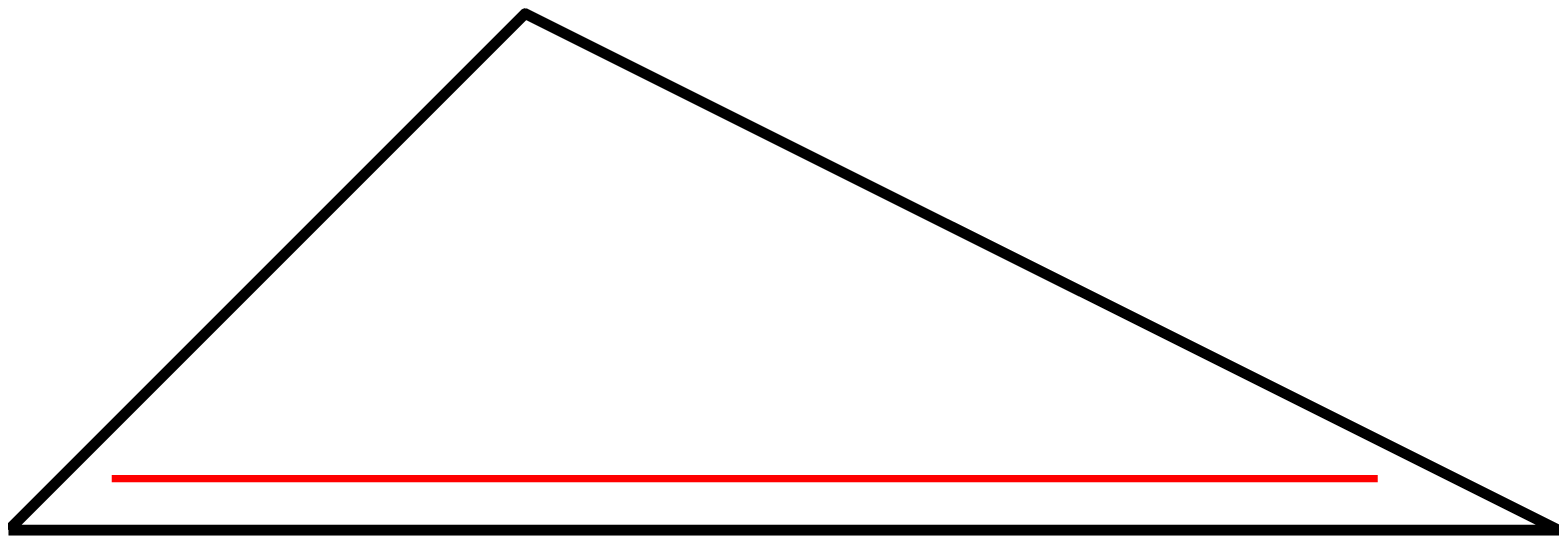} \quad
 \includegraphics[width=35mm]{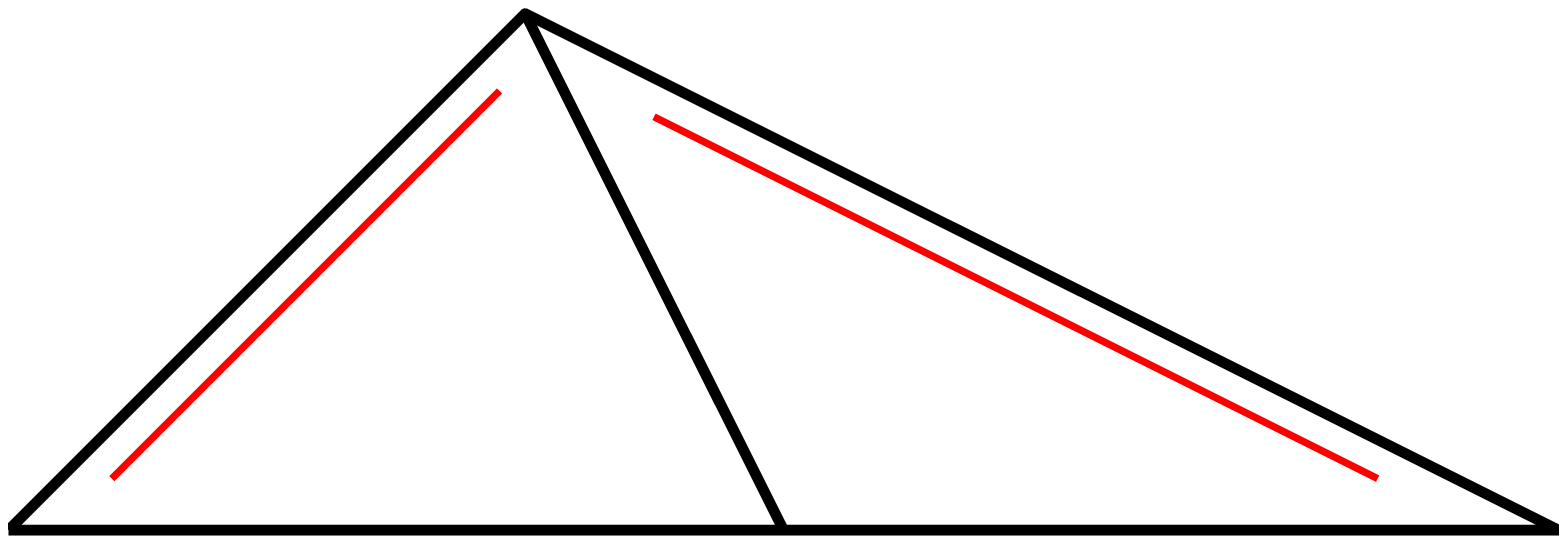} \quad
 \includegraphics[width=35mm]{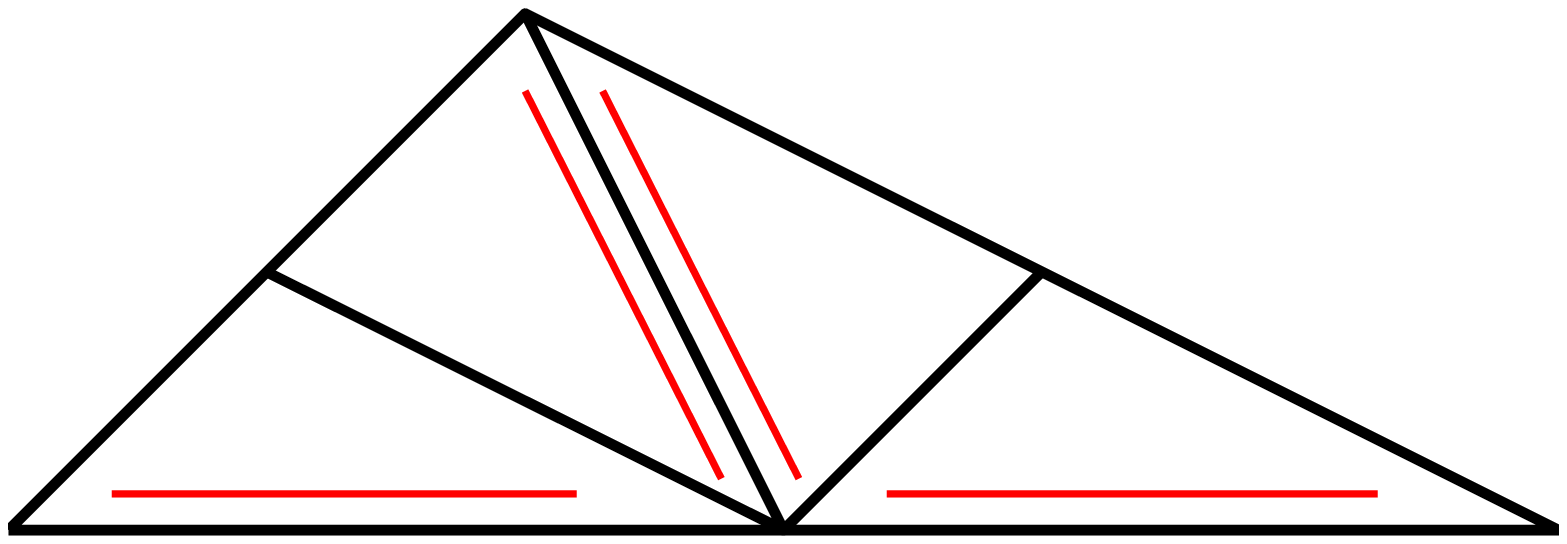} \quad
 \includegraphics[width=35mm]{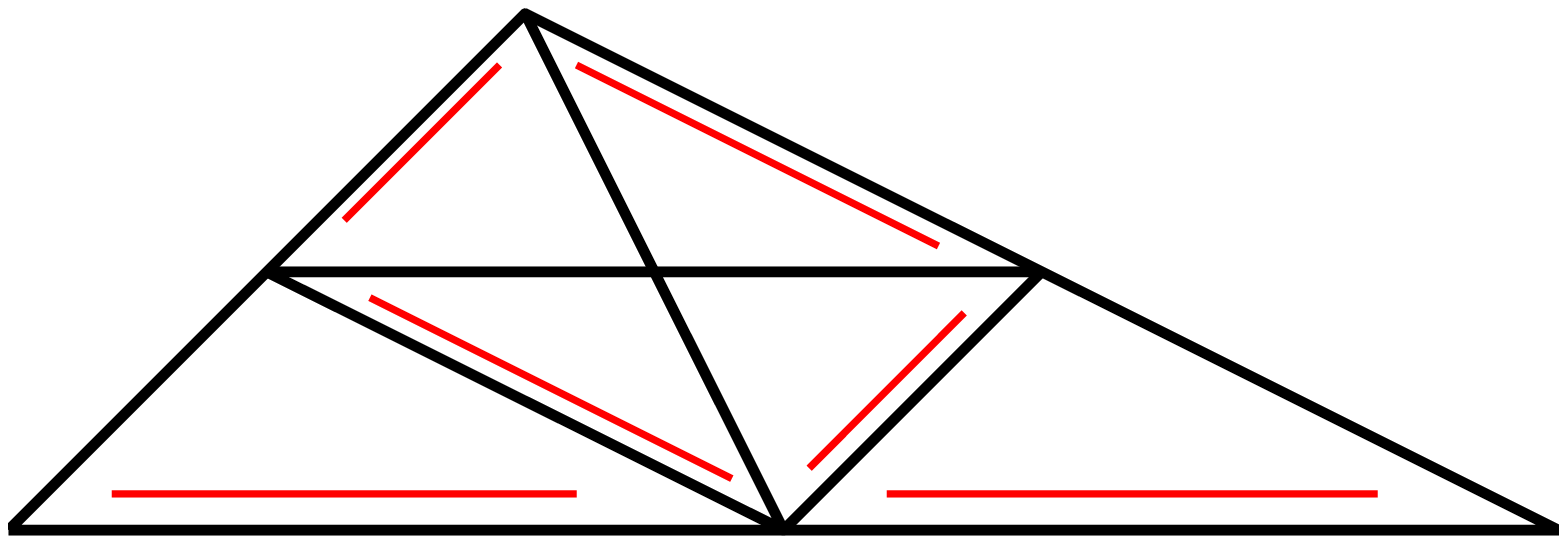} \\
 \caption{\small For NVB in 2D, each triangle $T\in\mathcal{T}_\bullet$ has one \emph{reference edge}, indicated by the double line (left). 
 Bisection of $T$ is achieved by halving the reference edge. The reference edges of the sons are always opposite to the new node. Recursive application of this refinement rule leads to conforming triangulations. It needs three bisections per element to halve all edges of a triangle. Five bisection create 
 an interior node and hence a discrete element bubble function within $T$.}
  \label{fig:nvb}
\end{figure}
%
\newcommand{\flaechena}[1]{
	\filldraw[fill=#1!30,fill opacity=1] (T1_4)--(T1_1)--(T1_3)--cycle;
	\filldraw[fill=#1!30!black!30,fill opacity=1](T1_1)--(T1_2)--(T1_3)--cycle;
	\filldraw[fill=#1!30,fill opacity=1](T1_4)--(T1_1)--(T1_2)--cycle;
	\filldraw[fill=#1!30!black!50,fill opacity=1](T1_2)--(T1_3)--(T1_4)--cycle;
	\draw[color=#1!70!black!100,fill opacity=1, line width=3pt] (T1_4)-- (T1_1);
}
\newcommand{\flaechenb}[1]{

	\filldraw[fill=#1!30,fill opacity=1] (T1_4)-- (T1_1)--(T1_3)--cycle;
	\filldraw[fill=#1!30,fill opacity=1](T1_4)--(T1_1)--(T1_2)--cycle;
	\filldraw[fill=#1!30,fill opacity=1](T1_2)--(T1_3)--(T1_4)--cycle;
	\filldraw[fill=#1!30!black!30,fill opacity=1](T1_1)--(T1_2)--(T1_3)--cycle;

	\draw[dashed,color=#1!70!black!100,fill opacity=1, line width=3pt] (T1_4) -- (T1_1);
}
\newcommand{\flaechenc}[1]{
	\filldraw[fill=#1!20!black!50,fill opacity=1](T1_1)--(T1_2)--(T1_3)--cycle;
	\filldraw[fill=#1!30!black!40,fill opacity=1](T1_2)--(T1_3)--(T1_4)--cycle;
	\filldraw[fill=#1!30!black!50,fill opacity=1] (T1_4)-- (T1_1)--(T1_3)--cycle;
	\filldraw[fill=#1!30,fill opacity=1](T1_4)--(T1_1)--(T1_2)--cycle;
	\draw[color=#1!70!black!100,fill opacity=1, line width=3pt] (T1_4)--(T1_1);
}
\newcommand{\flaechend}[1]{
	\filldraw[fill=#1!30!black!40,fill opacity=1](T1_2)--(T1_3)--(T1_4)--cycle;
	\filldraw[fill=#1!30!black!50,fill opacity=1] (T1_4)-- (T1_1)--(T1_3)--cycle;
	\filldraw[fill=#1!30,fill opacity=1](T1_1)--(T1_2)--(T1_3)--cycle;
	\filldraw[fill=#1!30!black!50,fill opacity=1](T1_4)--(T1_1)--(T1_2)--cycle;
	\draw[color=#1!70!black!100,fill opacity=1, line width=3pt] (T1_4)--(T1_1);
}
\newcommand{\flaechene}[1]{
	\filldraw[fill=#1!30!black!50,fill opacity=1](T1_1)--(T1_2)--(T1_3)--cycle;
	\filldraw[fill=#1!30!black!40,fill opacity=1](T1_2)--(T1_3)--(T1_4)--cycle;
	\filldraw[fill=#1!30!black!50,fill opacity=1] (T1_4)-- (T1_1)--(T1_3)--cycle;
	\filldraw[fill=#1!50!black!30,fill opacity=1](T1_4)--(T1_1)--(T1_2)--cycle;
	\draw[color=#1!70!black!100,fill opacity=1, line width=3pt] (T1_4)--(T1_1);
}
\newcommand{\flaechenf}[1]{
	\filldraw[fill=#1!30,fill opacity=1] (T1_4)-- (T1_1)--(T1_3)--cycle;
	\filldraw[fill=#1!30,fill opacity=1](T1_4)--(T1_1)--(T1_2)--cycle;
	\filldraw[fill=#1!30!black!30,fill opacity=1](T1_2)--(T1_3)--(T1_4)--cycle;
	\filldraw[fill=#1!30!black!50,fill opacity=1](T1_1)--(T1_2)--(T1_3)--cycle;

	\draw[color=#1!70!black!100,dashed,fill opacity=1, line width=3pt] (T1_4) -- (T1_1);
}
\newcommand{\flaecheng}[1]{
	\filldraw[fill=#1!30,fill opacity=1](T1_4)--(T1_1)--(T1_2)--cycle;
	\filldraw[fill=#1!30,fill opacity=1] (T1_4)-- (T1_1)--(T1_3)--cycle;
	\fill[fill=#1!30,fill opacity=1](T1_2)--(T1_3)--(T1_4)--cycle;
	\filldraw[fill=#1!30!black!50,fill opacity=1](T1_1)--(T1_2)--(T1_3)--cycle;
	\draw[color=#1!70!black!100,fill opacity=1, line width=3pt] (T1_4)-- (T1_1);
}
\newcommand{\flaechenh}[1]{
	\filldraw[fill=#1!30,fill opacity=1](T1_1)--(T1_2)--(T1_3)--cycle;
	\filldraw[fill=#1!30!black!50,fill opacity=1](T1_4)--(T1_1)--(T1_2)--cycle;
	\filldraw[fill=#1!30!black!30,fill opacity=1](T1_2)--(T1_3)--(T1_4)--cycle;
	\filldraw[fill=#1!30!black!50,fill opacity=1] (T1_4)-- (T1_1)--(T1_3)--cycle;
	\draw[color=#1!70!black!100,fill opacity=1, line width=3pt] (T1_4)--(T1_1);
}
\newcommand{\flaecheni}[1]{
	\filldraw[fill=#1!30,fill opacity=1] (T1_4)-- (T1_1)--(T1_3)--cycle;
	\filldraw[fill=#1!30,fill opacity=1](T1_4)--(T1_1)--(T1_2)--cycle;
	\filldraw[fill=#1!30,fill opacity=1](T1_2)--(T1_3)--(T1_4)--cycle;
	\filldraw[fill=#1!30,fill opacity=1](T1_1)--(T1_2)--(T1_3)--cycle;

	\draw[dashed,color=#1!70!black!100,fill opacity=1, line width=3pt] (T1_4) -- (T1_1);
}
\newcommand{\tet}[7]{
	\tikzset{knoten/.style={circle,draw=black!100,fill=#2!30,inner sep=1pt,minimum size=4mm}}

	\newarray\coord
	\readarray{coord} {#1}
	\newarray\punkti
	\readarray{punkti}{#3&0&0&0}
	\newarray\name
	\readarray{name}{#4}

	\foreach \i in {1,...,4}{
		\checkcoord(\i)
		\coordinate (T1_\i) at (#5+\cachedata+#6);
	}

		\ifcase#7
		\flaechena{#2}\or 
		\flaechenb{#2} \or
		\flaechenc{#2} \or
		\flaechend{#2} \or
		\flaechene{#2} \or
		\flaechenf{#2} \or
		\flaecheng{#2} \or
		\flaechenh{#2} \or
		\flaecheni{#2}
		\fi

	\foreach \i in {1,...,3}{
		\checkpunkti(\i)
		\ifnum \cachedata>0
			\ifcase\cachedata
			\or 
			\draw[dashed] (T1_1) -- (T1_2); \or
			\draw[dashed] (T1_1) -- (T1_3); \or
			\draw[dashed] (T1_1) -- (T1_4); \or
			\draw[dashed] (T1_2) -- (T1_3); \or
			\draw[dashed] (T1_2) -- (T1_4); \or
			\draw[dashed] (T1_3) -- (T1_4);
			\fi
		\fi
	}
	\foreach \i in {1,...,4}{
		\checkname(\i)
		\node[knoten] at (T1_\i) {\cachedata};
	}
}
\begin{figure}[t]
\begin{center}
\begin{tikzpicture}[line join=round,scale=0.45]

\def \abstand {1.2}

\def \xa {12+\abstand /2}
\def \ya {10}

\def \xb {0}
\def \yb {0}

\def \xc {12}
\def \yc {0}

\def \xd {24}
\def \yd {0}

	\tet{-1.5,0&-5.1,6&-5,2&-11,0}{white!80!black}{6&2&4}{4&3&2&1}{\xa}{\ya}{2}
	
	\tet{-11,0 &-6.25,0&-5,2&-5.1,6}{blue}{2}{1&2&3&4}{\xb}{\yb}{0}
	\tet{-1.5,0&-6.25,0&-5.1,6&-5,2}{blue}{5&6}{1&2&3&4}{\xb+\abstand}{\yb}{8}

	\tet{-11,0 &-6.25,0&-5,2&-5.1,6}{red}{2}{1&2&3&4}{\xc}{\yc}{0}
	\tet{-1.5,0&-6.25,0&-5,2&-5.1,6}{red}{6&2&4}{1&2&3&4}{\xc+\abstand}{\yc}{2}
	
	\tet{-11,0 &-6.25,0&-5,2&-5.1,6}{green}{2}{1&2&3&4}{\xd}{\yd}{0}
	\tet{-1.5,0&-6.25,0&-5,2&-5.1,6}{green}{6&2&4}{1&2&3&4}{\xd+\abstand}{\yd}{2}

	\node[rounded corners=4pt,fill=blue!30,draw=black] at (\xb-9,\yb+5){type 1};
	\node[rounded corners=4pt,fill=red!30,draw=black] at (\xc-9,\yc+5){type 2};
	\node[rounded corners=4pt,fill=green!30,draw=black] at (\xd-9,\yd+5){type 0};
	\path (\xa-7,\ya+3) edge[-latex,pos=0.2, in=90,out=170] node[above,rounded corners=4pt,fill=green!30,draw=black]{type  $\tau=0$} (\xb-4.5,\yb+7);
	\path (\xa-5.7,\ya+3) edge[-latex,pos=0.7, in=90,out=220] node[left,rounded corners=4pt,fill=blue!30,draw=black]{type  $\tau=1$} (\xc-4.5,\yc+7);
	\path (\xa-4,\ya+3) edge[-latex,pos=0.2, in=90,out=10] node[above,rounded corners=4pt,fill=red!30,draw=black]{type  $\tau=2$} (\xd-4.5,\yd+7);

\end{tikzpicture}
\end{center}
\caption{\small For NVB in 3D, each tetrahedron $T\in\mathcal{T}_\bullet$ is assigned with a permutation $(z_{\pi(1)},z_{\pi(2)},z_{\pi(3)},z_{\pi(4)})$ of its vertices $\{z_1,z_2,z_3,z_4\}$ and a type $\tau\in\{0,1,2\}$.  
The numbers in the figure are the positions of the nodes in the corresponding tuple. 
 Bisection of $T$ is achieved by halving the \emph{reference edge} between $z_{\pi(1)}$ and $z_{\pi(4)}$ indicated by the bold line. 
The permutations as well as the types of its sons depend on the permuation and the type of $T$.
 Recursive application of this refinement rule leads to conforming triangulations.}
\label{fig:nvb3d}
\end{figure}
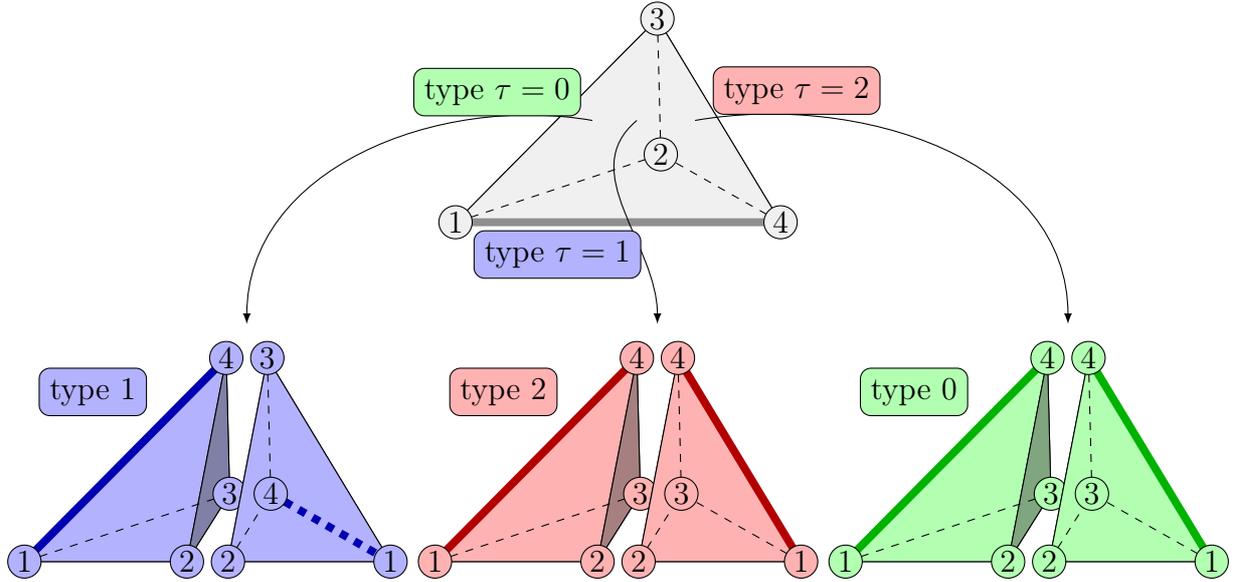


%
\section{Main result}
\label{section:main_result}

\subsection{Conforming triangulations and mesh-refinement}
\label{section:triangulations}
Throughout, $\mathcal{T}_\bullet$ denotes a conforming triangulation of $\Omega$ into non-degenerate compact simplices. In particular, we avoid hanging nodes. The triangulation is called \emph{$\gamma$-shape regular}, if
\begin{align}\label{eq:shaperegular}
 \max_{T\in\mathcal{T}_\bullet} \frac{{\rm diam}(T)}{h_T} \le \gamma.
\end{align}
Here, ${\rm diam}(T)$ denotes the Euclidean diameter of $T$ and $h_T:=|T|^{1/d}$ with $|T|$ being its $d$-dimensional volume. Note that $\gamma$-shape regularity implies that $h_T \le {\rm diam}(T) \le \gamma \, h_T$. 

For given $\mathcal{T}_\bullet$, let ${\mathcal{N}}_\bullet$ be the set of nodes and $\mathcal{E}_\bullet$ be the set of facets. 
For $E \in \mathcal{E}_\bullet$, we define $h_E :=|E|^{1/(d-1)}$ with $|\cdot|$ being the $d$-dimensional surface measure. Note that $h_E \simeq {\rm diam}(T)$ if $E \subset T \in \mathcal{T}_\bullet$, where the hidden constants depend only on $\gamma$. 
Finally, $\mathcal{E}_\bullet^\Omega$ denotes the set of all interior facets, i.e., $E \in \mathcal{E}_\bullet^\Omega\subset\mathcal{E}_\bullet$ satisfies that  $E = T \cap T'$ for certain simplices $T, T' \in \mathcal{T}_\bullet$.

Throughout, we employ newest vertex bisection (NVB) to refine triangulations locally; see~\cite{stevenson,kpp} for details on the refinement algorithm. 
Figure~\ref{fig:nvb} and Figure~\ref{fig:nvb3d} give an  illustration for $d=2$ and $d=3$, respectively . For a conforming triangulation $\mathcal{T}_\bullet$ and $\mathcal{M}_\bullet \subseteq \mathcal{T}_\bullet$, let $\mathcal{T}_\circ := {\rm nvb}(\mathcal{T}_\bullet, \mathcal{M}_\bullet)$ be the coarsest conforming triangulation such that all marked elements $T\in\mathcal{M}_\bullet$ have been refined, i.e., $\mathcal{M}_\bullet \subseteq \mathcal{T}_\bullet \backslash \mathcal{T}_\circ$.
We write $\mathcal{T}_\circ \in {\rm nvb}(\mathcal{T}_\bullet)$, if there exists $n\in\mathbb{N}_0$, conforming triangulations $\mathcal{T}_{(0)},\dots,\mathcal{T}_{(n)}$, and corresponding sets of marked elements $\mathcal{M}_j\subseteq\mathcal{T}_j$ such that
\begin{itemize}
\item $\mathcal{T}_\bullet = \mathcal{T}_{(0)}$,
\item $\mathcal{T}_{(j+1)} = {\rm nvb}(\mathcal{T}_{(j)},\mathcal{M}_{(j)})$ for all $j=0,\dots,n-1$,
\item $\mathcal{T}_\circ  = \mathcal{T}_{(n)}$,
\end{itemize}
i.e., $\mathcal{T}_\circ$ is obtained from $\mathcal{T}_\bullet$ by finitely many refinement steps.

The analysis of the $(h-h/2)$-type error estimators requires a stronger mesh-refinement. 
We suppose that we are given some initial conforming triangulation $\mathcal{T}_0$.
For $\mathcal{T}_\bullet\in{\rm nvb}(\mathcal{T}_0)$, 
let $\mathcal{T}_\circ := {\rm refine}(\mathcal{T}_\bullet, \mathcal{M}_\bullet)$ be an NVB refinement which satisfies:
\begin{itemize}
\item[\bf(M1)] There exists a uniform constant $C_{\rm son}>0$ such that $\#\set{T'\in\mathcal{T}_\circ}{T'\subseteq T}\le C_{\rm son}$ for all $T\in\mathcal{T}_\bullet$, i.e., the number of sons per element is uniformly bounded.
\item[\bf(M2)] If $\mathcal{T}_\blacktriangle\in{\rm nvb}(\mathcal{T}_0)$, $\mathcal{M}_\blacktriangle\subseteq\mathcal{T}_\blacktriangle$, and $\mathcal{T}_\vartriangle:={\rm refine}(\mathcal{T}_\blacktriangle,\mathcal{M}_\blacktriangle)$, it holds that 
\begin{align*}
\set{T'\in\mathcal{T}_\circ}{T'\subset T}=\set{T'\in\mathcal{T}_\vartriangle}{T'\subset T}\quad\text{for all }T\in\mathcal{M}_\bullet\cap\mathcal{M}_\blacktriangle,
\end{align*}
i.e., refinement of a marked element is independent of its neighbors.
\end{itemize}
Further, we suppose that it satisfies one of the following constrains:
\begin{itemize}
\item[\bf(M3)] All facets of $T \in \mathcal{M}_\bullet$ contain an interior node $z \in {\mathcal{N}}_\circ$.
\item[\bf(M3')] All facets of $T \in \mathcal{M}_\bullet$ as well as $T$ contain an interior node $z \in {\mathcal{N}}_\circ$.
\end{itemize}
As above, we let $\mathcal{T}_\circ \in {\rm refine}(\mathcal{T}_\bullet)$ be the set of all possible refinements.

For $d=2$, (M3) corresponds to refinement of marked elements $T \in \mathcal{M}_\bullet$ by at least 3  bisections, while (M3') follows from at least 5 bisections; cf.~Figure~\ref{fig:nvb}. 
Obviously, these refinements also satisfy (M1)--(M2); cf.~Figure~\ref{fig:nvb}.
For $d=3$, each $T={\rm conv}\{z_1,\dots,z_4\}\in\mathcal{T}_\bullet$ is assigned with a permutation $(z_{\pi(1)},z_{\pi(2)},z_{\pi(3)},z_{\pi(4)})$ of its nodes $\{z_1,z_2,z_3,z_4\}$  and a type $\tau\in\{0,1,2\}$; see Figure~\ref{fig:nvb3d}. 
To achieve (M3), one can bisect each marked element $T\in\mathcal{M}_\bullet$ depending on its type as follows:
\begin{itemize}
\item[$\boldsymbol{\tau=0}$:] First, bisect $T$ uniformly into 8 sons, then bisect all resulting sons which do not contain $z_{\pi(1)}$ nor $z_{\pi(4)}$, finally, bisect all resulting sons which either contain the two nodes $z_{\pi(2)}$ and $\frac{1}{2}(z_{\pi(1)}+z_{\pi(3)})$ or the two nodes $z_{\pi(3)}$ and $\frac{1}{2}(z_{\pi(2)}+z_{\pi(4)})$.
Altogether, $T$ is split into 18 sons with $14$ nodes.
\item[$\boldsymbol{\tau=1}$:] First, bisect $T$ uniformly into 8 sons, then bisect all resulting sons which do not contain $z_{\pi(1)}$ nor $z_{\pi(4)}$, finally, bisect all resulting sons which contain $z_{\pi(2)}$.
Altogether, $T$ is split into 18 sons with $14$ nodes.
\item[$\boldsymbol{\tau=2}$:] First, bisect $T$ uniformly into 8 sons, then bisect all resulting sons which do not contain $z_{\pi(1)}$ nor $z_{\pi(4)}$, finally, bisect all resulting sons  which contain $z_{\pi(2)}$.
Altogether, $T$ is split into 20 sons with $14$ nodes.
\end{itemize}
The resulting sons of $T$ are visualized in Figure~\ref{fig:bisec17}.
Note that the proposed strategy satisfies (M1)--(M2) with $C_{\rm son}=20$.

\begin{figure}[t]
 \centering
 \begin{minipage}{.3\textwidth}
 \centering
 \includegraphics[width=52mm]{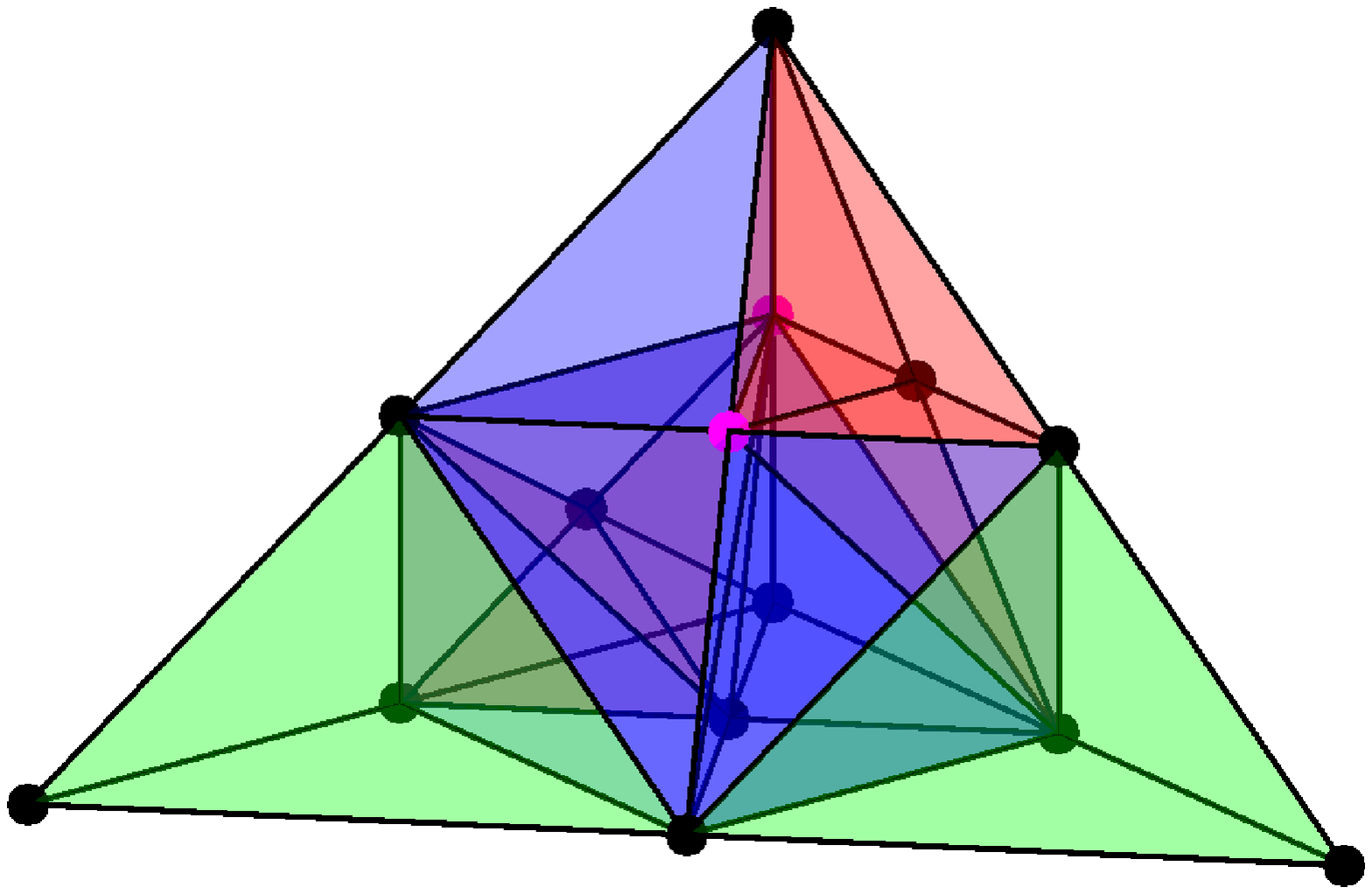}%
 \\[-8mm]%
 \hspace*{4mm} type $\tau = 0$
 \end{minipage}
 \begin{minipage}{.3\textwidth}
 \centering  
 \includegraphics[width=52mm]{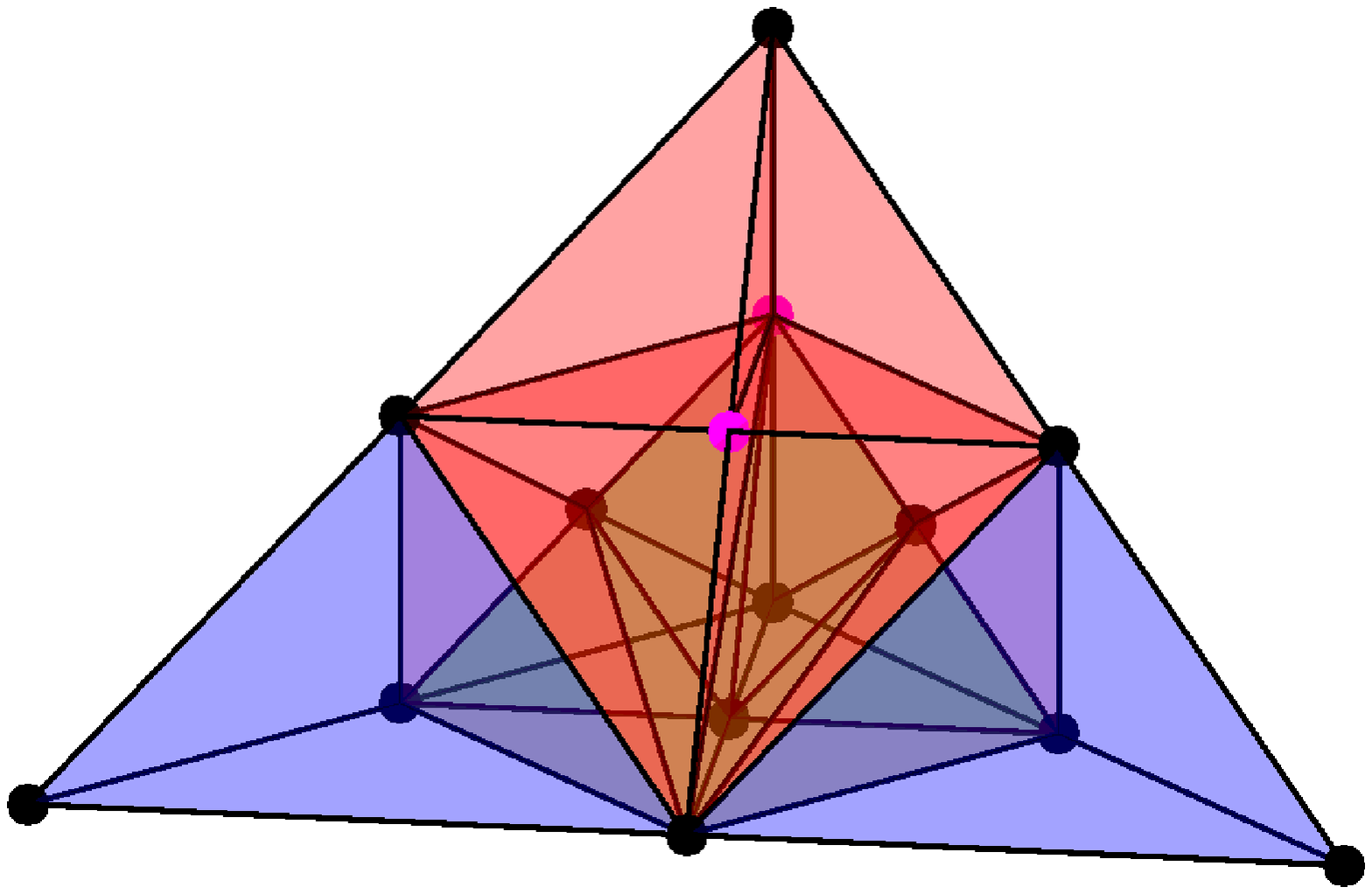}%
 \\[-8mm]%
 \hspace*{4mm} type $\tau = 1$
 \end{minipage}
 \begin{minipage}{.3\textwidth}
 \centering  
 \includegraphics[width=52mm]{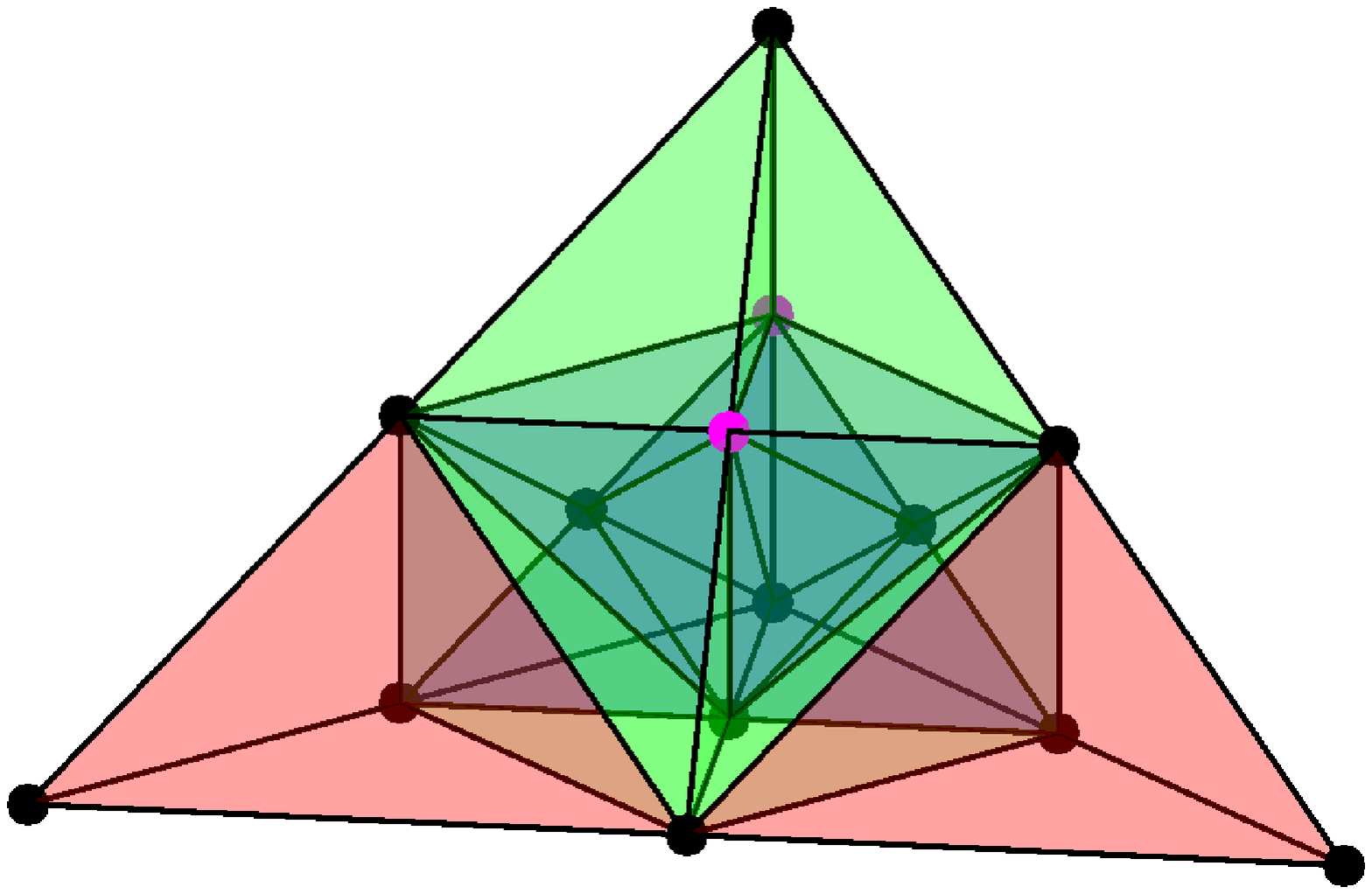}%
 \\[-8mm]%
 \hspace*{4mm} type $\tau = 2$
 \end{minipage}
 \caption{\small 
Starting with the same configuration for $T$ as in Figure~\ref{fig:nvb3d}, the resulting sons for the 3D refinement guaranteeing (M3) are depicted.
The outcome depends on the type $\tau$ of $T$.
The type of the sons is indicated by their color, green for $\tau=0$, blue for $\tau=1$, and red for $\tau=2$.
Finally, two nodes are highlighted in magenta. 
The product of their hat functions is a discrete element bubble function within $T$.}
  \label{fig:bisec17}
\end{figure}

\begin{remark}
{\rm(i)} 
We came up with this refinement by considering all possible configurations of the element $T$ in our MATLAB implementation of 3D NVB.
Indeed, it is sufficient to consider only 4 node permutations instead of $4!$, since the others can be obtained by rotating the element.
Hence, the number of all possible configurations is $4\cdot3=12$.
This refinement leads to 5 two-dimensional NVBs of each facet of $T$ as in Figure~\ref{fig:nvb}.
In particular, uniform refinement with $\mathcal{M}_\bullet=\mathcal{T}_\bullet$ leads to a conforming triangulation.
For $\mathcal{M}_\bullet\neq\mathcal{T}_\bullet$, further bisections are required to obtain conformity.
However, since uniform refinement automatically guarantees conformity, only non-marked elements have to be additionally bisected.

{\rm(ii)} 
Using our {\rm MATLAB} implementation of 3D NVB, we saw that it is not possible to satisfy {\rm(M3')} strictly in the sense that ${\rm refine}(\cdot)$ generates exactly one interior node  per facet and exactly one interior node in each marked element $T\in\mathcal{M}_\bullet$.
Indeed, this is only possible for $T$ being of type  $\tau\in\{0,2\}$,  while type $\tau=1$ enforces even three interior nodes on one facet, if interior nodes on each facet and inside of $T$ are generated.
\end{remark}

\subsection{Finite element method}
\label{section:fem}
The Lax--Milgram theorem proves existence and uniqueness of $u \in H^1_0(\Omega)$ with
\begin{align}\label{eq:weak}
 \int_\Omega \boldsymbol{A}\nabla u \cdot \nabla v \,dx
 = \int_\Omega f v \,dx
 \quad \text{for all } v \in H^1_0(\Omega),
\end{align}
which is the variational formulation of~\eqref{eq:strong}. 
Given a triangulation $\mathcal{T}_\bullet$ and $p\in\mathbb{N}$, define the space of $\mathcal{T}_\bullet$-piecewise polynomials
\begin{align}\label{eq:Pp}
 \mathcal{P}^p(\mathcal{T}_\bullet) := \set{v_\bullet \in L^2(\Omega)}{\forall T \in \mathcal{T}_\bullet\quad v_\bullet |_T \text{ is a polynomial of degree} \le p }.
\end{align}
Define ${\mathcal{S}}^p(\mathcal{T}_\bullet)
 := \mathcal{P}^p(\mathcal{T}_\bullet) \cap H^1(\Omega)
 = \mathcal{P}^p(\mathcal{T}_\bullet) \cap C(\Omega)$ as well as the $H^1$-conforming FE space 
\begin{align} \label{eq:discretespace}
 {\mathcal{S}}^p_0(\mathcal{T}_\bullet) 
 := \mathcal{P}^p(\mathcal{T}_\bullet) \cap H^1_0(\Omega) 
 = \set{ v_\bullet \in {\mathcal{S}}^p(\mathcal{T}_\bullet) }{ v_\bullet |_\Gamma = 0}.
\end{align}
The Lax--Milgram theorem proves existence and uniqueness of $u_\bullet \in {\mathcal{S}}^p_0(\mathcal{T}_\bullet) $ such that
\begin{align}\label{eq:discrete}
 \int_\Omega \boldsymbol{A}\nabla u_\bullet \cdot \nabla v_\bullet \,dx
 = \int_\Omega f v_\bullet \,dx
 \quad \text{for all } v_\bullet \in {\mathcal{S}}^p_0(\mathcal{T}_\bullet) .
\end{align}
Recall the Galerkin orthogonality 
\begin{align}\label{eq:galerkin}
 \int_\Omega \boldsymbol{A}\nabla (u-u_\bullet) \cdot \nabla v_\bullet\,dx = 0
 \quad \text{for all } v_\bullet \in {\mathcal{S}}^p_0(\mathcal{T}_\bullet),
\end{align}
which results in the Pythagoras theorem
\begin{align}\label{eq:pythagoras}
  \norm{\boldsymbol{A}^{1/2}\nabla (u \!-\! u_\bullet)}{\Omega}^2 + \norm{\boldsymbol{A}^{1/2}\nabla (u_\bullet \!-\! v_\bullet)}{\Omega}^2
 = \norm{\boldsymbol{A}^{1/2}\nabla (u \!-\! v_\bullet)}{\Omega}^2
 \quad \text{for all } v_\bullet \in {\mathcal{S}}^p_0(\mathcal{T}_\bullet).
\end{align}

\subsection{Simple $\boldsymbol{(h-h/2)}$-type error estimators}
\label{section:hh2}
Given a triangulation $\mathcal{T}_\bullet$, let $\widehat{\mathcal{T}}_\bullet := {\rm refine}(\mathcal{T}_\bullet,\mathcal{T}_\bullet)$ be the uniform refinement. 
To define the error estimators, we require the following three operators: Let $I_\bullet : C(\overline\Omega)\to {\mathcal{S}}^p(\mathcal{T}_\bullet)$ denote the nodal interpolation operator. 
Let $\pi_\bullet : L^2(\Omega) \to \mathcal{P}^{p-1}(\mathcal{T}_\bullet)$ be the $L^2(\Omega)$-orthogonal projection onto $\mathcal{P}^{p-1}(\mathcal{T}_\bullet)$.
Let $\Pi_\bullet : L^2(\Omega) \to \mathcal{P}^{\max\{p-2,0\}}(\mathcal{T}_\bullet)$ be the $L^2(\Omega)$-orthogonal projection onto $\mathcal{P}^{\max\{p-2,0\}}(\mathcal{T}_\bullet)$.
Recall the natural $h-h/2$ error estimator 
\begin{align*}
 \widetilde\mu_\bullet = \norm{\boldsymbol{A}^{1/2}\nabla(\widehat u_\bullet - u_\bullet)}{\Omega}.
\end{align*}
One drawback of $\widetilde\mu_\bullet$ is that it requires to compute two FE solutions $\widehat u_\bullet \in {\mathcal{S}}^p_0(\widehat{\mathcal{T}}_\bullet)$ and $u_\bullet \in {\mathcal{S}}^p_0(\mathcal{T}_\bullet)$, even though the Pythagoras theorem~\eqref{eq:pythagoras} predicts that
\begin{align}\label{eq:hh2:pythagoras}
 \norm{\boldsymbol{A}\nabla(u - \widehat u_\bullet)}{\Omega}^2 + \widetilde\mu_\bullet^{\,2}
 = \norm{\boldsymbol{A}^{1/2}\nabla(u - u_\bullet)}{\Omega}^2,
\end{align}
i.e., $\widehat u_\bullet$ is more accurate than $u_\bullet$. One remedy is to replace $u_\bullet$ by some (cheap) postprocessing of $\widehat u_\bullet$ as proposed in~\cite{fop}: 
Recalling the convention~\eqref{eq:convention}, we define, for all $T \in \mathcal{T}_\bullet$ and all $\widehat v_\bullet\in{\mathcal{S}}^p(\widehat{\mathcal{T}}_\bullet)$, 
\begin{align*}
 \mu_\bullet(T,\widehat v_\bullet) 
 &:= \norm{\boldsymbol{A}^{1/2}\nabla(1-I_\bullet)\widehat v_\bullet}{T}
 \quad \text{and} \quad
 \lambda_\bullet(T,\widehat v_\bullet) 
 := \norm{(1-\pi_\bullet)\boldsymbol{A}^{1/2}\nabla\widehat v_\bullet}{T}.
\end{align*}%
Since $\boldsymbol{A}^{1/2}$ is $\mathcal{T}_\bullet$-piecewise constant and $\pi_\bullet$ acts elementwise and componentwise, we immediately see the alternative representation
\begin{align}\label{eq:lambda2}
 \lambda_\bullet(T,\widehat v_\bullet) 
 = \norm{\boldsymbol{A}^{1/2}(1-\pi_\bullet)\nabla\widehat v_\bullet}{T}.
\end{align}
The following lemma is proved in~\cite[Prop.~3]{fop} for $p=1$ and the Poisson model problem by use of scaling arguments, but also holds for general $p\ge1$ and our model problem \eqref{eq:strong}.

\begin{lemma}[simple \!$\boldsymbol{(h-h/2)}$-type error estimators]\label{lem:simple}
There exists $C_{\rm hh2}\ge1$ such that there holds local equivalence
\begin{align}\label{eq1:hh2:equivalent}
 \lambda_\bullet(T,\widehat v_\bullet)
 \le \mu_\bullet(T,\widehat v_\bullet) \le C_{\rm hh2} \, \lambda_\bullet(T,\widehat v_\bullet)
 \quad \text{for all }
 T \in \mathcal{T}_\bullet
 \text{ and all }
 \widehat v_\bullet \in {\mathcal{S}}_0^p(\mathcal{T}_\bullet).
\end{align}
Moreover, for $\widehat v_\bullet = \widehat u_\bullet$, there holds global equivalence
\begin{align}\label{eq2:hh2:equivalent}
 \lambda_\bullet(\widehat u_\bullet)
 \le \widetilde\mu_\bullet
 = \norm{\boldsymbol{A}^{1/2}\nabla (\widehat u_\bullet - u_\bullet)}\Omega
 \le \norm{\boldsymbol{A}^{1/2}\nabla (1-I_\bullet) \widehat u_\bullet}\Omega = \mu_\bullet(\widehat u_\bullet)
 \le C_{\rm hh2} \, \lambda_\bullet(\widehat u_\bullet)
\end{align}
as well as efficiency
\begin{align}\label{eq:hh2:efficient}
 C_{\rm hh2}^{-1} \, \mu_\bullet(\widehat u_\bullet) 
 \le \lambda_\bullet(\widehat u_\bullet) 
 \le \widetilde\mu_\bullet
 \le \norm{\boldsymbol{A}^{1/2}\nabla(u-u_\bullet)}{\Omega}.
\end{align}%
The constant $C_{\rm hh2}$ depends only on $d$, $\boldsymbol{A}$, $p$, and $\gamma$-shape regularity of $\mathcal{T}_\bullet$.
\end{lemma}

\begin{proof}[Sketch of proof]
Note that $\boldsymbol{A}^{1/2}\nabla I_\bullet\widehat v_\bullet, \boldsymbol{A}^{1/2}\nabla u_\bullet \in \mathcal{P}^{p-1}(\mathcal{T}_\bullet)$ and that $\pi_\bullet$ is also the $\mathcal{T}_\bullet$-elementwise best approximation onto $\mathcal{P}^{p-1}(\mathcal{T}_\bullet)$. This proves the first estimate in~\eqref{eq1:hh2:equivalent} as well as the first estimate in~\eqref{eq2:hh2:equivalent}. The second estimate in~\eqref{eq2:hh2:equivalent} follows from $\mathcal{P}^p(\mathcal{T}_\bullet) \subseteq \mathcal{P}^p(\widehat{\mathcal{T}}_\bullet)$ and the best approximation property of the Galerkin solution in the energy norm, since $u_\bullet$ is also a Galerkin approximation to $\widehat u_\bullet$.
Since~\eqref{eq:hh2:efficient} is a direct consequence of~\eqref{eq:hh2:pythagoras} and~\eqref{eq2:hh2:equivalent},
it only remains to prove the second estimate in~\eqref{eq1:hh2:equivalent}, which also implies the third estimate in~\eqref{eq2:hh2:equivalent}. 

Let $T \in \mathcal{T}_\bullet$. Note that $\mu_\bullet(T,\cdot)$ and $\lambda_\bullet(T,\cdot)$ are seminorms on $\mathcal{P}^{p-1}(\widehat{\mathcal{T}}_\bullet)$. Recall that seminorms on finite-dimensional spaces are equivalent if the kernels coincide. For $\widehat v_\bullet \in \mathcal{P}^{p-1}(\widehat{\mathcal{T}}_\bullet)$, it holds that $\mu_\bullet(T,\widehat v_\bullet)=0$ and $\lambda_\bullet(T,\widehat v_\bullet) = 0$,  if and only if $\widehat v_\bullet|_T \in \mathcal{P}^{p-1}(\{T\})$. Hence, we derive the equivalence~\eqref{eq1:hh2:equivalent}. A scaling argument proves that the constant $C_{\rm hh2}$ depends only on $d$, $\boldsymbol{A}$,  $p$, and $\gamma$-shape regularity of $\mathcal{T}_\bullet$, while $C_{\rm hh2}\ge1$ is obvious.
\end{proof}

With the convention~\eqref{eq:convention}, we define, for all $T\in\mathcal{T}_\bullet$ and all  $\widehat v_\bullet \in {\mathcal{S}}_0^p(\widehat{\mathcal{T}}_\bullet)$,
\begin{subequations}\label{eq:osc}
\begin{align*}
 &{\rm res}_\bullet(T,\widehat v_\bullet)^2 
 := h_T^2 \, \sum_{\substack{T' \in \widehat{\mathcal{T}}_\bullet \\ T' \subseteq T}} \norm{f+{\rm div}(\boldsymbol{A}\nabla \widehat v_\bullet)}{T'}^2,\\ 
&{\rm osc}_\bullet(T)^2 
 := h_T^2 \, \norm{(1-\pi_\bullet)f}{T}^2
 \quad\text{and}\quad
 {\rm apx}_\bullet(T)^2:= h_T^2 \, \norm{(1-\Pi_\bullet)f}{T}^2 ,
\end{align*}
\end{subequations}
where ${\rm apx}_\bullet(T)$ is only required for $p\ge2$.
Further, we abbreviate ${\rm osc}_\bullet^2:=\sum_{T\in\mathcal{T}_\bullet}{\rm osc}_\bullet(T)^2$  as well as ${\rm apx}_\bullet^2:=\sum_{T\in\mathcal{T}_\bullet}{\rm apx}_\bullet(T)^2$.
Note that ${\rm res}_\bullet(T,v_\bullet)^2 = h_T^2 \, \norm{f+ {\rm div}(\boldsymbol{A}\nabla  v_\bullet)}{T}^2$ for all $v_\bullet \in {\mathcal{S}}^p_0(\mathcal{T}_\bullet)$.
Then, we consider the following {\sl a~posteriori} error estimators
\begin{align}\label{table}
\begin{tabular}{|rcl|c|c|c|}
\hline
\multicolumn{3}{|c|}{$\eta_\bullet(T,\widehat v_\bullet)^2 :=$} & \multicolumn{2}{|c|}{requirements} & refinement 
\\
\hline
\hline
$\lambda_\bullet(T,\widehat v_\bullet)^2$ & $+$ & ${\rm res}_\bullet(T,\widehat v_\bullet)^2$ & $p\ge1$ & $d\ge2$ & (M3) \\
$\lambda_\bullet(T,\widehat v_\bullet)^2$ & $+$ & ${\rm osc}_\bullet(T)^2$ & $p\ge1$ & $d\ge2$ & (M3') \\
$\lambda_\bullet(T,\widehat v_\bullet)^2$ & $+$ & ${\rm apx}_\bullet(T)^2$ & $p\ge2$ & $d\in\{2,3\}$ & (M3) \\
\hline
\hline
$\mu_\bullet(T,\widehat v_\bullet)^2$ & $+$ & ${\rm res}_\bullet(T,\widehat v_\bullet)^2$ & $p\ge1$ & $d\ge2$ & (M3) \\
$\mu_\bullet(T,\widehat v_\bullet)^2$ & $+$ & ${\rm osc}_\bullet(T)^2$ & $p\ge1$ & $d\ge2$ & (M3') \\
$\mu_\bullet(T,\widehat v_\bullet)^2$ & $+$ & ${\rm apx}_\bullet(T)^2$ & $p\ge2$ & $d\in\{2,3\}$ & (M3) \\
\hline
\end{tabular}
\end{align}

\begin{remark}
We note that, for $p\ge 2$, in contrast to the oscillation terms ${\rm osc}_\bullet$, the data approximation terms  $ {\rm apx}_\bullet$ are in general not of higher, but of the same order as the discretization error $\norm{\boldsymbol{A}^{1/2}\nabla(u - u_\bullet)}{\Omega}$. 
%
\end{remark}

\subsection{Adaptive algorithm}
We analyze the following adaptive strategy which is driven by one of the error estimators $\eta_\bullet$ from~\eqref{table}.

\begin{algorithm}\label{algorithm}
{\bfseries Input:} Conforming triangulation $\mathcal{T}_0$ of $\Omega$, adaptivity parameter $0<\theta\le1$.\\
{\bfseries Loop:} For all $\ell = 0,1,2,\dots$, iterate the following steps~{\rm(i)--(iv)}:
\begin{itemize}
\item[\rm(i)] Compute the discrete solution $\widehat u_\ell \in {\mathcal{S}}^p_0(\widehat{\mathcal{T}}_\ell)$,
where $\widehat{\mathcal{T}}_\ell := {\rm refine}(\mathcal{T}_\ell,\mathcal{T}_\ell)$.
\item[\rm(ii)] Compute the indicators $\eta_\ell(T,\widehat u_\ell)$ for all $T \in \mathcal{T}_\ell$.
\item[\rm(iii)] Determine some $\mathcal{M}_\ell \subseteq \mathcal{T}_\ell$ with minimal cardinality such that 
$\theta \, \eta_\ell(\widehat u_\ell)^2 \le \eta_\ell(\mathcal{M}_\ell, \widehat u_\ell)^2$.
\item[\rm(iv)] Generate $\mathcal{T}_{\ell+1} := {\rm refine}(\mathcal{T}_\ell,\mathcal{M}_\ell)$.
\end{itemize}
{\bfseries Output:} Sequences of successively refined triangulations $\mathcal{T}_\ell$, discrete solutions $\widehat u_\ell$, and corresponding error estimators $\eta_\ell(\widehat u_\ell)$, for all $\ell \ge 0$.
\end{algorithm}%

\subsection{Main result}
Given the initial triangulation $\mathcal{T}_0$, we define the following two approximation classes for $s>0$:
With the error estimator $\eta_\bullet$ from~\eqref{table} used for Algorithm~\ref{algorithm} and the convention~\eqref{eq:convention}, let
\begin{align}\label{eq:As}
\norm{u}{\mathbb{A}_s^{\eta}}
 := \sup_{N\in\mathbb{N}_0}\min_{\substack{\mathcal{T}_\bullet\in{\rm nvb}(\mathcal{T}_0)\\\#\mathcal{T}_\bullet - \#\mathcal{T}_0\le N}}(N+1)^s\,\eta_\bullet(\widehat u_\bullet) 
 \in [0,\infty].
\end{align}
Moreover, let
\begin{align}\label{eq:As tot}
 \norm{u}{\mathbb{A}_s^{\rm tot}}
 := \sup_{N\in\mathbb{N}_0} \min_{\substack{\mathcal{T}_\bullet\in{\rm nvb}(\mathcal{T}_0)\\\#\mathcal{T}_\bullet - \#\mathcal{T}_0\le N}} (N+1)^s\, \big(\min_{v_\bullet \in {\mathcal{S}}^p_0(\mathcal{T}_\bullet)} \norm{\boldsymbol{A}^{1/2}\nabla (u - v_\bullet)}{\Omega}+{{\rm osc}_\bullet} \big)
  \in [0,\infty].
\end{align}
Note that the definition of $\norm{u}{\mathbb{A}_s^{\rm tot}}$ is independent of the error estimator $\eta_\bullet$.

By definition, $\norm{u}{\mathbb{A}^{\rho}_s}<\infty$ and $\norm{u}{\mathbb{A}^{\rm tot}_s}<\infty$ imply that the quantity $\eta_\bullet(\widehat u_\bullet)$ and the \textit{total error} on the optimal meshes $\mathcal{T}_\bullet$ decay at least with rate ${\mathcal{O}}\big((\#\mathcal{T}_\bullet)^{-s}\big)$.  The following main theorem states that each possible rate $s>0$ is in fact realized by Algorithm~\ref{algorithm}.
The proof requires some technical preparations and is thus postponed to Section~\ref{section:proof}.

\begin{theorem}\label{thm:abstract}
Let $\eta_\bullet$ be one of the error estimators from~\eqref{table}. 
Then, the error estimator $\eta_\bullet(\widehat u_\bullet)$ is  reliable and efficient, i.e., there exist constants $C_{\rm eff},C_{\rm rel}>0$ such that  
\begin{align}\label{eq:reliable}
 C_{\rm eff}^{-1} \, \eta_\bullet(\widehat u_\bullet)
 \le \! \Big( \! \min_{v_\bullet \in {\mathcal{S}}^p_0(\mathcal{T}_\bullet)} \! \norm{\boldsymbol{A}^{1/2}\nabla (u - v_\bullet)}{\Omega}^2 + {\rm osc}_\bullet^2 \Big)^{1/2} \!\!
 \le C_{\rm rel} \, \eta_\bullet(\widehat u_\bullet)\text{ for }\mathcal{T}_\bullet\in{\rm nvb}(\mathcal{T}_0).
\end{align}
In particular, this implies that 
\begin{align}\label{eq:approximationclass}
 C_{\rm rel}^{-1} \, \norm{u}{\mathbb{A}_s^{\rm tot}}
 \le \norm{u}{\mathbb{A}_s^{\eta}}
 \le C_{\rm eff} \, \norm{u}{\mathbb{A}_s^{\rm tot}}.
\end{align}%
For arbitrary $0<\theta\le1$, the error estimator sequence generated by Algorithm~\ref{algorithm} converges linearly, i.e., there exist constants $0<q_{\rm lin}<1$ and $C_{\rm lin}\ge1$ such that
\begin{align}\label{eq:linear}
\eta_{\ell+n}(\widehat u_{\ell+j})\le C_{\rm lin}q_{\rm lin}^j \eta_\ell(\widehat u_\ell)\quad\text{for all }\ell,n\in\mathbb{N}_0.
\end{align}
Moreover, there exists a constant $0<\theta_{\rm opt}<1$ such that for all $0<\theta<\theta_{\rm opt}$, the estimator  $\eta_\ell(\widehat u_\ell)$ converges at optimal algebraic  rate, i.e., for all $s>0$ there exist constants $c_{\rm opt},C_{\rm opt}>0$ such that
\begin{align}\label{eq:optimal}
 c_{\rm opt}\norm{u}{\mathbb{A}_s^{\eta}}
 \le \sup_{\ell\in\mathbb{N}_0}{(\# \mathcal{T}_\ell-\#\mathcal{T}_0+1)^{s}}\,{\eta_\ell(\widehat u_\ell)}
 \le C_{\rm opt}\norm{u}{\mathbb{A}_s^{\eta}}.
\end{align}
\noindent All involved constants $C_{\rm rel}$, $C_{\rm eff}$, $C_{\rm lin}$, $q_{\rm lin}$, and $\theta_{\rm opt}$ depend only on $d$, $\boldsymbol{A}$, $p$, $C_{\rm son}$, and $\gamma$-shape regularity of $\mathcal{T}_0$, whereas $C_{\rm lin}$ and $q_{\rm lin}$ depend additionally on $\theta$, and $C_{\rm opt}$ depends furthermore on $s$.
The constant $c_{\rm opt}$ depends only on $C_{\rm son}$, $\#\mathcal{T}_0$, and  $s$.
\end{theorem}

\begin{remark}
 \label{rem:estimator}
{\rm(i)} 
Recall that $\lambda_\bullet(\widehat u_\bullet)\le \norm{\boldsymbol{A}^{1/2} \nabla(u-u_\bullet)}{\Omega}$ according to \eqref{eq:hh2:efficient}.
For $\eta_\bullet^2 = \lambda_\bullet(\widehat u_\bullet)^2 + {\rm osc}_\bullet^2$, this yields that $C_{\rm eff} = 1$ in~\eqref{eq:reliable}, i.e., the estimator $\eta_\bullet$ is a guaranteed lower bound for the unknown total error with constant $1$.
\\
{\rm(ii)}
In general, one expects an optimal convergence rate of $s=p/d$ for the error.
Asymptotically, this leads to $\norm{\boldsymbol{A}^{1/2} \nabla(u-u_\ell)}{\Omega}= C (\#\mathcal{T}_\ell)^{-p/d}$ for some constant $C>0$.
If uniform refinement  ${\rm refine}(\mathcal{T}_\ell,\mathcal{T}_\ell)$ bisects all elements into exactly $C_{\rm son}$ elements, this suggests that 
$\norm{\boldsymbol{A}^{1/2} \nabla(u-\widehat u_\ell)}{\Omega}= C (C_{\rm son}\#\mathcal{T}_\ell)^{-p/d}$.
In particular, one obtains that $q_{\rm sat}=C_{\rm son}^{-p/d}$ in~\eqref{intro:saturation_assumption}. 
Together with \eqref{intro:efficiency+reliability} and \eqref{eq:hh2:efficient}, this yields the asymptotical upper bound 
\begin{align*}
\norm{\boldsymbol{A}^{1/2} \nabla (u - u_\ell)}{\Omega}
 \le (1-C_{\rm son}^{-2p/d})^{-1/2}\mu_\ell.
 \end{align*}
For $\eta_\bullet^2 = \mu_\bullet(\widehat u_\bullet)^2 + {\rm osc}_\bullet^2$, the estimator $\eta_\bullet$ is hence an upper bound for the unknown total error in~\eqref{eq:reliable} with known asymptotical reliability constant $C_{\rm rel} = (1-C_{\rm son}^{-2p/d})^{-1/2}$.
\\
{\rm(iii)} Note that the approximation norm $\norm{u}{\mathbb{A}_s^{\rm tot}}$ 
is the same as for residual error estimators; cf.~\cite{ckns,ks,ffp14,axioms}. In explicit terms, the $(h-h/2)$-type estimators from~\eqref{table} thus lead to the same  algebraic convergence rates as the residual error estimators. 
\\
{\rm (iv)} Alternatively, one could define the approximation classes $\norm{u}{\mathbb{A}_s^{\eta}}$ and $\norm{u}{\mathbb{A}_s^{\rm tot}}$ with ${\rm refine}(\cdot)$ instead of ${\rm nvb}(\cdot)$, i.e.,
\begin{align*}
\norm{u}{\widetilde{\mathbb{A}}_s^{\eta}}
 &:= \sup_{N\in\mathbb{N}_0}\min_{\substack{\mathcal{T}_\bullet\in{\rm refine}(\mathcal{T}_0)\\\#\mathcal{T}_\bullet - \#\mathcal{T}_0\le N}}(N+1)^s\,\eta_\bullet(\widehat u_\bullet),
\\
 \norm{u}{\widetilde{\mathbb{A}}_s^{\rm tot}}
 &:= \sup_{N\in\mathbb{N}_0} \min_{\substack{\mathcal{T}_\bullet\in{\rm refine}(\mathcal{T}_0)\\\#\mathcal{T}_\bullet - \#\mathcal{T}_0\le N}} (N+1)^s\, \big(\min_{v_\bullet \in {\mathcal{S}}^p_0(\mathcal{T}_\bullet)} \norm{\boldsymbol{A}^{1/2}\nabla (u - v_\bullet)}{\Omega}+{{\rm osc}_\bullet} \big).
\end{align*}
Clearly, \eqref{eq:reliable} gives that $\norm{u}{\widetilde{\mathbb{A}}_s^{\eta}}\simeq\norm{u}{\widetilde{\mathbb{A}}_s^{\rm tot}}$. 
Moreover, 
$\norm{u}{{\mathbb{A}}_s^{\eta}}\le \norm{u}{\widetilde{\mathbb{A}}_s^{\eta}}$ follows from ${\rm refine}(\mathcal{T}_0)\subseteq{\rm nvb}(\mathcal{T}_0)$.
Arguing as in  \cite[Prop.~4.15]{axioms}, we prove that
\begin{align}\label{eq:lower optimal estimate}
\widetilde c_{\rm opt}\norm{u}{\widetilde{\mathbb{A}}_s^{\eta}}
 \le \sup_{\ell\in\mathbb{N}_0}{(\# \mathcal{T}_\ell-\#\mathcal{T}_0+1)^{s}}\,{\eta_\ell(\widehat u_\ell)},
 \end{align}
where $\widetilde c_{\rm opt}$ depends only on $C_{\rm son}$, $\#\mathcal{T}_0$, and  $s$.
Together with \eqref{eq:optimal}, we conclude 
\begin{align}\label{eq:refine or nvb}
\frac{\widetilde c_{\rm opt}}{C_{\rm opt}}\norm{u}{\widetilde{\mathbb{A}}_s^{\eta}} \le \norm{u}{{\mathbb{A}}_s^{\eta}}\le \norm{u}{\widetilde{\mathbb{A}}_s^{\eta}}
\end{align} 
\end{remark}%

\section{Residual error estimator}
\label{section:residual}
As an auxiliary tool, we consider the residual error estimator. Because of its later application, we use the notation $\mathcal{T}_\blacktriangle$ and $\mathcal{T}_\vartriangle \in {\rm nvb}(\mathcal{T}_\blacktriangle)$ for a given triangulation $\mathcal{T}_\blacktriangle\in{\rm nvb}(\mathcal{T}_0)$ and a corresponding refinement.
Recall the definition of
${\rm res}_\blacktriangle(\cdot)$ and
${\rm osc}_\blacktriangle(\cdot)$ from~\eqref{eq:osc}.
We define, for all $v_\blacktriangle \in {\mathcal{S}}^p_0(\mathcal{T}_\blacktriangle)$,
\begin{align}\label{eq:eta}
 \varrho_\blacktriangle(\tau,v_\blacktriangle)^2 
 := 
 \begin{cases}
  {\rm res}_\blacktriangle(T,v_\blacktriangle)^2
 \quad&
 \text{for } \tau = T \in \mathcal{T}_\blacktriangle,
 \\ 
 h_E \, \norm{[\boldsymbol{A}\nabla v_\blacktriangle\cdot n]}{E}^2,
 \quad&
 \text{for } \tau = E \in \mathcal{E}_\blacktriangle^\Omega.
 \end{cases}
\end{align}
Generalizing the convention~\eqref{eq:convention}, we define, for all $v_\blacktriangle \in {\mathcal{S}}^p_0(\mathcal{T}_\blacktriangle)$,  
\begin{align*}
 \varrho_\blacktriangle(v_\blacktriangle) := \varrho_\blacktriangle(\mathcal{T}_\blacktriangle\cup\mathcal{E}_\blacktriangle^\Omega,v_\blacktriangle),
 \text{ where }
 \varrho_\blacktriangle({\mathcal{U}}_\blacktriangle,v_\blacktriangle)^2 := \sum_{\tau \in {\mathcal{U}}_\blacktriangle} \varrho_\blacktriangle(\tau,v_\blacktriangle)^2
 \text{ for all } {\mathcal{U}}_\blacktriangle \subseteq \mathcal{T}_\blacktriangle\cup\mathcal{E}_\blacktriangle^\Omega.
\end{align*}%
It is well-known~\cite{aoAposteriori,verfuerth} that $\varrho_\blacktriangle(\cdot)$ is reliable and efficient in the sense that
\begin{align}\label{eq:rho:releff}
 C_{\rm rel}^{-1} \,  \norm{\boldsymbol{A}^{1/2}\nabla(u - u_\blacktriangle)}{\Omega} + {\rm osc}_\blacktriangle
 \le \varrho_\blacktriangle(u_\blacktriangle)
 \le C_{\rm eff} \, \big( \norm{\boldsymbol{A}^{1/2}\nabla(u - u_\blacktriangle)}{\Omega} + {\rm osc}_\blacktriangle \big),
\end{align}
where  $C_{\rm rel},C_{\rm eff}>0$ depend only on $d$, $\boldsymbol{A}$, $p$, and $\gamma$-shape regularity of $\mathcal{T}_\blacktriangle$.

The next lemma recalls the discrete reliability estimate which originally goes back to~\cite{stevenson07}. While the proof of~\cite{stevenson07} relied on the refined elements $\mathcal{T}_\blacktriangle \backslash \mathcal{T}_\vartriangle$ plus one additional layer of elements for the localized upper bound, the proof of~\cite{ckns} involves only the refined elements $\mathcal{T}_\blacktriangle \backslash \mathcal{T}_\vartriangle$. Even though~\cite{stevenson07,ckns} consider an element-based formulation of the residual error estimator, their ideas of the proof also yield the following slightly stronger estimate for our variant of $\varrho_\blacktriangle(\cdot)$ (which is indexed by elements and interior facets).
While~\cite[Lemma~3.6]{ckns} would also involve non-refined facets of refined elements on the right-hand side of~\eqref{eq:rho:drel}, we only require refined facets. 

\begin{lemma}[discrete reliability of residual error estimator]
It holds that
\begin{align}\label{eq:rho:drel}
 \norm{\boldsymbol{A}^{1/2}\nabla(u_\vartriangle - u_\blacktriangle)}{\Omega}
 \le C_{\rm drl} \, \varrho_\blacktriangle\big((\mathcal{T}_\blacktriangle \backslash \mathcal{T}_\vartriangle) \cup (\mathcal{E}_\blacktriangle^\Omega \backslash \mathcal{E}_\vartriangle^\Omega) , u_\blacktriangle\big).
\end{align}
The constant $C_{\rm drl} > 0$ depends only on $d$, $\boldsymbol{A}$, and $\gamma$-shape regularity of $\mathcal{T}_\blacktriangle$.
\end{lemma}

\begin{proof}[Sketch of proof]
Recall the (discrete) Galerkin orthogonality
\begin{align*}
 \int_\Omega \boldsymbol{A}\nabla (u_\vartriangle - u_\blacktriangle) \cdot \nabla v_\blacktriangle \, dx = 0
 \quad \text{for all } v_\blacktriangle \in {\mathcal{S}}^p_0(\mathcal{T}_\blacktriangle).
\end{align*}
For arbitrary $v_\blacktriangle \in {\mathcal{S}}^p_0(\mathcal{T}_\blacktriangle)$, define $w_\vartriangle := (u_\vartriangle - u_\blacktriangle)-v_\blacktriangle$.
The discrete formulation~\eqref{eq:discrete} for $u_\vartriangle$ proves that
\begin{align*}
 &\norm{\boldsymbol{A}^{1/2}\nabla (u_\vartriangle - u_\blacktriangle)}{\Omega}^2
 = \int_\Omega \boldsymbol{A}\nabla (u_\vartriangle - u_\blacktriangle) \cdot \nabla w_\vartriangle \,dx
 \reff{eq:discrete}= \int_\Omega f w_\vartriangleÃÂ \, dx
 - \sum_{T\in\mathcal{T}_\blacktriangle}\int_T \boldsymbol{A}\nabla u_\blacktriangle \cdot \nabla w_\vartriangle \,dx.
\end{align*}
For $T\in\mathcal{T}_\blacktriangle$, integration by parts and $w_\vartriangle \in {\mathcal{S}}^p_0(\mathcal{T}_\vartriangle)$ yield that
\begin{align*}
 -\int_T \boldsymbol{A}\nabla u_\blacktriangle \cdot \nabla w_\vartriangle \,dx
 = \int_T {\rm div}(\boldsymbol{A}\nabla u_\blacktriangle) \, w_\vartriangle \, dx
 - \int_{\partial T\backslash\Gamma} \boldsymbol{A}\nabla u_\blacktriangle\cdot n\,w_\vartriangle \, ds.
\end{align*}
Combining these identities, we see that
\begin{align}\label{eq1:rho:drel}
&\norm{\boldsymbol{A}^{1/2}\nabla (u_\vartriangle - u_\blacktriangle)}{\Omega}^2
= \sum_{T \in \mathcal{T}_\blacktriangle} \int_T (f+{\rm div} (\boldsymbol{A}\nabla u_\blacktriangle)) w_\vartriangleÃÂ \, dx
- \sum_{E \in \mathcal{E}_\blacktriangle^\Omega} \int_E[ \boldsymbol{A}\nabla u_\blacktriangle\cdot n] \, w_\vartriangle \,ds.
\end{align}
To proceed, we will choose $v_\blacktriangle = \mathcal{J}_\blacktriangle (u_\vartriangle - u_\blacktriangle)$, where $\mathcal{J}_\blacktriangle : H^1(\Omega) \to {\mathcal{S}}^p(\mathcal{T}_\blacktriangle)$ is a Scott-Zhang projector~\cite{sz}. For the convenience of the reader, we recall the construction of $\mathcal{J}_\blacktriangle$: Let $\mathcal{L}_\blacktriangle \subseteq \overline\Omega$ be the set of  Lagrange nodes of ${\mathcal{S}}^p(\mathcal{T}_\blacktriangle)$. Let $\set{\phi_{\blacktriangle,z}\in{\mathcal{S}}^p(\mathcal{T}_\blacktriangle)}{z\in\mathcal{L}_\blacktriangle}$ be the corresponding nodal basis of ${\mathcal{S}}^p(\mathcal{T}_\blacktriangle)$, i.e., with Kroneckers's delta, it holds that $\phi_{\blacktriangle,z}(z') = \delta_{zz'}$ for all $z,z' \in \mathcal{L}_\blacktriangle$. If $z\in\mathcal{L}_\blacktriangle$ is on the skeleton $\bigcup_{E \in \mathcal{E}_\blacktriangle} E$, choose a facet $\tau_z := E \in \mathcal{E}_\blacktriangle$ with $z \in \tau_z$ subject to the following constraints (which further specify the constraints from~\cite{sz}):
\begin{itemize}
\item If $z \in \Gamma$, then choose $\tau_z = E \subset \Gamma$.
\item If $z \in E \in \mathcal{E}_\blacktriangle^\Omega \cap \mathcal{E}_\vartriangle^\Omega$, then choose $\tau_z = E$ (which is not necessarily unique).
\item Otherwise, choose an arbitrary $\tau_z = E \in \mathcal{E}_\blacktriangle^\Omega \backslash \mathcal{E}_\vartriangle^\Omega$ with $z \in E$.
\end{itemize}
If $z$ is not on the skeleton, then there exists a unique element $T\in\mathcal{T}_\blacktriangle$ such that $z$ lies in the interior of $\tau_z := T$.
Consider the nodal basis $\{\phi_{\blacktriangle,z'}\}$ restricted to $\mathcal{P}^p(\tau_z)$ and let $\{\psi_{\blacktriangle,z'}\} \subset \mathcal{P}^p(\tau_z)$ be the corresponding dual basis, i.e., 
$\int_{\tau_z}\phi_{\blacktriangle,z'}\psi_{\blacktriangle,z''}\,dx = \delta_{z'z''}$ for all $z',z'' \in \tau_z \cap \mathcal{L}_\blacktriangle$. Then, the Scott-Zhang projector is defined by
\begin{align*}
 \mathcal{J}_\blacktriangle v := \sum_{z \in \mathcal{L}_\blacktriangle} \bigg(\int_{\tau_z} v \psi_{\blacktriangle,z} \, dx\bigg)\,\phi_{\blacktriangle,z}.
\end{align*}
According to~\cite{sz}, $\mathcal{J}_\blacktriangle$ has the following properties for all $w \in H^1(\Omega)$, all $w_\blacktriangle \in {\mathcal{S}}^p(\mathcal{T}_\blacktriangle)$, and all $T\in\mathcal{T}_\blacktriangle$, where $\omega_\blacktriangle(T):=\bigcup\set{T'\in\mathcal{T}_\blacktriangle}{T\cap T'\neq\emptyset}$ denotes the element patch:
\begin{itemize}
\item {\bf projection property:} $w = w_\blacktriangle$ on $\omega_\blacktriangle(T)$ implies that $\mathcal{J}_\blacktriangle w = w_\blacktriangle$ on $T$;
\item {\bf preservation of discrete traces:} $w = w_\blacktriangle$ on $\Gamma$ implies that $\mathcal{J}_\blacktriangle w = w_\blacktriangle$ on $\Gamma$;
\item {\bf $\boldsymbol{L^2}$ approximation property:} $\norm{w - \mathcal{J}_\blacktriangle w}T \lesssim h_T \, \norm{\nabla w}{\omega_\blacktriangle(T)}$;
\item {\bf $\boldsymbol{H^1}$ stability:} $\norm{\nabla (w - \mathcal{J}_\blacktriangle w)}T \lesssim \norm{\nabla w}{\omega_\blacktriangle(T)}$.
\end{itemize}
In addition, our choice of $\tau_z$ yields further structure: Let $v_\vartriangle \in {\mathcal{S}}^p_0(\mathcal{T}_\vartriangle)$ and $z \in \mathcal{L}_\blacktriangle$.
\begin{itemize}
\item
If $z \in \Gamma$, it holds that $v_\vartriangle |_{\tau_z} = 0$ and hence $(\mathcal{J}_\blacktriangle v_\vartriangle)(z) = 0 = v_\vartriangle(z)$.

\item
If $\tau_z = E \in \mathcal{E}_\blacktriangle^\Omega \cap \mathcal{E}_\vartriangle^\Omega$, then $v_\vartriangle |_{\tau_z} \in \mathcal{P}^p(\tau_z)$ and hence $(\mathcal{J}_\blacktriangle v_\vartriangle)(z) = v_\vartriangle(z)$ by choice of the dual basis.

\item 
Let $E \in \mathcal{E}_\blacktriangle^\Omega \cap \mathcal{E}_\vartriangle^\Omega$. Suppose that $z \in E \neq \tau_z$. Then, 
$(\mathcal{J}_\blacktriangle v_\vartriangle)(z) = v_\vartriangle(z)$ follows from the previous steps.

\item
If $z \in \tau_z = T \in \mathcal{T}_\blacktriangle \capÃÂ \mathcal{T}_\vartriangle$ is in the interior of $T$, then $v_\vartriangle |_{\tau_z} \in \mathcal{P}^p(\tau_z)$ and hence $(\mathcal{J}_\blacktriangle v_\vartriangle)(z) = v_\vartriangle(z)$ by choice of the dual basis.
\end{itemize}
Overall, this proves that $v_\vartriangle - \mathcal{J}_\blacktriangle v_\vartriangle = 0$ on all $T \in \mathcal{T}_\blacktriangle \capÃÂ \mathcal{T}_\vartriangle$ as well as on all $E \in \mathcal{E}_\blacktriangle^\Omega \cap \mathcal{E}_\vartriangle^\Omega$. For $v_\vartriangle := u_\vartriangle - u_\blacktriangle$ and $w_\vartriangle := v_\vartriangle - \mathcal{J}_\blacktriangle v_\vartriangle$, we plug this into~\eqref{eq1:rho:drel} and observe that
\begin{align*}
&\norm{\boldsymbol{A}^{1/2}\nabla (u_\vartriangle - u_\blacktriangle)}{\Omega}^2
= \sum_{T \in \mathcal{T}_\blacktriangle \backslash \mathcal{T}_\vartriangle} 
\int_T(f+{\rm div}(\boldsymbol{A} \nabla u_\blacktriangle)) w_\vartriangleÃÂ \, dx
- \sum_{E \in \mathcal{E}_\blacktriangle^\Omega \backslash \mathcal{E}_\vartriangle^\Omega} \int_E [\boldsymbol{A}\nabla u_\blacktriangle\cdot n] \, w_\vartriangle \,ds.
\end{align*}
With the usual arguments (see, e.g.,~\cite{aoAposteriori,verfuerth}), this leads to~\eqref{eq:rho:drel}.
\end{proof}

Next, we recall that the error estimator $\varrho_\blacktriangle(\cdot)$ depends (locally) Lipschitz continuously on the discrete functions. The following result is obtained analogously to~\cite[Prop.~3.3]{ckns}, where the proof relies only on the trace inequality plus inverse estimates. 

\begin{lemma}[local stability of residual error estimator]\label{lem:local stability rho}
Let $v_\blacktriangle, w_\blacktriangle \in {\mathcal{S}}^p_0(\mathcal{T}_\blacktriangle)$. Let $T,T'\in\mathcal{T}_\blacktriangle$ and $E := T \cap T' \in \mathcal{E}_\blacktriangle^\Omega$. Then, it holds that
\begin{subequations}\label{eq:rho:stab}
\begin{align}\label{eq:rho:stab elements}
 \varrho_\blacktriangle(T,v_\blacktriangle) &\le \varrho_\blacktriangle(T,w_\blacktriangle) + C_{\rm stb}\,\norm{\boldsymbol{A}^{1/2}\nabla (v_\blacktriangle - w_\blacktriangle)}{T},
 \\
 \varrho_\blacktriangle(E,v_\blacktriangle) &\le \varrho_\blacktriangle(E,w_\blacktriangle) + C_{\rm stb}\,\norm{\boldsymbol{A}^{1/2}\nabla (v_\blacktriangle - w_\blacktriangle)}{T\cup T'}.
\end{align}
\end{subequations}
The constant $C_{\rm stb} > 0$ depends only on $d$, $\boldsymbol{A}$, $p$, and $\gamma$-shape regularity of $\mathcal{T}_\blacktriangle$.\qed
\end{lemma}

\begin{remark}\label{rem:local stability rho}
We note that \eqref{eq:rho:stab elements} is also satisfied if $v_\blacktriangle, w_\blacktriangle\in {\mathcal{S}}^p_0(\widehat{\mathcal{T}}_\blacktriangle)$. 
In this case the constant $C_{\rm stb}>0$ depends additionally on $C_{\rm son}$.
\end{remark}

The following lemma is proved along the lines of~\cite[Prop.~2]{fop} and adapts the classical efficiency proof by using cleverly chosen bubble functions. We note that the idea goes back to the seminal works~\cite{doerfler,mns}. 

\begin{lemma}[local discrete efficiency of residual error estimator]
\label{lemma:rho:efficiency}
Let $T, T',T'' \in \mathcal{T}_\blacktriangle \backslash \mathcal{T}_\vartriangle$ and $E = T'\cap T'' \in \mathcal{E}_\blacktriangle^\Omega \backslash \mathcal{E}_\vartriangle^\Omega$. Let $p\ge1$. 
If $E$ contains an interior node $z \in {\mathcal{N}}_\vartriangle$, 
then it holds that
\begin{subequations}\label{eq:rho:eff}
\begin{align}
\label{eq2:rho:eff}
 C_{\rm eff}^{-1}\,\varrho_\blacktriangle(E,v_\blacktriangle)
 &\le \norm{\boldsymbol{A}^{1/2}\nabla(u_\vartriangle - v_\blacktriangle)}{T'\cup T''}
 + \varrho_\blacktriangle(\{T', T''\},v_\blacktriangle).
 \intertext{If one of the following cases is satisfied
\begin{itemize}
\item $d=2$ and 
 $p\ge2$,
 \item $d=3$, $p\ge2$, and each facet of $T$ contains an interior node $z\in{\mathcal{N}}_\vartriangle$,
 \item $T$ contains an interior node $z \in {\mathcal{N}}_\vartriangle$, 
 \end{itemize}then it holds that}
 \label{eq1:rho:eff}
 C_{\rm eff}^{-1}\,\varrho_\blacktriangle(T,v_\blacktriangle)
 &\le \norm{\boldsymbol{A}^{1/2}\nabla(u_\vartriangle - v_\blacktriangle)}{T}
 + \begin{cases}{\rm apx}_\blacktriangle(T) \text{ in the first two cases,}
 \\ {\rm osc}_\blacktriangle(T) \,\text{ in the third case}. \end{cases}
\end{align}
\end{subequations}
The constant $C_{\rm eff}>0$ depends only on $d$, $\boldsymbol{A}$, $p$, and $\gamma$-shape regularity of $\mathcal{T}_\blacktriangle$ and $\mathcal{T}_\vartriangle$.
\end{lemma}

\begin{proof}[Proof of~\eqref{eq2:rho:eff}]
Since NVB is a binary refinement rule, there exists a coarsest refinement $\mathcal{T}_\star \in {\rm nvb}(\mathcal{T}_\blacktriangle)$ such that $E$ contains an interior node $z \in {\mathcal{N}}_\star \backslash {\mathcal{N}}_\blacktriangle$.
Choose the corresponding hat function $\phi_{\star,z} \in {\mathcal{S}}^1_0(\mathcal{T}_\vartriangle)$ as discrete facet bubble function
\begin{align}\label{eq:betaE}
 \beta_E := \phi_{\star,z} \in {\mathcal{S}}^1_0(\mathcal{T}_\star) \subseteq {\mathcal{S}}^1_0(\mathcal{T}_\vartriangle).
\end{align}
In particular, $\beta_E \in H^1_0(T'\cup T'')$ and $|{\rm supp}(\beta_E)| \simeq |T' \cup T''|$.
Since $u_\blacktriangle\in\mathcal{P}^p(\mathcal{T}_\blacktriangle)$, a scaling argument shows the existence of some $r_\blacktriangle \in \mathcal{P}^{p-1}(T'\cup T'')$ such that 
\begin{align*}
 r_\blacktriangle|_E = [\boldsymbol{A}\nabla u_\blacktriangle \cdot n]|_E
 \quad \text{and} \quad
 \norm{r_\blacktriangle}{T'\cup T''} \lesssim h_E^{1/2} \, \norm{[\boldsymbol{A}\nabla u_\blacktriangle \cdot n]}{E}.
\end{align*}
Choose $v := r_\blacktriangle \beta_E \in \mathcal{P}^p(\mathcal{T}_\vartriangle)$ and note that $v\in H^1_0(T'\cup T'')$.
Let ${\rm div}_\blacktriangle$ denote the $\mathcal{T}_\blacktriangle$-piecewise divergence operator.  A scaling argument and integration by parts prove that
\begin{align*}
 \norm{[\boldsymbol{A}\nabla u_\blacktriangle \cdot n]}{E}^2
& \lesssim \norm{[\boldsymbol{A}\nabla u_\blacktriangle \cdot n]\,\beta_E^{1/2}}{E}^2
 = \int_E [\boldsymbol{A}\nabla u_\blacktriangle \cdot n] \, v \, ds
 \\&= \int_{\partial T'} \boldsymbol{A}\nabla u_\blacktriangle \cdot n \, v \, ds+ \int_{\partial T''} \boldsymbol{A}\nabla u_\blacktriangle \cdot n \, v \, ds 
\\ &= \int_{T' \cup T''} \boldsymbol{A}\nabla u_\blacktriangle \cdot \nabla v \, dx
 + \int_{T' \cup T''} {\rm div}_\blacktriangle(\boldsymbol{A}\nabla u_\blacktriangle) \, v \, dx.
\end{align*}
Since $v \in \mathcal{P}^p(\mathcal{T}_\vartriangle) \cap H^1_0(T'\cup T'') \subset {\mathcal{S}}^p_0(\mathcal{T}_\vartriangle)$, the discrete formulation~\eqref{eq:discrete} yields that
\begin{align}\label{eq:rho:eff:problem2}
\begin{split}
 &\norm{[\boldsymbol{A}\nabla u_\blacktriangle \cdot n]}{E}^2
 \lesssim \int_{T' \cup T''} \boldsymbol{A}\nabla u_\blacktriangle \cdot \nabla v \, dx
 + \int_{T' \cup T''} {\rm div}_\blacktriangle(\boldsymbol{A}\nabla u_\blacktriangle ) \, v \, dx
 \\&\quad
 \reff{eq:discrete}= -\int_{T'\cup T''} \boldsymbol{A}\nabla (u_\vartriangle - u_\blacktriangle) \cdot \nabla v \, dx
 + \int_{T' \cup T''} \big( f + {\rm div}_\blacktriangle(\boldsymbol{A}\nabla u_\blacktriangle ) \big) \, v \, dx
 \\&\quad
 \le \norm{\boldsymbol{A}^{1/2}\nabla (u_\vartriangle - u_\blacktriangle)}{T' \cup T''} \, \norm{\boldsymbol{A}^{1/2}\nabla v}{T' \cup T''}
 + \norm{f + {\rm div}_\blacktriangle(\boldsymbol{A}\nabla u_\blacktriangle )}{T' \cup T''} \, \norm{v}{T' \cup T''}.
\end{split}
\end{align}
With $h_E \simeq h_{T'} \simeq h_{T''}$, an inverse estimate and $0\le\beta_E\le1$ prove that
\begin{align*}
 h_E \, \norm{\boldsymbol{A}^{1/2}\nabla v}{T' \cup T''} 
 \lesssim \norm{v}{T' \cup T''}
 \le \norm{r_\blacktriangle}{T' \cup T''}
 \lesssim h_E^{1/2} \, \norm{[\boldsymbol{A}\nabla u_\blacktriangle \cdot n]}{E}.
\end{align*}
This leads to 
\begin{align*}
 \varrho_\blacktriangle(E,u_\blacktriangle) = h_E^{1/2} \, \norm{[\boldsymbol{A}\nabla u_\blacktriangle \cdot n]}{E}
 \lesssim \norm{\boldsymbol{A}^{1/2}\nabla (u_\vartriangle - u_\blacktriangle)}{T' \cup T''}
 + h_E \, \norm{f + {\rm div}_\blacktriangle(\boldsymbol{A}\nabla u_\blacktriangle )}{T'\cup T''}
\end{align*}
and concludes the proof.
\end{proof}

\begin{proof}[Proof of~\eqref{eq1:rho:eff}]
The proof is split into three steps.

\textbf{Step 1.}
First, we consider
$d=2$ and
 $p\ge 2$. 
Since NVB is a binary refinement rule, there exists a coarsest refinement $\mathcal{T}_\star \in {\rm nvb}(\mathcal{T}_\blacktriangle)$,
where $T$ is only bisected once into triangles $T_1, T_2 \in \mathcal{T}_\star$, i.e., there exists $E = {\rm conv}\{z_1,z_2\} \in \mathcal{E}_\star$ which bisects the interior of $T$, such that $z_1 \in {\mathcal{N}}_\blacktriangle \subset {\mathcal{N}}_\vartriangle$ and $z_2 \in {\mathcal{N}}_\star \backslashÂ {\mathcal{N}}_\blacktriangle \subseteq {\mathcal{N}}_\vartriangle \backslashÂ {\mathcal{N}}_\blacktriangle$. With the corresponding hat functions $\phi_{\star,j}\in\mathcal{P}^1(\{T_1,T_2\})$, define the discrete bubble function
\begin{align}\label{eq:betaT}
\beta_T = \prod_{j=1}^{2} \phi_{\star,j} \in \mathcal{P}^{2}(\{T_1,T_2\}) \cap H^1_0(T) \subset {\mathcal{S}}^{2}_0(\mathcal{T}_\vartriangle)
\quad\text{with }{\rm supp}(\beta_T) = T;
\end{align}
we note that $\beta_T$ is, in fact, the ``classical'' edge bubble for the new edge $E$.
Recall that,  $\Pi_\blacktriangle$ is the $L^2(\Omega)$-orthogonal projection onto $\mathcal{P}^{p-2}(\mathcal{T}_\blacktriangle)$.
Let $q_\blacktriangle := \Pi_\blacktriangle (f + {\rm div}_\blacktriangle(\boldsymbol{A}\nabla u_\blacktriangle )) = \Pi_\blacktriangle f + {\rm div}_\blacktriangle(\boldsymbol{A}\nabla u_\blacktriangle ) \in \mathcal{P}^{p-2}(\mathcal{T}_\blacktriangle)$, where ${\rm div}_\blacktriangle$ is the $\mathcal{T}_\blacktriangle$-piecewise divergence and hence ${\rm div}_\blacktriangle(\boldsymbol{A}\nabla u_\blacktriangle ) \in \mathcal{P}^{p-2}(\mathcal{T}_\blacktriangle)$. Define $v := q_\blacktriangle \beta_T \in H^1_0(T) \cap {\mathcal{S}}^{p}_0(\mathcal{T}_\vartriangle)$. A scaling argument proves that
\begin{align*}
 \norm{q_\blacktriangle}{T}^2 
 \simeq \norm{q_\blacktriangle\beta_T^{1/2}}{T}^2
 = \int_T \big(q_\blacktriangle - (f + {\rm div}(\boldsymbol{A}\nabla u_\blacktriangle ))\big)\,v\,dx
 + \int_T (f + {\rm div}(\boldsymbol{A}\nabla u_\blacktriangle ))\,v\,dx.
\end{align*}
The first integral is estimated by 
\begin{align*}
 \int_T \big(q_\blacktriangle - (f + {\rm div}(\boldsymbol{A}\nabla u_\blacktriangle ))\big)\,v\,dx
 \le \norm{q_\blacktriangle - (f + {\rm div}(\boldsymbol{A}\nabla u_\blacktriangle ))}{T} \, \norm{v}{T}
 = \norm{(1-\Pi_\blacktriangle)f}{T} \, \norm{v}{T}.
\end{align*}
For the second integral, integration by parts and $v \in H^1_0(T)$ prove that
\begin{align*}
 \int_T (f + {\rm div}(\boldsymbol{A}\nabla u_\blacktriangle ))\,v\,dx
 =  \int_T fv\,dx - \int_T \boldsymbol{A}\nabla u_\blacktriangle \cdot \nabla v\,dx.
\end{align*}
Recall  that $v \in  H^1_0(T)\cap {\mathcal{S}}^p_0(\mathcal{T}_\vartriangle)$. Therefore,
\begin{align}\label{eq:rho:eff:problem}
\begin{split}
 \int_T (f + {\rm div}(\boldsymbol{A}\nabla u_\blacktriangle ))\,v\,dx
& \reff{eq:discrete}= \int_T \boldsymbol{A}\nabla (u_\vartriangle - u_\blacktriangle) \cdot \nabla v\,dx
 \\
 &\le \norm{\boldsymbol{A}^{1/2}\nabla (u_\vartriangle - u_\blacktriangle)}{T} \, \norm{\boldsymbol{A}^{1/2}\nabla v}{T}.
 \end{split}
\end{align}
An inverse estimate and $0 \le \beta_T \le 1$ prove that $h_T \, \norm{\boldsymbol{A}^{1/2}\nabla v}{T} \lesssim \norm{v}{T}
\le \norm{q_\blacktriangle}{T}$. Hence, 
\begin{align*}
 h_T \, \norm{q_\blacktriangle}{T}
 \lesssim h_T \, \norm{(1-\Pi_\blacktriangle)f}{T} + \norm{\boldsymbol{A}^{1/2}\nabla (u_\vartriangle - u_\blacktriangle)}{T}.
\end{align*}
The triangle inequality and $f + {\rm div}(\boldsymbol{A}\nabla u_\blacktriangle ) - q_\blacktriangle = (1 - \Pi_\blacktriangle) f$ yield that
\begin{align*}
 \varrho_\blacktriangle(T,v_\blacktriangle) = h_T \, \norm{f + {\rm div}(\boldsymbol{A}\nabla u_\blacktriangle )}{T}
 \le h_T \, \norm{(1 - \Pi_\blacktriangle)f}{T} + h_T\, \norm{q_\blacktriangle}{T}.
\end{align*}
Combining the last two estimates, we prove~\eqref{eq1:rho:eff} for $d=2$ and $p\ge 2$.

\textbf{Step 2.}
For $d=3$ and $p \ge 2$, we suppose that each facet of $T$ contains an interior node.
Since NVB is a binary refinement rule, there exists a coarsest refinement $\mathcal{T}_\star \in {\rm nvb}(\mathcal{T}_\blacktriangle)$ with this property. Then, $T$ is refined as depicted in Figure~\ref{fig:bisec17}.
Consider the product of hat functions for the highlighted nodes $\frac{1}{2}(\frac{1}{2}(z_{\pi(1)}+z_{\pi(4)})+z_{\pi(3)})$ and $\frac{1}{2}(z_{\pi(2)}+z_{\pi(3)})$ of Figure~\ref{fig:bisec17}.
This provides a discrete element bubble function $\beta_T\in\mathcal{P}^2(\{T\})\cap H_0^1(T)$.
Arguing as in Step~1, we conclude~\eqref{eq1:rho:eff}.

\textbf{Step 3.}
Finally, suppose that $d \ge 2$, $p\ge 1$, and $T$ contains an interior node $z \in {\mathcal{N}}_\vartriangle$.
Since NVB is a binary refinement rule, there exists a coarsest refinement $\mathcal{T}_\star \in {\rm nvb}(\mathcal{T}_\blacktriangle)$ such that $T$ contains an interior node $z \in {\mathcal{N}}_\star$.
In particular, the corresponding hat function $\beta_T := \phi_{\star,z}Â \in {\mathcal{S}}^1_0(\mathcal{T}_\star) \subseteq  {\mathcal{S}}^1_0(\mathcal{T}_\vartriangle)$ 
satisfies that ${\rm supp}(\beta_T) \subseteq T$ and $|{\rm supp}(\beta_T)| \simeq |T|$, and
may thus serve as an element bubble function. 
Defining $q_\blacktriangle := \pi_\blacktriangle (f + {\rm div}_\blacktriangle(\boldsymbol{A}\nabla u_\blacktriangle )) = \pi_\blacktriangle f + {\rm div}_\blacktriangle(\boldsymbol{A}\nabla u_\blacktriangle ) \in \mathcal{P}^{p-1}(\mathcal{T}_\blacktriangle)$, we conclude~\eqref{eq1:rho:eff} as in Step~1.
\end{proof}%

\section{Proof of Theorem~\ref{thm:abstract}}
\label{section:proof}

\subsection{Proof of efficiency and reliability~(\ref{eq:reliable})}
Recall the different estimators from~\eqref{table}. 
The proof is split into several steps. 

{\bf Step~1.} 
We recall that the residual error estimator \eqref{eq:eta} satisfies that
\begin{align}\label{eq:aux3}
 \norm{\boldsymbol{A}^{1/2}\nabla(u-u_\bullet)}{\Omega} + {\rm osc}_\bullet
\stackrel{\eqref{eq:rho:releff}} \simeq \varrho_\bullet(u_\bullet)
 \simeq \varrho_\bullet(\mathcal{E}_\bullet,u_\bullet) + \varrho_\bullet(\mathcal{T}_\bullet,u_\bullet).
\end{align}
Moreover, the stability from Remark~\ref{rem:local stability rho} implies that
\begin{align}\label{eq:aux4}
 \norm{\boldsymbol{A}^{1/2}\nabla(\widehat u_\bullet - u_\bullet)}{\Omega}+ \varrho_\bullet(\mathcal{T}_\bullet)
 \simeq \norm{\boldsymbol{A}^{1/2}\nabla(\widehat u_\bullet - u_\bullet)}{\Omega}+ {\rm res}_\bullet(\mathcal{T}_\bullet,\widehat u_\bullet) .
\end{align}

{\bf Step~2.} 
According to~\eqref{eq:hh2:efficient}, it holds that
\begin{align*}
\lambda_\bullet(\widehat u_\bullet)
\le\norm{\boldsymbol{A}^{1/2}\nabla(\widehat u_\bullet-u_\bullet)}{\Omega}
 \le \norm{\boldsymbol{A}^{1/2}\nabla(u-u_\bullet)}{\Omega}.
 \end{align*}
Moreover, it holds that
\begin{align*}
  {\rm res}_\bullet(\mathcal{T}_\bullet,\widehat u_\bullet)^2
&  \reff{eq:aux4}\lesssim  
 \norm{\boldsymbol{A}^{1/2}\nabla(\widehat u_\bullet-u_\bullet)}{\Omega}^2 
 + \varrho_\bullet(\mathcal{T}_\bullet)^2 \reff{eq:aux3}\lesssim 
 \norm{\boldsymbol{A}^{1/2}\nabla(u-u_\bullet)}{\Omega}^2 + {\rm osc}_\bullet^2.
\end{align*}
In any case (cf.~\eqref{table}), the estimator equivalence~\eqref{eq2:hh2:equivalent} proves efficiency
$$
\eta(\widehat u_\bullet)^2 \lesssim \norm{\boldsymbol{A}^{1/2}\nabla(u-u_\bullet)}{\Omega}^2 + {\rm osc}_\bullet^2.
$$

{\bf Step~3.} 
Recall that the refinement employed in Algorithm~\ref{algorithm} satisfies (M3). Therefore, Lemma~\ref{lemma:rho:efficiency} implies that
\begin{align*}
 \varrho_\bullet(\mathcal{E}_\bullet) 
 \lesssim \norm{\boldsymbol{A}^{1/2}\nabla(\widehat u_\bullet - u_\bullet)}{\Omega} + \varrho_\bullet(\mathcal{T}_\bullet).
\end{align*}
Hence, we are led to
\begin{align}\label{eq:some reliability}
 \norm{\boldsymbol{A}^{1/2}\nabla(u-u_\bullet)}{\Omega} + {\rm osc}_\bullet
 \reff{eq:aux3}\simeq \varrho_\bullet(\mathcal{E}_\bullet) + \varrho_\bullet(\mathcal{T}_\bullet)
 &\lesssim \norm{\boldsymbol{A}^{1/2}\nabla(\widehat u_\bullet - u_\bullet)}{\Omega} + \varrho_\bullet(\mathcal{T}_\bullet).
\end{align}
In the first and fourth case of \eqref{table}, the equivalence \eqref{eq2:hh2:equivalent} of the $(h-h/2)$-type error estimators shows that  $\eta_\bullet(\widehat u_\bullet) \simeq \norm{\boldsymbol{A}^{1/2}\nabla(\widehat u_\bullet - u_\bullet)}{\Omega} + {\rm res}_\bullet(\mathcal{T}_\bullet,\widehat u_\bullet)$. 
This yields that
\begin{align*}
 \norm{\boldsymbol{A}^{1/2}\nabla(\widehat u_\bullet - u_\bullet)}{\Omega} + \varrho_\bullet(\mathcal{T}_\bullet) \reff{eq:aux4}\simeq \norm{\boldsymbol{A}^{1/2}\nabla(\widehat u_\bullet - u_\bullet)}{\Omega}
 + {\rm res}_\bullet(\mathcal{T}_\bullet,\widehat u_\bullet)
 \simeq \eta_\bullet(\widehat u_\bullet).
\end{align*}
In the other cases of \eqref{table},  the  equivalence \eqref{eq2:hh2:equivalent} shows that  $\eta_\bullet(\widehat u_\bullet) \simeq \norm{\boldsymbol{A}^{1/2}\nabla(\widehat u_\bullet - u_\bullet)}{\Omega} + {\rm osc}_\bullet$.
We recall that according to~\eqref{table}, it holds that either $p\ge2$ or that the refinement ensures (M3').
Therefore, Lemma~\ref{lemma:rho:efficiency} implies again that
\begin{align*}
\varrho_\bullet(\mathcal{T}_\bullet) 
\lesssim \norm{\boldsymbol{A}^{1/2}\nabla(\widehat u_\bullet - u_\bullet)}{\Omega} + {\rm osc}_\bullet. 
\end{align*}
Then, we are led to
\begin{align*}
 \norm{\boldsymbol{A}^{1/2}\nabla(\widehat u_\bullet - u_\bullet)}{\Omega} + \varrho(\mathcal{T}_\bullet)\,\lesssim\, \norm{\boldsymbol{A}^{1/2}\nabla(\widehat u_\bullet - u_\bullet)}{\Omega} + {\rm osc}_\bullet
 \simeq \eta(\widehat u_\bullet).
\end{align*}
In any case, this proves that
$$
\norm{\boldsymbol{A}^{1/2}\nabla(u-u_\bullet)}{\Omega}^2 + {\rm osc}_\bullet^2 \lesssim \eta(\widehat u_\bullet)^2.
$$
This concludes the proof.\qed

\subsection{Proof of~(\ref{eq:linear})--(\ref{eq:optimal})}
In the following, we verify that the $\lambda_\bullet$-based error estimators $\eta_\bullet$ from~\eqref{table} satisfy the \emph{axioms of adaptivity} from~\cite{axioms}. 
To prove linear convergence with optimal rates for the $\mu_\bullet$-based error estimators, we then exploit the local equivalence~\eqref{eq1:hh2:equivalent}.
We stress that unlike the various {\sl a~posteriori} error estimators in~\cite{ks,axioms}, the $(h-h/2)$-type estimators $\eta_\bullet$ are \emph{not} locally equivalent to the residual error estimator. Throughout, let $\mathcal{T}_\bullet \in {\rm nvb}(\mathcal{T}_0)$.

\begin{lemma}[local stability of $\boldsymbol{\lambda_\bullet(\cdot)}$]\label{lem:local stability lambda}
Let $\mathcal{T}_\circ \in {\rm nvb}(\mathcal{T}_\bullet)$.
For all $\widehat v_\bullet \in {\mathcal{S}}^p_0(\widehat{\mathcal{T}}_\bullet)$ and all
$\widehat v_\circ \in {\mathcal{S}}^p_0(\widehat{\mathcal{T}}_\circ)$, it  holds that
\begin{align}\label{eq:mu:stability}
 |\lambda_\circ(T,\widehat v_\circ) - \lambda_\bullet(T,\widehat v_\bullet) | 
 \le \norm{\boldsymbol{A}^{1/2}\nabla(\widehat v_\circ - \widehat v_\bullet)}{T}
 \quad \text{for all } T \in \mathcal{T}_\bullet \cap \mathcal{T}_\circ.
\end{align}
In particular, this implies that
\begin{align}\label{eq:mu:stability2}
 |\lambda_\circ(\mathcal{T}_\bullet \cap \mathcal{T}_\circ,\widehat v_\circ) 
 - \lambda_\bullet(\mathcal{T}_\bullet \cap \mathcal{T}_\circ,\widehat v_\bullet) | 
 \le  \norm{\boldsymbol{A}^{1/2}\nabla(\widehat v_\circ - \widehat v_\bullet)}{\bigcup(\mathcal{T}_\bullet \cap \mathcal{T}_\circ)}.
\end{align}%
Further, there exists  $C_{\rm stb}>0$ such that the $\lambda_\bullet$-based estimators $\eta_\bullet$ from~\eqref{table} satisfy that
\begin{align}\label{eq:eta stable}
 |\eta_\circ({\mathcal{S}},\widehat v_\circ) 
 - \eta_\bullet({\mathcal{S}},\widehat v_\bullet) | 
 \le C_{\rm stb} \norm{\boldsymbol{A}^{1/2}\nabla(\widehat v_\circ - \widehat v_\bullet)}{\bigcup{\mathcal{S}}}\quad\text{for all }{\mathcal{S}}\subseteq \mathcal{T}_\bullet \cap \mathcal{T}_\circ.
\end{align}
The constant $C_{\rm stb}$ depends only on $d$,  $\boldsymbol{A}$, $p$, $C_{\rm son}$, and shape-regularity of $\mathcal{T}_0$.
\end{lemma}

\begin{proof}
We prove the lemma in two steps.

{\bf Step~1.}
Note that $\pi_\circ$ and $\pi_\bullet$ coincide on $T$.
The triangle inequality thus proves that
\begin{align*}
 \lambda_\circ(T,\widehat v_\circ) 
 = \norm{(1-\pi_\bullet) \boldsymbol{A}^{1/2}\nabla \widehat v_\circ}{T}
 \le  \lambda_\bullet(T,\widehat v_\bullet)
 + \norm{(1-\pi_\bullet)\boldsymbol{A}^{1/2}\nabla(\widehat v_\circ - \widehat v_\bullet)}{T}.
\end{align*}
The same argument shows that
\begin{align*}
 \lambda_\bullet(T,\widehat v_\bullet) 
 = \norm{(1-\pi_\bullet) \boldsymbol{A}^{1/2}\nabla \widehat v_\bullet}{T}
 \le  \lambda_\circ(T,\widehat v_\circ)
 + \norm{(1-\pi_\bullet)\boldsymbol{A}^{1/2}\nabla(\widehat v_\circ - \widehat v_\bullet)}{T}.
\end{align*}
Together with $\norm{(1-\pi_\bullet)\boldsymbol{A}^{1/2}\nabla(\widehat v_\circ - \widehat v_\bullet)}{T} 
\le \norm{\boldsymbol{A}^{1/2}\nabla(\widehat v_\circ - \widehat v_\bullet)}{T}$, we conclude the proof of \eqref{eq:mu:stability}, which immediately yields \eqref{eq:mu:stability2}.

{\bf Step~2.}
If   $\eta_\bullet(\widehat v_\bullet)^2=\lambda_\bullet(\widehat v_\bullet)^2+{\rm osc}_\bullet^2$,  \eqref{eq:eta stable} follows from Step~1 and the fact that ${\rm osc}_\circ(T)={\rm osc}_\bullet(T)$ for all $T\in\mathcal{T}_\bullet\cap\mathcal{T}_\circ$. 
The same argument works if $\eta_\bullet(\widehat v_\bullet)^2=\lambda_\bullet(\widehat v_\bullet)^2+{\rm apx}_\bullet^2$.
If   $\eta_\bullet(\widehat v_\bullet)^2=\lambda_\bullet(\widehat v_\bullet)^2+{\rm res}_\bullet(\widehat v_\bullet)^2$, we note that 
${\rm res}_\bullet(T,\widehat v_\bullet)={\rm res}_\circ(T,\widehat v_\bullet)$ for all $T\in\mathcal{T}_\bullet\cap\mathcal{T}_\circ$.
Therefore, \eqref{eq:eta stable} follows from Step~1 and Remark~\ref{rem:local stability rho} with $\mathcal{T}_\blacktriangle=\mathcal{T}_\circ$. 
\end{proof}

\begin{lemma}[local reduction of $\boldsymbol{\lambda_\bullet(\cdot)}$]\label{lem:local reduction lambda}
Let $\mathcal{M}_\bullet\subseteq\mathcal{T}_\bullet$ and $\mathcal{T}_\circ\in{\rm nvb}(\mathcal{T}_\vartriangle)$ with $\mathcal{T}_\vartriangle={\rm refine}(\mathcal{T}_\bullet,\mathcal{M}_\bullet)$.
For all $\widehat v_\bullet \in {\mathcal{S}}^p_0(\widehat{\mathcal{T}}_\bullet)$, it  holds that
\begin{align}\label{eq1:mu:reduction}
 \lambda_\circ(\set{T' \in \mathcal{T}_\circ}{T' \subset T},\widehat v_\bullet) = 0
 \quad \text{for all }  T \in \mathcal{M}_\bullet.
\end{align}
In particular, this implies that
\begin{align}\label{eq2:mu:reduction}
 \lambda_\circ\big(\set{T'\in\mathcal{T}_\circ}{T'\subset\bigcup\mathcal{M}_\bullet},\widehat v_\circ\big) \le 
 \norm{\boldsymbol{A}^{1/2}\nabla(\widehat v_\circ - \widehat v_\bullet)}{\bigcup\mathcal{M}_\bullet}. 
\end{align}
Further, there exist constants $0<q_{\rm red}<1$ and $C_{\rm red}>0$ such that the $\lambda_\bullet$-based estimators $\eta_\bullet$ from~\eqref{table} satisfy that 
\begin{align}\label{eq3:mu:reduction}
 \eta_\circ\big(\set{T'\in\mathcal{T}_\circ}{T'\subset\bigcup\mathcal{M}_\bullet},\widehat v_\circ\big) \le q_{\rm red} \, \eta_\bullet(\mathcal{M}_\bullet,\widehat v_\bullet)
+ C_{\rm red}\, \norm{\boldsymbol{A}^{1/2}\nabla(\widehat v_\circ - \widehat v_\bullet)}{\bigcup\mathcal{M}_\bullet}. 
\end{align}
The constant $q_{\rm red}$  depends only on $d$, while $C_{\rm red}$ depends additionally on  $\boldsymbol{A}$, $p$, $C_{\rm son}$, and shape-regularity of $\mathcal{T}_0$.
\end{lemma}

\begin{proof}
We prove the lemma in two steps. 

{\bf Step~1.}
Recall that NVB is a binary refinement rule. Therefore, $T \in \mathcal{M}_\bullet$ and (M2) imply that $\mathcal{T}_\circ|_T := \set{T' \in \mathcal{T}_\circ}{T' \subset T}$ is finer than $\widehat{\mathcal{T}}_\bullet|_T=\mathcal{T}_\vartriangle|_T$. This proves that $\norm{(1-\pi_\circ)\boldsymbol{A}^{1/2}\nabla \widehat v_\bullet}{T} = 0$ and hence~\eqref{eq1:mu:reduction}. 
The triangle inequality, the fact that orthogonal projections have operator norm one, and the Young inequality prove for all $\delta>0$ that
\begin{align*}
 &\lambda_\circ\big(\set{T'\in\mathcal{T}_\circ}{T'\subset\bigcup\mathcal{M}_\bullet}, \widehat v_\circ\big)^2\\
 &\quad\le (1+\delta^{-1})\,  \lambda_\circ\big(\set{T'\in\mathcal{T}_\circ}{T'\subset\bigcup\mathcal{M}_\bullet},\widehat v_\bullet\big)^2
+ (1+\delta) \, \norm{(1-\pi_\circ)\boldsymbol{A}^{1/2}\nabla(\widehat v_\circ - \widehat v_\bullet)}{\bigcup\mathcal{M}_\bullet}^2\\
&\quad \reff{eq1:mu:reduction}\le (1+\delta) \, \norm{\boldsymbol{A}^{1/2}\nabla(\widehat v_\circ - \widehat v_\bullet)}{\bigcup\mathcal{M}_\bullet}^2.
\end{align*}
With $\delta \to 0$, this concludes the proof of \eqref{eq2:mu:reduction}.

{\bf Step~2.}
Since $\norm{(1-\pi_\circ)(\cdot)}{T}\le\norm{(1-\pi_\bullet)(\cdot)}{T}$ for all $T\in\mathcal{T}_\bullet$ and  each marked element is bisected at least once, we have that 
\begin{align}\label{eq:reduction osc}
{\rm osc}_\circ\big(\set{T'\in\mathcal{T}_\circ}{T'\subset \bigcup\mathcal{M}_\bullet}\big)\le 2^{-1/d}\, {\rm osc}_\bullet(\mathcal{M}_\bullet).
\end{align}
The same argument is valid for the approximation terms ${\rm apx}_\circ$.
Moreover, Remark~\ref{rem:local stability rho} with the fact that each marked element is bisected at least once yields that 
\begin{align}\label{eq:reduction res}
\begin{split}
 {\rm res}_\circ\big(\set{T'\in\mathcal{T}_\circ}{T'\subset\bigcup\mathcal{M}_\bullet},\widehat v_\circ\big) &\le  2^{-1/d}\,{\rm res}_\bullet(\mathcal{M}_\bullet,\widehat v_\bullet)\\
 &\quad+C \,\norm{\boldsymbol{A}^{1/2}\nabla(\widehat v_\circ - \widehat v_\bullet)}{\bigcup\mathcal{M}_\bullet}.
\end{split}
\end{align}
Together with Step~1 and the Young inequality, \eqref{eq:reduction osc} and \eqref{eq:reduction res} conclude the proof.
\end{proof}

\begin{lemma}[discrete reliability of $\boldsymbol{\eta_{\bullet}(\cdot)}$]
\label{lemma:eta:drel}
There exists $C_{\rm drl}>0$ such that
\begin{align}\label{eq:eta:drel}
\norm{\boldsymbol{A}^{1/2}\nabla(\widehat u_\circ - \widehat u_\bullet)}{\Omega}
 \le C_{\rm drl}\,\eta_\bullet( \mathcal{T}_\bullet \backslash \mathcal{T}_\circ, \widehat u_\bullet)\quad\text{for all } 
\mathcal{T}_\circ \in {\rm nvb}( \mathcal{T}_\bullet ).
\end{align}
The constant $C_{\rm drl}$ depends only on $\boldsymbol{A}$, $p\ge 1$, and $\gamma$-shape regularity of $\mathcal{T}_0$.\qed
\end{lemma}

\begin{proof}
Due to the local equivalence \eqref{eq1:hh2:equivalent}, it suffices to consider the $\mu_\bullet$-based estimators from~\eqref{table}.
The proof is split into three steps.

{\bf Step~1.}
Let $v_\bullet := I_\bullet(\widehat u_\bullet) \in {\mathcal{S}}^p_0(\mathcal{T}_\bullet)$. 
We apply the discrete reliability~\eqref{eq:rho:drel} of the residual error estimator for $\mathcal{T}_\vartriangle=\widehat{\mathcal{T}}_\circ$ and $\mathcal{T}_\blacktriangle = \widehat{\mathcal{T}}_\bullet$. Together with (local) stability~\eqref{eq:rho:stab} of the residual error estimator, this proves that 
\begin{align*}
 \norm{\boldsymbol{A}^{1/2}\nabla(\widehat u_\circ - \widehat u_\bullet)}{\Omega}
 &\reff{eq:rho:drel}\lesssim 
 \widehat\varrho_\bullet\big((\widehat{\mathcal{T}}_\bullet \backslash \widehat{\mathcal{T}}_\circ) \cup (\widehat{\mathcal{E}}_\bullet^\Omega \backslash \widehat{\mathcal{E}}_\circ^\Omega), \widehat u_\bullet \big)
 \\&
 \reff{eq:rho:stab}\lesssim \widehat\varrho_\bullet \big((\widehat{\mathcal{T}}_\bullet \backslash \widehat{\mathcal{T}}_\circ) \cup (\widehat{\mathcal{E}}_\bullet^\Omega \backslash \widehat{\mathcal{E}}_\circ^\Omega), v_\bullet \big)
 + \norm{\boldsymbol{A}^{1/2}\nabla(\widehat u_\bullet - v_\bullet)}{\bigcup(\widehat{\mathcal{T}}_\bullet \backslash \widehat{\mathcal{T}}_\circ)},
\end{align*}
 since the patch of a refined facet $\widehat E\in\widehat {\mathcal{E}}_\bullet^\Omega\setminus\widehat{\mathcal{E}}_\circ^\Omega$ belongs to $\widehat{\mathcal{T}}_\bullet\setminus\widehat{\mathcal{T}}_\circ$.
Next, we show that $ \bigcup(\widehat{\mathcal{T}}_\bullet \backslash \widehat{\mathcal{T}}_\circ) \subseteq \bigcup(\mathcal{T}_\bullet \backslash \mathcal{T}_\circ)$.
Let $\widehat T\in\widehat{\mathcal{T}}_\bullet \backslash \widehat{\mathcal{T}}_\circ$ and  $T\in\mathcal{T}_\bullet$  be the unique father element, i.e., $\widehat T\subset T$.
If $T\in\mathcal{T}_\circ$, then (M2) implies that $\widehat T\in\widehat{\mathcal{T}}_\circ$, which contradicts the assumption $\widehat T\not\in\widehat{\mathcal{T}}_\circ$.
This concludes the desired inclusion.
Since the local weights of the residual error estimator  are decreasing for (uniform) mesh-refinement, this yields that
\begin{align*}
 \widehat\varrho_\bullet(\widehat{\mathcal{T}}_\bullet \backslash \widehat{\mathcal{T}}_\circ,v_\bullet)
 \le \varrho_\bullet(\mathcal{T}_\bullet \backslash \mathcal{T}_\circ,v_\bullet).
\end{align*}%
 According to  the discrete efficiency~\eqref{eq2:rho:eff} of the residual error estimator for $\mathcal{T}_\vartriangle=\widehat{\mathcal{T}}_\circ$ and $\mathcal{T}_\blacktriangle =\widehat{\mathcal{T}}_\bullet$, it holds that
\begin{align*}
 \widehat\varrho_\bullet(\widehat{\mathcal{E}}_\bullet^\Omega \backslash \widehat{\mathcal{E}}_\circ^\Omega,v_\bullet)
\lesssim
 \norm{\boldsymbol{A}^{1/2}\nabla(\widehat u_\bullet - v_\bullet)}{\bigcup(\widehat{\mathcal{T}}_\bullet \backslash  \widehat{\mathcal{T}}_\circ)} + \widehat \varrho_\bullet(\widehat{\mathcal{T}}_\bullet \backslash\widehat{\mathcal{T}}_\circ,v_\bullet).
\end{align*}%
Combining the last three estimates and using that $\bigcup(\widehat{\mathcal{T}}_\bullet \backslash \widehat{\mathcal{T}}_\circ) \subseteq \bigcup(\mathcal{T}_\bullet \backslash \mathcal{T}_\circ)$, we are led to
\begin{align}\label{eq:rel:proof1}
 \norm{\boldsymbol{A}^{1/2}\nabla(\widehat u_\circ - \widehat u_\bullet)}{\Omega}
 \lesssim  
 \varrho_\bullet (\mathcal{T}_\bullet \backslash \mathcal{T}_\circ, v_\bullet)
 + \norm{\boldsymbol{A}^{1/2}\nabla(\widehat u_\bullet - v_\bullet)}{\bigcup(\mathcal{T}_\bullet \backslash \mathcal{T}_\circ)}.
\end{align}

{\bf Step~2.}
For arbitrary $p\ge1$, 
 we may use stability~\eqref{eq:rho:stab} of the residual error estimator to see that
\begin{align*}
 \varrho_\bullet (\mathcal{T}_\bullet \backslash \mathcal{T}_\circ, v_\bullet)
& \simeq \widehat\varrho_\bullet\big(\set{\widehat T'\in\widehat{\mathcal{T}}_\bullet}{\widehat T'\subseteq\bigcup(\mathcal{T}_\bullet\setminus\mathcal{T}_\circ)},v_\bullet\big) %
\\ &\reff{eq:rho:stab}\lesssim \widehat\varrho_\bullet\big(\set{\widehat T'\in\widehat{\mathcal{T}}_\bullet}{\widehat T'\subseteq\bigcup(\mathcal{T}_\bullet\setminus\mathcal{T}_\circ)},\widehat u_\bullet\big)  
 + \norm{\boldsymbol{A}^{1/2}\nabla(\widehat u_\bullet - v_\bullet)}{\bigcup(\mathcal{T}_\bullet \backslash \mathcal{T}_\circ)}
 \\
 &\simeq {\rm res}_\bullet(\mathcal{T}_\bullet \backslash \mathcal{T}_\circ, \widehat u_\bullet) 
 + \norm{\boldsymbol{A}^{1/2}\nabla(\widehat u_\bullet - v_\bullet)}{\bigcup(\mathcal{T}_\bullet \backslash \mathcal{T}_\circ)}. 
\end{align*}
Combining this with~\eqref{eq:rel:proof1} and the definition of $\mu_\bullet( \mathcal{T}_\bullet \backslash \mathcal{T}_\circ, \widehat u_\bullet)$, we prove \eqref{eq:eta:drel} for 
$\eta_\bullet^2=\mu_\bullet^2 + {\rm res}_\bullet^2$.

{\bf Step~3.}
If the refinement ensures (M3'), we use the discrete efficiency~\eqref{eq1:rho:eff} of the residual error estimator for $\mathcal{T}_\vartriangle=\widehat{\mathcal{T}}_\bullet$ and $\mathcal{T}_\blacktriangle = \mathcal{T}_\bullet$ to see that
\begin{align*}
 \varrho_\bullet (\mathcal{T}_\bullet \backslash \mathcal{T}_\circ, v_\bullet)
 \lesssim \norm{\boldsymbol{A}^{1/2}\nabla(\widehat u_\bullet - v_\bullet)}{\bigcup(\mathcal{T}_\bullet \backslash \mathcal{T}_\circ)} + {\rm osc}_\bullet(\mathcal{T}_\bullet \backslash \mathcal{T}_\circ).
\end{align*}
Combining this with~\eqref{eq:rel:proof1}, we prove \eqref{eq:eta:drel} for $\eta_\bullet^2=\mu_\bullet^2+{\rm osc}_\bullet^2$.

{\bf Step~4.}
Finally, if $p\ge2$, we can argue along the  lines of Step~3 that \eqref{eq:eta:drel} holds for $\eta_\bullet^2=\mu_\bullet^2+{\rm apx}_\bullet^2$.
\end{proof}

\begin{lemma}[general quasi-orthogonality for $\boldsymbol{\eta_\bullet(\cdot)}$]\label{lem:general quasi orthogonality}
Consider Algorithm~\ref{algorithm} with $\eta_\bullet$ from~\eqref{table}.
Then, it holds that 
\begin{align}\label{eq:general quasi orthogonality}
 \sum_{j=\ell}^\infty \norm{\boldsymbol{A}^{1/2}\nabla(\widehat u_{j+1}-\widehat u_j)}{\Omega}^2
 \le C_{\rm rel}^2 \, \eta_\ell(\widehat u_\ell)^2 
 \quad\text{for all } \ell \in \mathbb{N}_0,
\end{align}
where $C_{\rm rel}>0$ is the reliability constant from \eqref{eq:reliable}.
\end{lemma}

\begin{proof}
For $\mathcal{T}_\circ \in {\rm refine}(\mathcal{T}_\bullet)$, there holds the Pythagoras identity
\begin{align*}
 \norm{\boldsymbol{A}^{1/2}\nabla(u-\widehat u_\circ)}{\Omega}^2
 + \norm{\boldsymbol{A}^{1/2}\nabla(\widehat u_\circ - \widehat u_\bullet)}{\Omega}^2
 = \norm{\boldsymbol{A}^{1/2}\nabla(u-\widehat u_\bullet)}{\Omega}^2.
\end{align*}
Applying this for $\mathcal{T}_\circ = \mathcal{T}_{j+1}$ and $\mathcal{T}_\bullet = \mathcal{T}_j$, we are led to
\begin{align*}
 \sum_{j=\ell}^N \norm{\boldsymbol{A}^{1/2}\nabla(\widehat u_{j+1} - \widehat u_j)}{\Omega}^2
 &= \sum_{j=\ell}^N \big( \norm{\boldsymbol{A}^{1/2}\nabla(u-\widehat u_{j})}{\Omega}^2
 - \norm{\boldsymbol{A}^{1/2}\nabla(u-\widehat u_{j+1})}{\Omega}^2 \big)
 \\&
 = \norm{\boldsymbol{A}^{1/2}\nabla(u-\widehat u_{\ell})}{\Omega}^2
 - \norm{\boldsymbol{A}^{1/2}\nabla(u-\widehat u_{N+1})}{\Omega}^2
 \\&
 \le \norm{\boldsymbol{A}^{1/2}\nabla(u-\widehat u_{\ell})}{\Omega}^2.
\end{align*}
According to  the Pythagoras theorem \eqref{eq:pythagoras} and reliability~\eqref{eq:reliable},  last term satisfies that
\begin{align}\label{eq:reliable2}
 \norm{\boldsymbol{A}^{1/2}\nabla(u-\widehat u_{\ell})}{\Omega}^2\le  \norm{\boldsymbol{A}^{1/2}\nabla(u-u_{\ell})}{\Omega}^2\le C_{\rm rel}^2 \,\eta_\ell(\widehat u_\ell)^2.
\end{align}
As $N\to\infty$, we conclude the proof.
\end{proof}

\begin{proof}[Proof of~\eqref{eq:linear}--\eqref{eq:optimal}]
We prove the assertion in three steps.

\textbf{Step~1:}
First, we  consider only the $\lambda_\bullet$-based estimators from~\eqref{table}.
With ${\mathcal{S}}_{\ell+1,\ell}:=\set{T'\in\mathcal{T}_{\ell+1}}{T'\not\subset\bigcup\mathcal{M}_\ell}$ being the sons of the non-marked elements, it holds that  
\begin{align*}
\eta_{\ell+1}(\widehat u_{\ell+1})^2=\eta_{\ell+1}({\mathcal{S}}_{\ell+1,\ell},\widehat u_{\ell+1})^2+\eta_{\ell+1}(\mathcal{T}_{\ell+1}\setminus{\mathcal{S}}_{\ell+1,\ell},\widehat u_{\ell+1})^2.
\end{align*}
Stability~\eqref{eq:eta stable} with $\mathcal{T}_\bullet=\mathcal{T}_\circ=\mathcal{T}_{\ell+1}$, reduction~\eqref{eq3:mu:reduction} with $\mathcal{T}_\bullet=\mathcal{T}_\ell$ and $\mathcal{T}_\circ=\mathcal{T}_{\ell+1}$, and the Young inequality show for arbitrary $\delta>0$ that
\begin{align*}
\eta_{\ell+1}(\widehat u_{\ell+1})^2
&\le(1+\delta)\big(\eta_{\ell+1}({\mathcal{S}}_{\ell+1,\ell},\widehat u_{\ell})^2
+q_{\rm red}^2\,\eta_{\ell}(\mathcal{M}_\ell,\widehat u_{\ell})^2\big)\\
&\quad+(1+\delta^{-1})(C_{\rm stb}^2 +C_{\rm red}^2) \norm{\boldsymbol{A}^{1/2}\nabla(\widehat u_{\ell+1}-\widehat u_\ell)}{\Omega}^2.
\end{align*}
Due to the facts that $h_{\ell+1}\le h_\ell$ and $\norm{(1-\pi_{\ell+1})(\cdot)}{T}\le \norm{(1-\pi_\ell)(\cdot)}{T}$ as well as $\norm{(1-\Pi_{\ell+1})(\cdot)}{T}\le \norm{(1-\Pi_\ell)(\cdot)}{T}$ for all $T\in\mathcal{T}_{\ell+1}$, we have that 
\begin{align}\label{eq:quasi-mon aux}
\eta_{\ell+1}({\mathcal{S}}_{\ell+1,\ell},\widehat u_{\ell})^2\le \eta_{\ell}(\mathcal{T}_{\ell}\setminus\mathcal{M}_\ell,\widehat u_{\ell})^2=\eta_\ell(\widehat u_\ell)^2- \eta_{\ell}(\mathcal{M}_\ell,\widehat u_{\ell})^2.
\end{align}
Together with the D\"orfler marking in Algorithm~\ref{algorithm} {\rm (iii)}, we derive  the estimator reduction 
\begin{align}\label{eq:estimator reduction}
\begin{split}
\eta_{\ell+1}(\widehat u_{\ell+1})^2&\le 
(1+\delta)\big(1-(1-q_{\rm red}^2)\theta\big)\eta_\ell(\widehat u_\ell)^2\\
&\quad +(1+\delta^{-1}) (C_{\rm stb}^2 +C_{\rm red}^2)\,\norm{\boldsymbol{A}^{1/2}\nabla(\widehat u_{\ell+1}-\widehat u_{\ell})}{\Omega}^2.
\end{split}
\end{align}
According to \cite[Prop.~4.10]{axioms},  general quasi-orthogonality~\eqref{eq:general quasi orthogonality}, reliability~\eqref{eq:reliable2},
 and estimator reduction \eqref{eq:estimator reduction} yield linear convergence \eqref{eq:linear} for the $\lambda_\bullet$-based estimators from~\eqref{table}.

\textbf{Step~2:}
Again, we only consider the $\lambda_\bullet$-based estimators from~\eqref{table}.
The first inequality in \eqref{eq:optimal} follows immediately from \eqref{eq:lower optimal estimate} and \eqref{eq:refine or nvb}. 
We prove the second inequality.
Similarly as in \eqref{eq:quasi-mon aux}, one shows that  $\eta_\circ(\mathcal{T}_\circ,\widehat u_\circ)\lesssim\eta_\bullet(\mathcal{T}_\bullet,\widehat u_\bullet)+\norm{\boldsymbol{A}^{1/2}\nabla(\widehat u_\circ - \widehat u_\bullet)}{\Omega}$  for arbitrary $\mathcal{T}_\bullet\in{\rm nvb}(\mathcal{T}_0)$ and $\mathcal{T}_\circ\in{\rm nvb}(\mathcal{T}_\bullet)$.
Then, discrete reliability \eqref{eq:eta:drel} immediately implies quasi-monotonicity
\begin{align}\label{eq:eta monotone}
\eta_\circ(\widehat u_\circ)\lesssim \eta_\bullet(\widehat u_\bullet).
\end{align}
Altogether, stability \eqref{eq:eta stable}, discrete reliability \eqref{eq:eta:drel},  quasi-monotonicity \eqref{eq:eta monotone},  and the overlay estimate \cite[Eq.~(2.2)]{ckns} for ${\rm nvb}(\cdot)$ allow to apply optimality of D\"orfler marking  \cite[Prop.~4.12]{axioms} and the comparison lemma \cite[Lem.~4.14]{axioms}, which show the following:
There exists a constant $0<\theta_{\rm opt}^{\lambda}<1$ such that for $0<\theta<\theta_{\rm opt}^\lambda$, $s>0$, and all meshes $\mathcal{T}_\bullet$, there exists a refinement $\widetilde{\mathcal{T}}_\bullet\in{\rm nvb}(\mathcal{T}_\bullet)$ such that 
\begin{subequations}
\begin{align}
\theta \, \eta_\bullet(\widehat u_\bullet)^2 &\le \eta_\bullet(\mathcal{T}_\bullet\setminus\widetilde{\mathcal{T}}_\bullet, \widehat u_\bullet)^2,\label{eq:comp1}\\
\#\widetilde{\mathcal{T}}_\bullet-\#\mathcal{T}_\bullet&\lesssim \norm{u}{\mathbb{A}_s^{\eta}}^{1/s}\,\eta_\bullet(\widehat u_\bullet)^{-1/s}.\label{eq:comp2}
\end{align}
\end{subequations}
For $0<\theta<\theta_{\rm opt}^\lambda$, and arbitrary  $\ell\in\mathbb{N}$ with $\mathcal{T}_\ell\neq\mathcal{T}_0$, the  closure estimate \cite[Section~2.6]{ckns} for ${\rm refine}(\cdot)$,  and  minimality of the set of marked elements $\mathcal{M}_\ell$ yield that
\begin{align*}
\#\mathcal{T}_\ell-\#\mathcal{T}_0+1&\,\lesssim\,\#\mathcal{T}_\ell-\#\mathcal{T}_0\lesssim\sum_{j=0}^{\ell-1}\#\mathcal{M}_j\\
&\reff{eq:comp1}\le\sum_{j=0}^{\ell-1}\#(\mathcal{T}_j\setminus\widetilde{\mathcal{T}}_j) 
\le \sum_{j=0}^{\ell-1}(\#\widetilde{\mathcal{T}}_j-\#\mathcal{T}_j) \reff{eq:comp2}\lesssim \norm{u}{\mathbb{A}_s^{\eta}}^{1/s}\sum_{j=0}^{\ell-1}\eta_j(\widehat u_j)^{-1/s}.
\end{align*}
With the linear convergence \eqref{eq:linear}, one can elementarily show that 
\begin{align*}
\sum_{j=0}^{\ell-1}\eta_j(\widehat u_j)^{-1/s}\lesssim \eta_\ell(\widehat u_\ell)^{-1/s};
\end{align*} see, e.g., \cite[Lemma~4.9]{axioms}.
Since $\eta_0(\widehat u_0)\le \norm{u}{\mathbb{A}_s^{\eta}}$, this concludes \eqref{eq:optimal} for the $\lambda_\bullet$-based estimators from \eqref{table}.

\textbf{Step~3:}
Finally, we consider the $\mu_\bullet$-based estimators from~\eqref{table}.
Recall the local equivalence of $\lambda_\bullet$ and $\mu_\bullet$ which immediately transfers to the corresponding estimators from~\eqref{table}. 
Hence, $\mu_\bullet$-based D\"orfler marking $\theta\eta_\bullet(\widehat u_\bullet)\le \eta_\bullet(\mathcal{M}_\bullet,\widehat u_\bullet)$ with parameter $\theta$ and marked elements $\mathcal{M}_\bullet$  implies $\lambda_\bullet$-based D\"orfler marking with parameter $C_{\rm hh2}^{-1}\,\theta$ and the same marked elements $\mathcal{M}_\bullet$ for the corresponding $\lambda_\bullet$-based estimator and vice versa. 

Therefore, $\mu_\bullet$-based D\"orfler marking implies linear convergence of the corresponding $\lambda_\bullet$-based estimator and by equivalence also linear convergence of the $\mu_\bullet$-based estimator.
Moreover, for sufficiently small $\theta$, the $\lambda_\bullet$-based estimator converges with optimal algebraic rates and hence does the $\mu_\bullet$-based estimator. Details are left to the reader.
\end{proof}

\section{Numerical experiments}
\label{section:numerics}

\newcommand{\logLogSlopeTriangle}[6]
{

    \pgfplotsextra
    {
        \pgfkeysgetvalue{/pgfplots/xmin}{\xmin}
        \pgfkeysgetvalue{/pgfplots/xmax}{\xmax}
        \pgfkeysgetvalue{/pgfplots/ymin}{\ymin}
        \pgfkeysgetvalue{/pgfplots/ymax}{\ymax}

        \pgfmathsetmacro{\xArel}{#1}
        \pgfmathsetmacro{\yArel}{#3}
        \pgfmathsetmacro{\xBrel}{#1-#2}
        \pgfmathsetmacro{\yBrel}{\yArel}
        \pgfmathsetmacro{\xCrel}{\xArel}

        \pgfmathsetmacro{\lnxB}{\xmin*(1-(#1-#2))+\xmax*(#1-#2)} 
        \pgfmathsetmacro{\lnxA}{\xmin*(1-#1)+\xmax*#1} 
        \pgfmathsetmacro{\lnyA}{\ymin*(1-#3)+\ymax*#3} 
        \pgfmathsetmacro{\lnyC}{\lnyA+#4*(\lnxA-\lnxB)}
        \pgfmathsetmacro{\yCrel}{\lnyC-\ymin)/(\ymax-\ymin)} 

        \coordinate (A) at (rel axis cs:\xArel,\yArel);
        \coordinate (B) at (rel axis cs:\xBrel,\yBrel);
        \coordinate (C) at (rel axis cs:\xCrel,\yCrel);

        \draw[#5]   (A)-- node[pos=0.5,anchor=south] {#6 1}
                    (B)-- 
                    (C)-- node[pos=0.5,anchor=west] {#6 #4}
                    cycle;
    }
}
\newcommand{\logLogSlopeTrianglelow}[6]
{

    \pgfplotsextra
    {
        \pgfkeysgetvalue{/pgfplots/xmin}{\xmin}
        \pgfkeysgetvalue{/pgfplots/xmax}{\xmax}
        \pgfkeysgetvalue{/pgfplots/ymin}{\ymin}
        \pgfkeysgetvalue{/pgfplots/ymax}{\ymax}

        \pgfmathsetmacro{\xArel}{#1}
        \pgfmathsetmacro{\yArel}{#3}
        \pgfmathsetmacro{\xBrel}{#1-#2}
        \pgfmathsetmacro{\yBrel}{\yArel}
        \pgfmathsetmacro{\xCrel}{\xBrel}

        \pgfmathsetmacro{\lnxB}{\xmin*(1-(#1-#2))+\xmax*(#1-#2)} 
        \pgfmathsetmacro{\lnxA}{\xmin*(1-#1)+\xmax*#1} 
        \pgfmathsetmacro{\lnyA}{\ymin*(1-#3)+\ymax*#3} 
        \pgfmathsetmacro{\lnyC}{\lnyA-#4*(\lnxA-\lnxB)}
        \pgfmathsetmacro{\yCrel}{\lnyC-\ymin)/(\ymax-\ymin)} 

        \coordinate (A) at (rel axis cs:\xArel,\yArel);
        \coordinate (B) at (rel axis cs:\xBrel,\yBrel);
        \coordinate (C) at (rel axis cs:\xCrel,\yCrel);

        \draw[#5]   (A)-- node[pos=0.5,anchor=north] {#6 $1$}
                    (B)-- node[pos=0.5,anchor=east] {#6 $#4$}
                    (C)-- 
                    cycle;
    }
}


In this section, we present three examples in two dimensions to empirically verify our theoretical results.
For all examples, we choose the L-shaped domain 
\begin{align*}
\Omega=\linebreak(-1,1)^2\backslash \big([0,1]\times[-1,0]\big).
\end{align*} 
The uniform initial mesh $\mathcal{T}_0$ consists of $12$ triangles. 
We run Algorithm~\ref{algorithm} either with $\theta=1$ for uniform refinement or with $\theta=0.5$
for adaptive refinement based on the indicators from~\eqref{table}
\begin{align}
 \label{eq:refinementindicator}
\eta_\bullet(T,\widehat u_\bullet)^2 :=\lambda_\bullet(T,\widehat u_\bullet)^2+
\begin{cases}
{\rm res}_\bullet(T,\widehat u_\bullet)^2\quad & \text{if } p=1\text{ and (M3)}, \\
{\rm osc}_\bullet(T)^2\quad & \text{if } p\in\{1,2\}\text{ and (M3')}, \\
{\rm apx}_\bullet(T)^2\quad & \text{if } p=2\text{ and (M3)}. 
\end{cases}
\end{align}
We consider the model problem~\eqref{eq:strong} with $\boldsymbol{A}=\mathbf{I}$, where we 
now allow inhomogeneous Dirichlet conditions. 
In our examples, we replace the Dirichlet data 
by its nodal interpolant for the numerical calculations.
In all figures, we plot the error $\big(\norm{\nabla(u - u_\bullet)}{\Omega}^2+{{\rm osc}_\bullet^2}\big)^{1/2}$ (if available)
as well as the overall error estimators 
\begin{alignat*}{2}
\lambda_\bullet' 
&:= \bigg(\sum_{T\in\mathcal{T}_\bullet}\big[\lambda_\bullet(T,\widehat u_\bullet)^2+{\rm res}_\bullet(T,\widehat u_\bullet)^2\big] \! \bigg)^{1/2}, 
&\quad 
\mu_\bullet' 
&:= \bigg(\sum_{T\in\mathcal{T}_\bullet}\big[\mu_\bullet(T,\widehat u_\bullet)^2+{\rm res}_\bullet(T,\widehat u_\bullet)^2\big] \! \bigg)^{1/2}, 
\\
\lambda_\bullet''
&:= \bigg(\sum_{T\in\mathcal{T}_\bullet}\big[\lambda_\bullet(T,\widehat u_\bullet)^2+{\rm osc}_\bullet(T)^2\big] \! \bigg)^{1/2}, 
& 
\mu_\bullet''
&:= \bigg(\sum_{T\in\mathcal{T}_\bullet}\big[\mu_\bullet(T,\widehat u_\bullet)^2 +{\rm osc}_\bullet(T)^2\big] \! \bigg)^{1/2}
\end{alignat*}
for uniform (unif.) and adaptive (adap.) refinement with respect to the number of elements $N$ of
$\mathcal{T}_\bullet$. 
For $p=2$, we additionally plot the overall estimators
\begin{alignat*}{2}
\lambda_\bullet''' 
&:= \bigg(\sum_{T\in\mathcal{T}_\bullet}\big[\lambda_\bullet(T,\widehat u_\bullet)^2+{\rm apx}_\bullet(T)^2\big] \! \bigg)^{1/2}, 
&\quad 
\mu_\bullet'''
&:= \bigg(\sum_{T\in\mathcal{T}_\bullet}\big[\mu_\bullet(T,\widehat u_\bullet)^2 +{\rm apx}_\bullet(T)^2\big] \! \bigg)^{1/2}.
\end{alignat*}
We use either
three or five bisections for refinement of a marked element; cf.~Figure~\ref{fig:nvb}.
This guarantees (M3) or (M3'). 
Note that for uniform refinement with (M3), 
the convergence order ${\mathcal{O}}(N^{-s})$ with $s>0$ corresponds to ${\mathcal{O}}(h^{2s})$, where $h:=\max_{T\in \mathcal{T}_\bullet}h_T$. 
This does not hold for uniform refinement with (M3'), since one refinement step leads to element sons of different levels.
In particular, the uniform convergence rates seem to be slightly worse than naively expected.
However, plotted over the maximal mesh-size $h$, one obtains the expected rates (not displayed). 

\begin{figure}[t]
\centering
\subfigure[${\mathcal{S}}^1$-FEM with (M3) and $\eta_\bullet(T,\widehat u_\bullet)^2 :=\lambda_\bullet(T,\widehat u_\bullet)^2+{\rm res}_\bullet(T,\widehat u_\bullet)^2$.]{
\label{subfig:mybsp1errornvb3P1}
\begin{tikzpicture}[scale=0.5]
\begin{loglogaxis}[width=0.92\textwidth,
xlabel={\small \# elements}, ylabel={\small error and error estimator}, font={\scriptsize}, 
ymin=5e-6,ymax=2e1,
legend style={font=\small, draw=none, fill=none, cells={anchor=west}, legend pos=south west}]
\addplot [color=red!50!green, mark=x, line width=0.5pt, mark size=3pt, mark options={solid, red!50!green}]                     table [x=nrelements,y=eta1] {datanumerics/mybsp1-P1FEM-nvb3eta1T2uni.dat};
\addplot [color=gray, mark=asterisk, line width=0.5pt, mark size=3pt, mark options={solid, gray}]                              table [x=nrelements,y=eta2] {datanumerics/mybsp1-P1FEM-nvb3eta1T2uni.dat};
\addplot [color=black!50!red, mark=+, line width=0.5pt, mark size=3pt, mark options={solid, black!50!red}]                     table [x=nrelements,y=eta4] {datanumerics/mybsp1-P1FEM-nvb3eta1T2uni.dat};
\addplot [color=violet, mark=diamond, line width=0.5pt, mark size=3pt, mark options={solid, violet}]                           table [x=nrelements,y=eta5] {datanumerics/mybsp1-P1FEM-nvb3eta1T2uni.dat};
\addplot [color=black, mark=o, line width=0.5pt, mark size=3pt, mark options={solid, black}]                                   table [x=nrelements,y=errorH1semiosc] {datanumerics/mybsp1-P1FEM-nvb3eta1T2uni.dat};
\addplot [color=black!50!green, mark=triangle, line width=0.5pt, mark size=3pt, mark options={solid, black!50!green}]          table [x=nrelements,y=eta1] {datanumerics/mybsp1-P1FEM-nvb3eta1T2ada.dat};
\addplot [color=black!50!blue, mark=triangle, line width=0.5pt, mark size=3pt, mark options={solid, rotate=90, black!50!blue}] table [x=nrelements,y=eta2] {datanumerics/mybsp1-P1FEM-nvb3eta1T2ada.dat};
\addplot [color=black!50!red, mark=triangle, line width=0.5pt, mark size=3pt, mark options={solid, rotate=180, black!50!red}]  table [x=nrelements,y=eta4] {datanumerics/mybsp1-P1FEM-nvb3eta1T2ada.dat};
\addplot [color=blue, mark=triangle, line width=0.5pt, mark size=3pt, mark options={solid, rotate=270, blue}]                  table [x=nrelements,y=eta5] {datanumerics/mybsp1-P1FEM-nvb3eta1T2ada.dat};
\addplot [color=red, mark=square, line width=0.5pt, mark size=3pt, mark options={solid, red}]                                  table [x=nrelements,y=errorH1semiosc] {datanumerics/mybsp1-P1FEM-nvb3eta1T2ada.dat};
%
\logLogSlopeTriangle{0.8}{0.2}{0.72}{-1/2}{black}{\scriptsize};
\logLogSlopeTrianglelow{0.8}{0.2}{0.45}{-1/2}{black}{\scriptsize};

\legend{
$\lambda_\bullet'$ (unif.),
$\lambda_\bullet''$ (unif.),
$\mu_\bullet'$ (unif.),
$\mu_\bullet''$ (unif.),
$\big(\norm{\nabla(u - u_\bullet)}{\Omega}^2+{{\rm osc}_\bullet^2}\big)^{1/2}$ (unif.),
$\lambda_\bullet'$ (adap.),
$\lambda_\bullet''$ (adap.),
$\mu_\bullet'$ (adap.),
$\mu_\bullet''$ (adap.),
$\big(\norm{\nabla(u - u_\bullet)}{\Omega}^2+{{\rm osc}_\bullet^2}\big)^{1/2}$ (adap.),
}
\end{loglogaxis}
\end{tikzpicture}
}
\hspace{0.015\textwidth}
\subfigure[${\mathcal{S}}^2$-FEM with (M3) and $\eta_\bullet(T,\widehat u_\bullet)^2 :=\lambda_\bullet(T,\widehat u_\bullet)^2+{\rm apx}_\bullet(T)^2$.]{\label{subfig:mybsp1errornvb3P2}
\begin{tikzpicture}[scale=0.5]
\begin{loglogaxis}[width=0.92\textwidth,
xlabel={\small \# elements}, ylabel={\small error and error estimator}, font={\scriptsize}, 
ymin=5e-6,ymax=2e1,
legend style={font=\small, draw=none, fill=none, cells={anchor=west}, legend pos=south west}]
\addplot [color=red!50!green, mark=x, line width=0.5pt, mark size=3pt, mark options={solid, red!50!green}]                    table [x=nrelements,y=eta1] {datanumerics/mybsp1-P2FEM-nvb3eta3T2uni.dat};
\addplot [color=gray, mark=asterisk, line width=0.5pt, mark size=3pt, mark options={solid, gray}]                             table [x=nrelements,y=eta2] {datanumerics/mybsp1-P2FEM-nvb3eta3T2uni.dat};
\addplot [color=black!25!red, mark=otimes, line width=0.5pt, mark size=2.5pt, mark options={solid, black!25!red}]             table [x=nrelements,y=eta3] {datanumerics/mybsp1-P2FEM-nvb3eta3T2uni.dat};
\addplot [color=black!50!red, mark=+, line width=0.5pt, mark size=3pt, mark options={solid, black!50!red}]                    table [x=nrelements,y=eta4] {datanumerics/mybsp1-P2FEM-nvb3eta3T2uni.dat};
\addplot [color=violet, mark=diamond, line width=0.5pt, mark size=3pt, mark options={solid, violet}]                          table [x=nrelements,y=eta5] {datanumerics/mybsp1-P2FEM-nvb3eta3T2uni.dat};
\addplot [color=black!25!yellow, mark=Mercedes star, line width=0.5pt, mark size=3pt, mark options={solid, black!25!yellow}]  table [x=nrelements,y=eta6] {datanumerics/mybsp1-P2FEM-nvb3eta3T2uni.dat};
\addplot [color=black, mark=o, line width=0.5pt, mark size=3pt, mark options={solid, black}]                                  table [x=nrelements,y=errorH1semiosc] {datanumerics/mybsp1-P2FEM-nvb3eta3T2uni.dat};
\addplot [color=black!50!green, mark=triangle, line width=0.5pt, mark size=3pt, mark options={solid, black!50!green}]         table [x=nrelements,y=eta1] {datanumerics/mybsp1-P2FEM-nvb3eta3T2ada.dat};
\addplot [color=black!50!blue, mark=triangle, line width=0.5pt, mark size=3pt, mark options={solid, rotate=90, black!50!blue}]table [x=nrelements,y=eta2] {datanumerics/mybsp1-P2FEM-nvb3eta3T2ada.dat};
\addplot [color=black!75!green, mark=pentagon, line width=0.5pt, mark size=2.5pt, mark options={solid, black!75!green}]       table [x=nrelements,y=eta3] {datanumerics/mybsp1-P2FEM-nvb3eta3T2ada.dat};
\addplot [color=black!50!red, mark=triangle, line width=0.5pt, mark size=3pt, mark options={solid, rotate=180, black!50!red}] table [x=nrelements,y=eta4] {datanumerics/mybsp1-P2FEM-nvb3eta3T2ada.dat};
\addplot [color=blue, mark=triangle, line width=0.5pt, mark size=3pt, mark options={solid, rotate=270, blue}]                 table [x=nrelements,y=eta5] {datanumerics/mybsp1-P2FEM-nvb3eta3T2ada.dat};
\addplot [color=brown, mark=oplus, line width=0.5pt, mark size=2.5pt, mark options={solid, brown}]                            table [x=nrelements,y=eta6] {datanumerics/mybsp1-P2FEM-nvb3eta3T2ada.dat};
\addplot [color=red, mark=square, line width=0.5pt, mark size=3pt, mark options={solid, red}]                                 table [x=nrelements,y=errorH1semiosc] {datanumerics/mybsp1-P2FEM-nvb3eta3T2ada.dat};
\logLogSlopeTriangle{0.8}{0.2}{0.5}{-1}{black}{\scriptsize};
\logLogSlopeTrianglelow{0.8}{0.2}{0.08}{-1}{black}{\scriptsize};

\legend{
$\lambda_\bullet'$ (unif.),
$\lambda_\bullet''$ (unif.),
$\lambda_\bullet'''$ (unif.),
$\mu_\bullet'$ (unif.),
$\mu_\bullet''$ (unif.),
$\mu_\bullet'''$ (unif.),
$\big(\norm{\nabla(u - u_\bullet)}{\Omega}^2+{{\rm osc}_\bullet^2}\big)^{1/2}$ (unif.),
$\lambda_\bullet'$ (adap.),
$\lambda_\bullet''$ (adap.),
$\lambda_\bullet'''$ (adap.),
$\mu_\bullet'$ (adap.),
$\mu_\bullet''$ (adap.),
$\mu_\bullet'''$ (adap.),
$\big(\norm{\nabla(u - u_\bullet)}{\Omega}^2+{{\rm osc}_\bullet^2}\big)^{1/2}$ (adap.),
}
\end{loglogaxis}
\end{tikzpicture}
}
\caption{Experiment from Section~\ref{ex1} with known smooth solution. We use the (minimal) refinement with (M3)
for both, uniform and adaptive refinement.  
}
\label{fig:mybsp1errornvb3}
\end{figure}
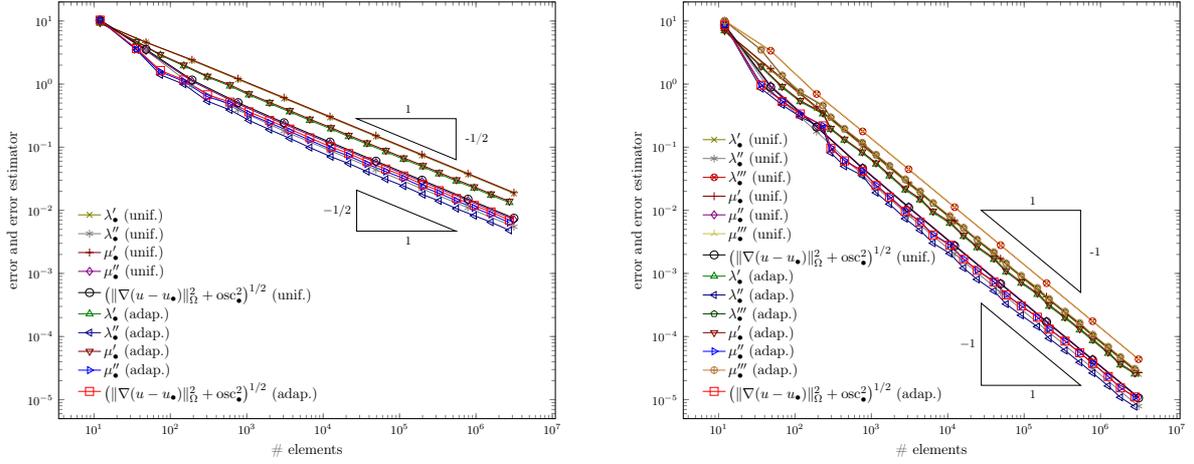
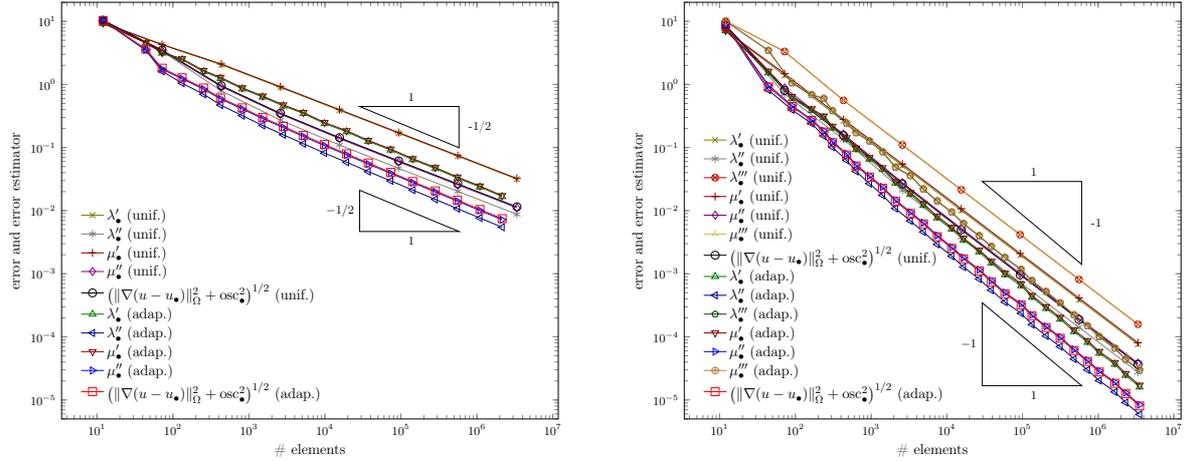

\begin{figure}
\centering
\subfigure[${\mathcal{S}}^1$-FEM with (M3') and $\eta_\bullet(T,\widehat u_\bullet)^2 :=\lambda_\bullet(T,\widehat u_\bullet)^2+{\rm osc}_\bullet(T)^2$.]{\label{subfig:mybsp1errornvb5P1}
\begin{tikzpicture}[scale=0.5]
\begin{loglogaxis}[width=0.92\textwidth,
xlabel={\small \# elements}, ylabel={\small error and error estimator}, font={\scriptsize}, 
ymin=5e-6,ymax=2e1,
legend style={font=\small, draw=none, fill=none, cells={anchor=west}, legend pos=south west}]
\addplot [color=red!50!green, mark=x, line width=0.5pt, mark size=3pt, mark options={solid, red!50!green}]                     table [x=nrelements,y=eta1] {datanumerics/mybsp1-P1FEM-nvb5eta2T2uni.dat};
\addplot [color=gray, mark=asterisk, line width=0.5pt, mark size=3pt, mark options={solid, gray}]                              table [x=nrelements,y=eta2] {datanumerics/mybsp1-P1FEM-nvb5eta2T2uni.dat};
\addplot [color=black!50!red, mark=+, line width=0.5pt, mark size=3pt, mark options={solid, black!50!red}]                     table [x=nrelements,y=eta4] {datanumerics/mybsp1-P1FEM-nvb5eta2T2uni.dat};
\addplot [color=violet, mark=diamond, line width=0.5pt, mark size=3pt, mark options={solid, violet}]                           table [x=nrelements,y=eta5] {datanumerics/mybsp1-P1FEM-nvb5eta2T2uni.dat};
\addplot [color=black, mark=o, line width=0.5pt, mark size=3pt, mark options={solid, black}]                                   table [x=nrelements,y=errorH1semiosc] {datanumerics/mybsp1-P1FEM-nvb5eta2T2uni.dat};
\addplot [color=black!50!green, mark=triangle, line width=0.5pt, mark size=3pt, mark options={solid, black!50!green}]          table [x=nrelements,y=eta1] {datanumerics/mybsp1-P1FEM-nvb5eta2T2ada.dat};
\addplot [color=black!50!blue, mark=triangle, line width=0.5pt, mark size=3pt, mark options={solid, rotate=90, black!50!blue}] table [x=nrelements,y=eta2] {datanumerics/mybsp1-P1FEM-nvb5eta2T2ada.dat};
\addplot [color=black!50!red, mark=triangle, line width=0.5pt, mark size=3pt, mark options={solid, rotate=180, black!50!red}]  table [x=nrelements,y=eta4] {datanumerics/mybsp1-P1FEM-nvb5eta2T2ada.dat};
\addplot [color=blue, mark=triangle, line width=0.5pt, mark size=3pt, mark options={solid, rotate=270, blue}]                  table [x=nrelements,y=eta5] {datanumerics/mybsp1-P1FEM-nvb5eta2T2ada.dat};
\addplot [color=red, mark=square, line width=0.5pt, mark size=3pt, mark options={solid, red}]                                  table [x=nrelements,y=errorH1semiosc] {datanumerics/mybsp1-P1FEM-nvb5eta2T2ada.dat};
%
\logLogSlopeTriangle{0.8}{0.2}{0.75}{-1/2}{black}{\scriptsize};
\logLogSlopeTrianglelow{0.8}{0.2}{0.45}{-1/2}{black}{\scriptsize};

\legend{
$\lambda_\bullet'$ (unif.),
$\lambda_\bullet''$ (unif.),
$\mu_\bullet'$ (unif.),
$\mu_\bullet''$ (unif.),
$\big(\norm{\nabla(u - u_\bullet)}{\Omega}^2+{{\rm osc}_\bullet^2}\big)^{1/2}$ (unif.),
$\lambda_\bullet'$ (adap.),
$\lambda_\bullet''$ (adap.),
$\mu_\bullet'$ (adap.),
$\mu_\bullet''$ (adap.),
$\big(\norm{\nabla(u - u_\bullet)}{\Omega}^2+{{\rm osc}_\bullet^2}\big)^{1/2}$ (adap.),
}
\end{loglogaxis}
\end{tikzpicture}
}
\hspace{0.015\textwidth}
\subfigure[${\mathcal{S}}^2$-FEM with (M3') and $\eta_\bullet(T,\widehat u_\bullet)^2 :=\lambda_\bullet(T,\widehat u_\bullet)^2+{\rm osc}_\bullet(T)^2$.]{\label{subfig:mybsp1errornvb5P2}
\begin{tikzpicture}[scale=0.5]
\begin{loglogaxis}[width=0.92\textwidth,
xlabel={\small \# elements}, ylabel={\small error and error estimator}, font={\scriptsize}, 
ymin=5e-6,ymax=2e1,
legend style={font=\small, draw=none, fill=none, cells={anchor=west}, legend pos=south west}]
\addplot [color=red!50!green, mark=x, line width=0.5pt, mark size=3pt, mark options={solid, red!50!green}]                    table [x=nrelements,y=eta1] {datanumerics/mybsp1-P2FEM-nvb5eta2T2uni.dat};
\addplot [color=gray, mark=asterisk, line width=0.5pt, mark size=3pt, mark options={solid, gray}]                             table [x=nrelements,y=eta2] {datanumerics/mybsp1-P2FEM-nvb5eta2T2uni.dat};
\addplot [color=black!25!red, mark=otimes, line width=0.5pt, mark size=2.5pt, mark options={solid, black!25!red}]             table [x=nrelements,y=eta3] {datanumerics/mybsp1-P2FEM-nvb5eta2T2uni.dat};
\addplot [color=black!50!red, mark=+, line width=0.5pt, mark size=3pt, mark options={solid, black!50!red}]                    table [x=nrelements,y=eta4] {datanumerics/mybsp1-P2FEM-nvb5eta2T2uni.dat};
\addplot [color=violet, mark=diamond, line width=0.5pt, mark size=3pt, mark options={solid, violet}]                          table [x=nrelements,y=eta5] {datanumerics/mybsp1-P2FEM-nvb5eta2T2uni.dat};
\addplot [color=black!25!yellow, mark=Mercedes star, line width=0.5pt, mark size=3pt, mark options={solid, black!25!yellow}]  table [x=nrelements,y=eta6] {datanumerics/mybsp1-P2FEM-nvb5eta2T2uni.dat};
\addplot [color=black, mark=o, line width=0.5pt, mark size=3pt, mark options={solid, black}]                                  table [x=nrelements,y=errorH1semiosc] {datanumerics/mybsp1-P2FEM-nvb5eta2T2uni.dat};
\addplot [color=black!50!green, mark=triangle, line width=0.5pt, mark size=3pt, mark options={solid, black!50!green}]         table [x=nrelements,y=eta1] {datanumerics/mybsp1-P2FEM-nvb5eta2T2ada.dat};
\addplot [color=black!50!blue, mark=triangle, line width=0.5pt, mark size=3pt, mark options={solid, rotate=90, black!50!blue}]table [x=nrelements,y=eta2] {datanumerics/mybsp1-P2FEM-nvb5eta2T2ada.dat};
\addplot [color=black!75!green, mark=pentagon, line width=0.5pt, mark size=2.5pt, mark options={solid, black!75!green}]       table [x=nrelements,y=eta3] {datanumerics/mybsp1-P2FEM-nvb5eta2T2ada.dat};
\addplot [color=black!50!red, mark=triangle, line width=0.5pt, mark size=3pt, mark options={solid, rotate=180, black!50!red}] table [x=nrelements,y=eta4] {datanumerics/mybsp1-P2FEM-nvb5eta2T2ada.dat};
\addplot [color=blue, mark=triangle, line width=0.5pt, mark size=3pt, mark options={solid, rotate=270, blue}]                 table [x=nrelements,y=eta5] {datanumerics/mybsp1-P2FEM-nvb5eta2T2ada.dat};
\addplot [color=brown, mark=oplus, line width=0.5pt, mark size=2.5pt, mark options={solid, brown}]                            table [x=nrelements,y=eta6] {datanumerics/mybsp1-P2FEM-nvb5eta2T2ada.dat};
\addplot [color=red, mark=square, line width=0.5pt, mark size=3pt, mark options={solid, red}]                                 table [x=nrelements,y=errorH1semiosc] {datanumerics/mybsp1-P2FEM-nvb5eta2T2ada.dat};
\logLogSlopeTriangle{0.8}{0.2}{0.57}{-1}{black}{\scriptsize};
\logLogSlopeTrianglelow{0.8}{0.2}{0.08}{-1}{black}{\scriptsize};

\legend{
$\lambda_\bullet'$ (unif.),
$\lambda_\bullet''$ (unif.),
$\lambda_\bullet'''$ (unif.),
$\mu_\bullet'$ (unif.),
$\mu_\bullet''$ (unif.),
$\mu_\bullet'''$ (unif.),
$\big(\norm{\nabla(u - u_\bullet)}{\Omega}^2+{{\rm osc}_\bullet^2}\big)^{1/2}$ (unif.),
$\lambda_\bullet'$ (adap.),
$\lambda_\bullet''$ (adap.),
$\lambda_\bullet'''$ (adap.),
$\mu_\bullet'$ (adap.),
$\mu_\bullet''$ (adap.),
$\mu_\bullet'''$ (adap.),
$\big(\norm{\nabla(u - u_\bullet)}{\Omega}^2+{{\rm osc}_\bullet^2}\big)^{1/2}$ (adap.),
}
\end{loglogaxis}
\end{tikzpicture}
}
\caption{Experiment from Section~\ref{ex1} with known smooth solution. We use the (minimal) refinement with (M3')
for both, uniform and adaptive refinement.
}
\label{fig:mybsp1errornvb5}
\end{figure}

%
\begin{figure}
\centering
\subfigure[${\mathcal{S}}^1$-FEM.]{\label{subfig:effectivityreliabilityindexoscS1}
\begin{tikzpicture}[scale=0.5]
\begin{axis}[width=0.92\textwidth,
xmode=log, xlabel={\small \# elements}, ylabel={\small efficiency index (lower half), reliability index (upper half)}, font={\scriptsize},
ymin=0.5,ymax=1.6, legend columns=2,
legend style={font=\small, draw=none, fill=none, cells={anchor=west}, legend pos=north west,
/tikz/column 2/.style={column sep=20pt},at={(axis cs:4000,1.5)},anchor=north west}]
\addplot [color=red!50!green, mark=x, line width=0.5pt, mark size=3pt, mark options={solid, red!50!green}]          table [x=nrelements,y=reliabilityindexeta5] {datanumerics/mybsp1-P1FEM-nvb3eta1T2ada.dat};
\addplot[color=gray, mark=asterisk, line width=0.5pt, mark size=3pt, mark options={solid, gray}]                    table [x=nrelements,y=reliabilityindexeta5] {datanumerics/mybsp1-P1FEM-nvb5eta2T2ada.dat};
\addplot  [color=black!50!red, mark=+, line width=0.5pt, mark size=3pt, mark options={solid, black!50!red}]         table [x=nrelements,y=reliabilityindexeta5] {datanumerics/mybsp2-P1FEM-nvb3eta1T2ada.dat};
\addplot [color=violet, mark=diamond, line width=0.5pt, mark size=3pt, mark options={solid, violet}]                table [x=nrelements,y=reliabilityindexeta5] {datanumerics/mybsp2-P1FEM-nvb5eta2T2ada.dat};
\addplot [color=black, line width=0.5pt, style=dashed]                 coordinates{(12,1.1547) (4*10^6,1.1547)};
\addplot [color=black, line width=0.5pt, style=dashed]                 coordinates{(12,1.0954) (4*10^6,1.0954)};
\node[above] at (axis cs:10^5,1.0857) {\small(M3')};
\node[above] at (axis cs:10^5,1.145) {\small(M3)};
\addplot  [color=red!50!green, mark=x, line width=0.5pt, mark size=3pt, mark options={solid, red!50!green}]         table [x=nrelements,y=effectivityindexeta2] {datanumerics/mybsp1-P1FEM-nvb3eta1T2ada.dat};
\addplot[color=gray, mark=asterisk, line width=0.5pt, mark size=3pt, mark options={solid, gray}]                    table [x=nrelements,y=effectivityindexeta2] {datanumerics/mybsp1-P1FEM-nvb5eta2T2ada.dat};
\addplot [color=black!50!red, mark=+, line width=0.5pt, mark size=3pt, mark options={solid, black!50!red}]          table [x=nrelements,y=effectivityindexeta2] {datanumerics/mybsp2-P1FEM-nvb3eta1T2ada.dat};
\addplot [color=violet, mark=diamond, line width=0.5pt, mark size=3pt, mark options={solid, violet}]                table [x=nrelements,y=effectivityindexeta2] {datanumerics/mybsp2-P1FEM-nvb5eta2T2ada.dat};

\legend{
Ex.~1 (M3),
Ex.~1 (M3'),
Ex.~2 (M3),
Ex.~2 (M3'),
}
\end{axis}
\end{tikzpicture}
}
\hspace{0.015\textwidth}
\subfigure[${\mathcal{S}}^2$-FEM]{\label{subfig:effectivityreliabilityindexoscS2}
\begin{tikzpicture}[scale=0.5]
\begin{axis}[width=0.92\textwidth,
xmode=log, xlabel={\small \# elements}, ylabel={\small efficiency index (lower half), reliability index (upper half)}, font={\scriptsize},
ymin=0.5,ymax=1.6, legend columns=2,
legend style={font=\small, draw=none, fill=none, cells={anchor=west}, legend pos=south east,
/tikz/column 2/.style={column sep=20pt},at={(axis cs:4000,1.5)},anchor=north west}]
\addplot  [color=red!50!green, mark=x, line width=0.5pt, mark size=3pt, mark options={solid, red!50!green}]         table [x=nrelements,y=reliabilityindexeta5] {datanumerics/mybsp1-P2FEM-nvb3eta3T2ada.dat};
\addplot[color=gray, mark=asterisk, line width=0.5pt, mark size=3pt, mark options={solid, gray}]                    table [x=nrelements,y=reliabilityindexeta5] {datanumerics/mybsp1-P2FEM-nvb5eta2T2ada.dat};
\addplot [color=black!50!red, mark=+, line width=0.5pt, mark size=3pt, mark options={solid, black!50!red}]          table [x=nrelements,y=reliabilityindexeta5] {datanumerics/mybsp2-P2FEM-nvb3eta3T2ada.dat};
\addplot[color=violet, mark=diamond, line width=0.5pt, mark size=3pt, mark options={solid, violet}]                 table [x=nrelements,y=reliabilityindexeta5] {datanumerics/mybsp2-P2FEM-nvb5eta2T2ada.dat};
\addplot [color=black, line width=0.5pt, style=dashed]                 coordinates{(12,1.0328) (4*10^6,1.0328)};
\addplot [color=black, line width=0.5pt, style=dashed]                 coordinates{(12,1.0142) (4*10^6,1.0142)};
\node[below] at (axis cs:10^1.55,1.02) {\small(M3')};
\node[above] at (axis cs:10^1.55,1.027) {\small(M3)};
\addplot[color=red!50!green, mark=x, line width=0.5pt, mark size=3pt, mark options={solid, red!50!green}]           table [x=nrelements,y=effectivityindexeta2] {datanumerics/mybsp1-P2FEM-nvb3eta3T2ada.dat};
\addplot [color=gray, mark=asterisk, line width=0.5pt, mark size=3pt, mark options={solid, gray}]                   table [x=nrelements,y=effectivityindexeta2] {datanumerics/mybsp1-P2FEM-nvb5eta2T2ada.dat};
\addplot[color=black!50!red, mark=+, line width=0.5pt, mark size=3pt, mark options={solid, black!50!red}]           table [x=nrelements,y=effectivityindexeta2] {datanumerics/mybsp2-P2FEM-nvb3eta3T2ada.dat};
\addplot [color=violet, mark=diamond, line width=0.5pt, mark size=3pt, mark options={solid, violet}]                table [x=nrelements,y=effectivityindexeta2] {datanumerics/mybsp2-P2FEM-nvb5eta2T2ada.dat};

\legend{
Ex.~1 (M3),
Ex.~1 (M3'),
Ex.~2 (M3),
Ex.~2 (M3'),
}
\end{axis}
\end{tikzpicture}
}
\caption{Reliability index $(\norm{\nabla(u - u_\bullet)}{\Omega}^2+{{\rm osc}_\bullet^2})^{1/2} / \mu_\bullet''$ (upper half)
and efficiency index $\lambda_\bullet'' / (\norm{\nabla(u - u_\bullet)}{\Omega}^2+{{\rm osc}_\bullet^2})^{1/2}$ (lower half)
for the example (Ex.~1) in Section~\ref{ex1} and for the example (Ex.~2) in Section~\ref{ex2},
and for adaptive refinement with (M3) and (M3') with the corresponding refinement indicator $\eta_\bullet(T,\widehat u_\bullet)$ defined in~\eqref{eq:refinementindicator}.
The (asymptotic) upper bounds for the reliability indices predicted in Remark~\ref{rem:estimator} are highlighted with dashed black  lines.
}
\label{fig:effectivityreliabilityindexosc}
\end{figure}

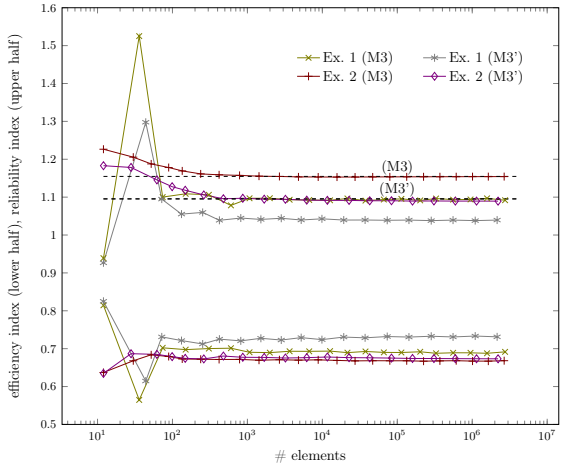
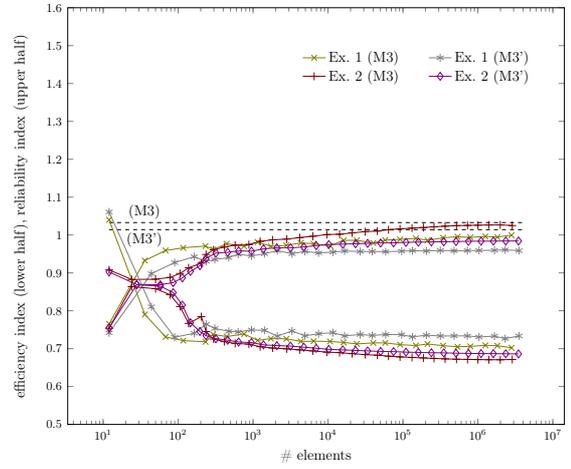
\begin{figure}
\centering
\subfigure[${\mathcal{S}}^1$-FEM.]{\label{subfig:effectivityreliabilityindexS1}
\begin{tikzpicture}[scale=0.5]
\begin{axis}[width=0.92\textwidth,
xmode=log, xlabel={\small \# elements}, ylabel={\small efficiency index (lower half), reliability index (upper half)}, font={\scriptsize},
ymin=0.5,ymax=1.6, legend columns=2,
legend style={font=\small, draw=none, fill=none, cells={anchor=west}, legend pos=north east,
/tikz/column 2/.style={column sep=20pt},at={(axis cs:4000,1.5)},anchor=north west}]
\addplot [color=red!50!green, mark=x, line width=0.5pt, mark size=3pt, mark options={solid, red!50!green}]          table [x=nrelements,y=reliabilityindexmu] {datanumerics/mybsp1-P1FEM-nvb3eta1T2ada.dat};
\addplot [color=gray, mark=asterisk, line width=0.5pt, mark size=3pt, mark options={solid, gray}]                   table [x=nrelements,y=reliabilityindexmu] {datanumerics/mybsp1-P1FEM-nvb5eta2T2ada.dat};
\addplot [color=black!50!red, mark=+, line width=0.5pt, mark size=3pt, mark options={solid, black!50!red}]          table [x=nrelements,y=reliabilityindexmu] {datanumerics/mybsp2-P1FEM-nvb3eta1T2ada.dat};
\addplot [color=violet, mark=diamond, line width=0.5pt, mark size=3pt, mark options={solid, violet}]                table [x=nrelements,y=reliabilityindexmu] {datanumerics/mybsp2-P1FEM-nvb5eta2T2ada.dat};
\addplot [color=black, line width=0.5pt, style=dashed]                 coordinates{(12,1.1547) (4*10^6,1.1547)};
\addplot [color=black, line width=0.5pt, style=dashed]                 coordinates{(12,1.0954) (4*10^6,1.0954)};
\addplot [color=black, line width=0.5pt, style=dashed]                 coordinates{(12,1.0954) (4*10^6,1.0954)};
\node[above] at (axis cs:10^5,1.0857) {\small(M3')};
\node[above] at (axis cs:10^5,1.145) {\small(M3)};
\addplot [color=red!50!green, mark=x, line width=0.5pt, mark size=3pt, mark options={solid, red!50!green}]          table [x=nrelements,y=effectivityindexlambda] {datanumerics/mybsp1-P1FEM-nvb3eta1T2ada.dat};
\addplot  [color=gray, mark=asterisk, line width=0.5pt, mark size=3pt, mark options={solid, gray}]                  table [x=nrelements,y=effectivityindexlambda] {datanumerics/mybsp1-P1FEM-nvb5eta2T2ada.dat};
\addplot [color=black!50!red, mark=+, line width=0.5pt, mark size=3pt, mark options={solid, black!50!red}]          table [x=nrelements,y=effectivityindexlambda] {datanumerics/mybsp2-P1FEM-nvb3eta1T2ada.dat};
\addplot [color=violet, mark=diamond, line width=0.5pt, mark size=3pt, mark options={solid, violet}]                table [x=nrelements,y=effectivityindexlambda] {datanumerics/mybsp2-P1FEM-nvb5eta2T2ada.dat};

\legend{
Ex.~1 (M3),
Ex.~1 (M3'),
Ex.~2 (M3),
Ex.~2 (M3'),
}
\end{axis}
\end{tikzpicture}
}
\hspace{0.015\textwidth}
\subfigure[${\mathcal{S}}^2$-FEM]{\label{subfig:effectivityreliabilityindexS2}
\begin{tikzpicture}[scale=0.5]
\begin{axis}[width=0.92\textwidth,
xmode=log, xlabel={\small \# elements}, ylabel={\small efficiency index (lower half), reliability index (upper half)}, font={\scriptsize},
ymin=0.5,ymax=1.6, legend columns=2,
legend style={font=\small, draw=none, fill=none, cells={anchor=west}, legend pos=south east,
/tikz/column 2/.style={column sep=20pt},at={(axis cs:4000,1.5)},anchor=north west}]
\addplot [color=red!50!green, mark=x, line width=0.5pt, mark size=3pt, mark options={solid, red!50!green}]            table [x=nrelements,y=reliabilityindexmu] {datanumerics/mybsp1-P2FEM-nvb3eta3T2ada.dat};
\addplot [color=gray, mark=asterisk, line width=0.5pt, mark size=3pt, mark options={solid, gray}]                     table [x=nrelements,y=reliabilityindexmu] {datanumerics/mybsp1-P2FEM-nvb5eta2T2ada.dat};
\addplot [color=black!50!red, mark=+, line width=0.5pt, mark size=3pt, mark options={solid, black!50!red}]            table [x=nrelements,y=reliabilityindexmu] {datanumerics/mybsp2-P2FEM-nvb3eta3T2ada.dat};
\addplot [color=violet, mark=diamond, line width=0.5pt, mark size=3pt, mark options={solid, violet}]                  table [x=nrelements,y=reliabilityindexmu] {datanumerics/mybsp2-P2FEM-nvb5eta2T2ada.dat};
\addplot [color=black, line width=0.5pt, style=dashed]                 coordinates{(12,1.0328) (4*10^6,1.0328)};
\addplot [color=black, line width=0.5pt, style=dashed]                 coordinates{(12,1.0142) (4*10^6,1.0142)};
\node[below] at (axis cs:10^1.55,1.02) {\small(M3')};
\node[above] at (axis cs:10^1.55,1.027) {\small(M3)};
\addplot [color=red!50!green, mark=x, line width=0.5pt, mark size=3pt, mark options={solid, red!50!green}]            table [x=nrelements,y=effectivityindexlambda] {datanumerics/mybsp1-P2FEM-nvb3eta3T2ada.dat};
\addplot [color=gray, mark=asterisk, line width=0.5pt, mark size=3pt, mark options={solid, gray}]                     table [x=nrelements,y=effectivityindexlambda] {datanumerics/mybsp1-P2FEM-nvb5eta2T2ada.dat};
\addplot  [color=black!50!red, mark=+, line width=0.5pt, mark size=3pt, mark options={solid, black!50!red}]           table [x=nrelements,y=effectivityindexlambda] {datanumerics/mybsp2-P2FEM-nvb3eta3T2ada.dat};
\addplot [color=violet, mark=diamond, line width=0.5pt, mark size=3pt, mark options={solid, violet}]                  table [x=nrelements,y=effectivityindexlambda] {datanumerics/mybsp2-P2FEM-nvb5eta2T2ada.dat};

\legend{
Ex.~1 (M3),
Ex.~1 (M3'),
Ex.~2 (M3),
Ex.~2 (M3'),
}
\end{axis}
\end{tikzpicture}
}
\caption{Reliability index $\norm{\nabla(u - u_\bullet)}{\Omega}/\mu_\bullet(\widehat{u}_\bullet)$ (upper half)
and efficiency index $\lambda_\bullet(\widehat{u}_\bullet)/\norm{\nabla(u - u_\bullet)}{\Omega}$ (lower half)
for the example (Ex.~1) in Section~\ref{ex1} and for the example (Ex.~2) in Section~\ref{ex2},
and for
adaptive refinement with (M3) and (M3') with the corresponding refinement indicator $\eta_\bullet(T,\widehat u_\bullet)$ defined in~\eqref{eq:refinementindicator}.
The (asymptotic) upper bounds for the reliability indices predicted in Remark~\ref{rem:estimator} are highlighted with dashed black  lines.
}
\label{fig:effectivityreliabilityindex}
\end{figure}
\subsection{Experiment with known smooth solution}
\label{ex1}
We prescribe the exact solution
\begin{align*}
u(x_1,x_2) = (1-10x_1^2-10x_2^2)e^{-5(x_1^2+x_2^2)} \quad\text{with } x=(x_1,x_2)\in\mathbb{R}^2.
\end{align*}
This also defines inhomogeneous Dirichlet
conditions and the right-hand side $f$ is calculated appropriately.
Since $u$ is smooth, uniform as well as adaptive mesh refinement
with (M3) or (M3') lead to optimal convergence behavior of ${\mathcal{O}}(N^{-1/2})$
and ${\mathcal{O}}(N^{-1})$
for ${\mathcal{S}}^1$-FEM and ${\mathcal{S}}^2$-FEM, respectively;
see Figure~\ref{fig:mybsp1errornvb3} for (M3) and Figure~\ref{fig:mybsp1errornvb5} for (M3').
In Figure~\ref{fig:effectivityreliabilityindexosc} and
Figure~\ref{fig:effectivityreliabilityindex}, we consider corresponding reliability and efficiency indices for adaptive refinement, which empirically confirm Remark~\ref{rem:estimator} (i) and (ii) (for inhomogeneous Dirichlet conditions).
Note that $C_{\rm son}=4$ for (M3) and $C_{\rm son}=6$ for (M3').

\begin{figure}
\centering
\subfigure[${\mathcal{S}}^1$-FEM with (M3) and $\eta_\bullet(T,\widehat u_\bullet)^2 :=\lambda_\bullet(T,\widehat u_\bullet)^2+{\rm res}_\bullet(T,\widehat u_\bullet)^2$.]{\label{subfig:mybsp2errornvb3P1}
\begin{tikzpicture}[scale=0.5]
\begin{loglogaxis}[width=0.92\textwidth,
xlabel={\small \# elements}, ylabel={\small error and error estimator}, font={\scriptsize}, 
ymin=1e-7,ymax=1e0,
legend style={font=\small, draw=none, fill=none, cells={anchor=west}, legend pos=south west}]
\addplot [color=red!50!green, mark=x, line width=0.5pt, mark size=3pt, mark options={solid, red!50!green}]                     table [x=nrelements,y=eta1] {datanumerics/mybsp2-P1FEM-nvb3eta1T2uni.dat};
\addplot [color=gray, mark=asterisk, line width=0.5pt, mark size=3pt, mark options={solid, gray}]                              table [x=nrelements,y=eta2] {datanumerics/mybsp2-P1FEM-nvb3eta1T2uni.dat};
\addplot [color=black!50!red, mark=+, line width=0.5pt, mark size=3pt, mark options={solid, black!50!red}]                     table [x=nrelements,y=eta4] {datanumerics/mybsp2-P1FEM-nvb3eta1T2uni.dat};
\addplot [color=violet, mark=diamond, line width=0.5pt, mark size=3pt, mark options={solid, violet}]                           table [x=nrelements,y=eta5] {datanumerics/mybsp2-P1FEM-nvb3eta1T2uni.dat};
\addplot [color=black, mark=o, line width=0.5pt, mark size=3pt, mark options={solid, black}]                                   table [x=nrelements,y=errorH1semiosc] {datanumerics/mybsp2-P1FEM-nvb3eta1T2uni.dat};
\addplot [color=black!50!green, mark=triangle, line width=0.5pt, mark size=3pt, mark options={solid, black!50!green}]          table [x=nrelements,y=eta1] {datanumerics/mybsp2-P1FEM-nvb3eta1T2ada.dat};
\addplot [color=black!50!blue, mark=triangle, line width=0.5pt, mark size=3pt, mark options={solid, rotate=90, black!50!blue}] table [x=nrelements,y=eta2] {datanumerics/mybsp2-P1FEM-nvb3eta1T2ada.dat};
\addplot [color=black!50!red, mark=triangle, line width=0.5pt, mark size=3pt, mark options={solid, rotate=180, black!50!red}]  table [x=nrelements,y=eta4] {datanumerics/mybsp2-P1FEM-nvb3eta1T2ada.dat};
\addplot [color=blue, mark=triangle, line width=0.5pt, mark size=3pt, mark options={solid, rotate=270, blue}]                  table [x=nrelements,y=eta5] {datanumerics/mybsp2-P1FEM-nvb3eta1T2ada.dat};
\addplot [color=red, mark=square, line width=0.5pt, mark size=3pt, mark options={solid, red}]                                  table [x=nrelements,y=errorH1semiosc] {datanumerics/mybsp2-P1FEM-nvb3eta1T2ada.dat};
%
\logLogSlopeTriangle{0.8}{0.2}{0.82}{-1/3}{black}{\scriptsize};
\logLogSlopeTrianglelow{0.8}{0.2}{0.525}{-1/2}{black}{\scriptsize};

\legend{
$\lambda_\bullet'$ (unif.),
$\lambda_\bullet''$ (unif.),
$\mu_\bullet'$ (unif.),
$\mu_\bullet''$ (unif.),
$\big(\norm{\nabla(u - u_\bullet)}{\Omega}^2+{{\rm osc}_\bullet^2}\big)^{1/2}$ (unif.),
$\lambda_\bullet'$ (adap.),
$\lambda_\bullet''$ (adap.),
$\mu_\bullet'$ (adap.),
$\mu_\bullet''$ (adap.),
$\big(\norm{\nabla(u - u_\bullet)}{\Omega}^2+{{\rm osc}_\bullet^2}\big)^{1/2}$ (adap.),
}
\end{loglogaxis}
\end{tikzpicture}
}
\hspace{0.015\textwidth}
\subfigure[${\mathcal{S}}^2$-FEM with (M3) and $\eta_\bullet(T,\widehat u_\bullet)^2 :=\lambda_\bullet(T,\widehat u_\bullet)^2+{\rm apx}_\bullet(T)^2$.]{\label{subfig:mybsp2errornvb3P2}
\begin{tikzpicture}[scale=0.5]
\begin{loglogaxis}[width=0.92\textwidth,
xlabel={\small \# elements}, ylabel={\small error and error estimator}, font={\scriptsize}, 
ymin=1e-7,ymax=1e0,
legend style={font=\small, draw=none, fill=none, cells={anchor=west}, legend pos=south west}]
\addplot [color=red!50!green, mark=x, line width=0.5pt, mark size=3pt, mark options={solid, red!50!green}]                    table [x=nrelements,y=eta1] {datanumerics/mybsp2-P2FEM-nvb3eta3T2uni.dat};
\addplot [color=gray, mark=asterisk, line width=0.5pt, mark size=3pt, mark options={solid, gray}]                             table [x=nrelements,y=eta2] {datanumerics/mybsp2-P2FEM-nvb3eta3T2uni.dat};
\addplot [color=black!25!red, mark=otimes, line width=0.5pt, mark size=2.5pt, mark options={solid, black!25!red}]             table [x=nrelements,y=eta3] {datanumerics/mybsp2-P2FEM-nvb3eta3T2uni.dat};
\addplot [color=black!50!red, mark=+, line width=0.5pt, mark size=3pt, mark options={solid, black!50!red}]                    table [x=nrelements,y=eta4] {datanumerics/mybsp2-P2FEM-nvb3eta3T2uni.dat};
\addplot [color=violet, mark=diamond, line width=0.5pt, mark size=3pt, mark options={solid, violet}]                          table [x=nrelements,y=eta5] {datanumerics/mybsp2-P2FEM-nvb3eta3T2uni.dat};
\addplot [color=black!25!yellow, mark=Mercedes star, line width=0.5pt, mark size=3pt, mark options={solid, black!25!yellow}]  table [x=nrelements,y=eta6] {datanumerics/mybsp2-P2FEM-nvb3eta3T2uni.dat};
\addplot [color=black, mark=o, line width=0.5pt, mark size=3pt, mark options={solid, black}]                                  table [x=nrelements,y=errorH1semiosc] {datanumerics/mybsp2-P2FEM-nvb3eta3T2uni.dat};
\addplot [color=black!50!green, mark=triangle, line width=0.5pt, mark size=3pt, mark options={solid, black!50!green}]         table [x=nrelements,y=eta1] {datanumerics/mybsp2-P2FEM-nvb3eta3T2ada.dat};
\addplot [color=black!50!blue, mark=triangle, line width=0.5pt, mark size=3pt, mark options={solid, rotate=90, black!50!blue}]table [x=nrelements,y=eta2] {datanumerics/mybsp2-P2FEM-nvb3eta3T2ada.dat};
\addplot [color=black!75!green, mark=pentagon, line width=0.5pt, mark size=2.5pt, mark options={solid, black!75!green}]       table [x=nrelements,y=eta3] {datanumerics/mybsp2-P2FEM-nvb3eta3T2ada.dat};
\addplot [color=black!50!red, mark=triangle, line width=0.5pt, mark size=3pt, mark options={solid, rotate=180, black!50!red}] table [x=nrelements,y=eta4] {datanumerics/mybsp2-P2FEM-nvb3eta3T2ada.dat};
\addplot [color=blue, mark=triangle, line width=0.5pt, mark size=3pt, mark options={solid, rotate=270, blue}]                 table [x=nrelements,y=eta5] {datanumerics/mybsp2-P2FEM-nvb3eta3T2ada.dat};
\addplot [color=brown, mark=oplus, line width=0.5pt, mark size=2.5pt, mark options={solid, brown}]                            table [x=nrelements,y=eta6] {datanumerics/mybsp2-P2FEM-nvb3eta3T2ada.dat};
\addplot [color=red, mark=square, line width=0.5pt, mark size=3pt, mark options={solid, red}]                                 table [x=nrelements,y=errorH1semiosc] {datanumerics/mybsp2-P2FEM-nvb3eta3T2ada.dat};
\logLogSlopeTriangle{0.8}{0.2}{0.85}{-1/3}{black}{\scriptsize};
\logLogSlopeTrianglelow{0.8}{0.2}{0.15}{-1}{black}{\scriptsize};

\legend{
$\lambda_\bullet'$ (unif.),
$\lambda_\bullet''$ (unif.),
$\lambda_\bullet'''$ (unif.),
$\mu_\bullet'$ (unif.),
$\mu_\bullet''$ (unif.),
$\mu_\bullet'''$ (unif.),
$\big(\norm{\nabla(u - u_\bullet)}{\Omega}^2+{{\rm osc}_\bullet^2}\big)^{1/2}$ (unif.),
$\lambda_\bullet'$ (adap.),
$\lambda_\bullet''$ (adap.),
$\lambda_\bullet'''$ (adap.),
$\mu_\bullet'$ (adap.),
$\mu_\bullet''$ (adap.),
$\mu_\bullet'''$ (adap.),
$\big(\norm{\nabla(u - u_\bullet)}{\Omega}^2+{{\rm osc}_\bullet^2}\big)^{1/2}$ (adap.),
}
\end{loglogaxis}
\end{tikzpicture}
}
\caption{Experiment from Section~\ref{ex2} with known solution with generic singularity.
We use the (minimal) refinement with (M3)
for both, uniform and adaptive refinement.
}
\label{fig:mybsp2errornvb3}
\end{figure}
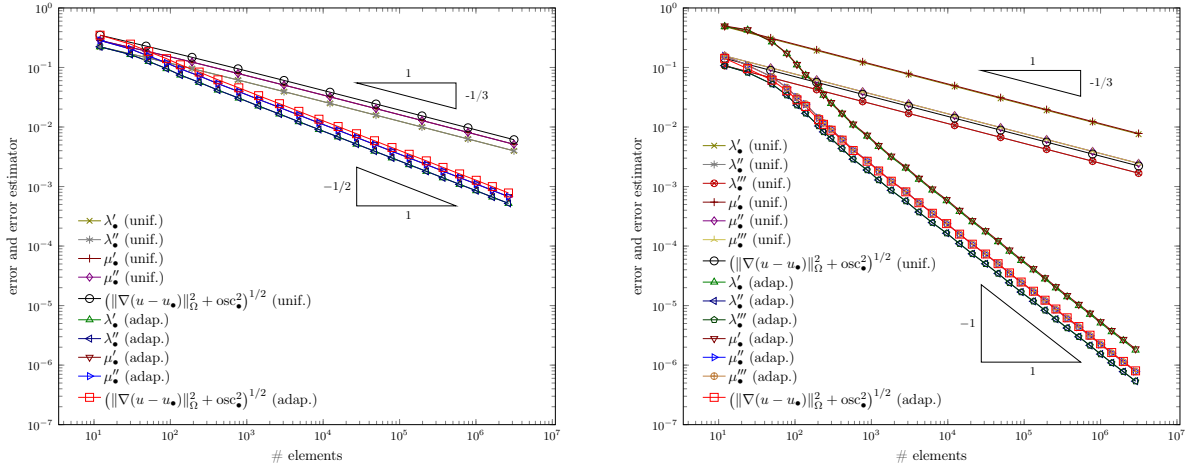

\begin{figure}
\centering
\subfigure[${\mathcal{S}}^1$-FEM with (M3') and $\eta_\bullet(T,\widehat u_\bullet)^2 :=\lambda_\bullet(T,\widehat u_\bullet)^2+{\rm osc}_\bullet(T)^2$.]{\label{subfig:mybsp2errornvb5P1}
\begin{tikzpicture}[scale=0.5]
\begin{loglogaxis}[width=0.92\textwidth,
xlabel={\small \# elements}, ylabel={\small error and error estimator}, font={\scriptsize}, 
ymin=1e-7,ymax=1e0,
legend style={font=\small, draw=none, fill=none, cells={anchor=west}, legend pos=south west}]
\addplot [color=red!50!green, mark=x, line width=0.5pt, mark size=3pt, mark options={solid, red!50!green}]                     table [x=nrelements,y=eta1] {datanumerics/mybsp2-P1FEM-nvb5eta2T2uni.dat};
\addplot [color=gray, mark=asterisk, line width=0.5pt, mark size=3pt, mark options={solid, gray}]                              table [x=nrelements,y=eta2] {datanumerics/mybsp2-P1FEM-nvb5eta2T2uni.dat};
\addplot [color=black!50!red, mark=+, line width=0.5pt, mark size=3pt, mark options={solid, black!50!red}]                     table [x=nrelements,y=eta4] {datanumerics/mybsp2-P1FEM-nvb5eta2T2uni.dat};
\addplot [color=violet, mark=diamond, line width=0.5pt, mark size=3pt, mark options={solid, violet}]                           table [x=nrelements,y=eta5] {datanumerics/mybsp2-P1FEM-nvb5eta2T2uni.dat};
\addplot [color=black, mark=o, line width=0.5pt, mark size=3pt, mark options={solid, black}]                                   table [x=nrelements,y=errorH1semiosc] {datanumerics/mybsp2-P1FEM-nvb5eta2T2uni.dat};
\addplot [color=black!50!green, mark=triangle, line width=0.5pt, mark size=3pt, mark options={solid, black!50!green}]          table [x=nrelements,y=eta1] {datanumerics/mybsp2-P1FEM-nvb5eta2T2ada.dat};
\addplot [color=black!50!blue, mark=triangle, line width=0.5pt, mark size=3pt, mark options={solid, rotate=90, black!50!blue}] table [x=nrelements,y=eta2] {datanumerics/mybsp2-P1FEM-nvb5eta2T2ada.dat};
\addplot [color=black!50!red, mark=triangle, line width=0.5pt, mark size=3pt, mark options={solid, rotate=180, black!50!red}]  table [x=nrelements,y=eta4] {datanumerics/mybsp2-P1FEM-nvb5eta2T2ada.dat};
\addplot [color=blue, mark=triangle, line width=0.5pt, mark size=3pt, mark options={solid, rotate=270, blue}]                  table [x=nrelements,y=eta5] {datanumerics/mybsp2-P1FEM-nvb5eta2T2ada.dat};
\addplot [color=red, mark=square, line width=0.5pt, mark size=3pt, mark options={solid, red}]                                  table [x=nrelements,y=errorH1semiosc] {datanumerics/mybsp2-P1FEM-nvb5eta2T2ada.dat};
%
\logLogSlopeTriangle{0.8}{0.2}{0.85}{-1/3}{black}{\scriptsize};
\logLogSlopeTrianglelow{0.8}{0.2}{0.525}{-1/2}{black}{\scriptsize};

\legend{
$\lambda_\bullet'$ (unif.),
$\lambda_\bullet''$ (unif.),
$\mu_\bullet'$ (unif.),
$\mu_\bullet''$ (unif.),
$\big(\norm{\nabla(u - u_\bullet)}{\Omega}^2+{{\rm osc}_\bullet^2}\big)^{1/2}$ (unif.),
$\lambda_\bullet'$ (adap.),
$\lambda_\bullet''$ (adap.),
$\mu_\bullet'$ (adap.),
$\mu_\bullet''$ (adap.),
$\big(\norm{\nabla(u - u_\bullet)}{\Omega}^2+{{\rm osc}_\bullet^2}\big)^{1/2}$ (adap.),
}
\end{loglogaxis}
\end{tikzpicture}
}
\hspace{0.015\textwidth}
\subfigure[${\mathcal{S}}^2$-FEM with (M3') and $\eta_\bullet(T,\widehat u_\bullet)^2 :=\lambda_\bullet(T,\widehat u_\bullet)^2+{\rm osc}_\bullet(T)^2$.]{\label{subfig:mybsp2errornvb5P2}
\begin{tikzpicture}[scale=0.5]
\begin{loglogaxis}[width=0.92\textwidth,
xlabel={\small \# elements}, ylabel={\small error and error estimator}, font={\scriptsize}, 
ymin=1e-7,ymax=1e0,
legend style={font=\small, draw=none, fill=none, cells={anchor=west}, legend pos=south west}]
\addplot [color=red!50!green, mark=x, line width=0.5pt, mark size=3pt, mark options={solid, red!50!green}]                    table [x=nrelements,y=eta1] {datanumerics/mybsp2-P2FEM-nvb5eta2T2uni.dat};
\addplot [color=gray, mark=asterisk, line width=0.5pt, mark size=3pt, mark options={solid, gray}]                             table [x=nrelements,y=eta2] {datanumerics/mybsp2-P2FEM-nvb5eta2T2uni.dat};
\addplot [color=black!25!red, mark=otimes, line width=0.5pt, mark size=2.5pt, mark options={solid, black!25!red}]             table [x=nrelements,y=eta3] {datanumerics/mybsp2-P2FEM-nvb5eta2T2uni.dat};
\addplot [color=black!50!red, mark=+, line width=0.5pt, mark size=3pt, mark options={solid, black!50!red}]                    table [x=nrelements,y=eta4] {datanumerics/mybsp2-P2FEM-nvb5eta2T2uni.dat};
\addplot [color=violet, mark=diamond, line width=0.5pt, mark size=3pt, mark options={solid, violet}]                          table [x=nrelements,y=eta5] {datanumerics/mybsp2-P2FEM-nvb5eta2T2uni.dat};
\addplot [color=black!25!yellow, mark=Mercedes star, line width=0.5pt, mark size=3pt, mark options={solid, black!25!yellow}]  table [x=nrelements,y=eta6] {datanumerics/mybsp2-P2FEM-nvb5eta2T2uni.dat};
\addplot [color=black, mark=o, line width=0.5pt, mark size=3pt, mark options={solid, black}]                                  table [x=nrelements,y=errorH1semiosc] {datanumerics/mybsp2-P2FEM-nvb5eta2T2uni.dat};
\addplot [color=black!50!green, mark=triangle, line width=0.5pt, mark size=3pt, mark options={solid, black!50!green}]         table [x=nrelements,y=eta1] {datanumerics/mybsp2-P2FEM-nvb5eta2T2ada.dat};
\addplot [color=black!50!blue, mark=triangle, line width=0.5pt, mark size=3pt, mark options={solid, rotate=90, black!50!blue}]table [x=nrelements,y=eta2] {datanumerics/mybsp2-P2FEM-nvb5eta2T2ada.dat};
\addplot [color=black!75!green, mark=pentagon, line width=0.5pt, mark size=2.5pt, mark options={solid, black!75!green}]       table [x=nrelements,y=eta3] {datanumerics/mybsp2-P2FEM-nvb5eta2T2ada.dat};
\addplot [color=black!50!red, mark=triangle, line width=0.5pt, mark size=3pt, mark options={solid, rotate=180, black!50!red}] table [x=nrelements,y=eta4] {datanumerics/mybsp2-P2FEM-nvb5eta2T2ada.dat};
\addplot [color=blue, mark=triangle, line width=0.5pt, mark size=3pt, mark options={solid, rotate=270, blue}]                 table [x=nrelements,y=eta5] {datanumerics/mybsp2-P2FEM-nvb5eta2T2ada.dat};
\addplot [color=brown, mark=oplus, line width=0.5pt, mark size=2.5pt, mark options={solid, brown}]                            table [x=nrelements,y=eta6] {datanumerics/mybsp2-P2FEM-nvb5eta2T2ada.dat};
\addplot [color=red, mark=square, line width=0.5pt, mark size=3pt, mark options={solid, red}]                                 table [x=nrelements,y=errorH1semiosc] {datanumerics/mybsp2-P2FEM-nvb5eta2T2ada.dat};
\logLogSlopeTriangle{0.8}{0.2}{0.875}{-1/3}{black}{\scriptsize};
\logLogSlopeTrianglelow{0.8}{0.2}{0.12}{-1}{black}{\scriptsize};

\legend{
$\lambda_\bullet'$ (unif.),
$\lambda_\bullet''$ (unif.),
$\lambda_\bullet'''$ (unif.),
$\mu_\bullet'$ (unif.),
$\mu_\bullet''$ (unif.),
$\mu_\bullet'''$ (unif.),
$\big(\norm{\nabla(u - u_\bullet)}{\Omega}^2+{{\rm osc}_\bullet^2}\big)^{1/2}$ (unif.),
$\lambda_\bullet'$ (adap.),
$\lambda_\bullet''$ (adap.),
$\lambda_\bullet'''$ (adap.),
$\mu_\bullet'$ (adap.),
$\mu_\bullet''$ (adap.),
$\mu_\bullet'''$ (adap.),
$\big(\norm{\nabla(u - u_\bullet)}{\Omega}^2+{{\rm osc}_\bullet^2}\big)^{1/2}$ (adap.),
}
\end{loglogaxis}
\end{tikzpicture}
}
\caption{Experiment from Section~\ref{ex2} with known solution with generic singularity.
We use the (minimal) refinement with (M3')
for both, uniform and adaptive refinement.
}
\label{fig:mybsp2errornvb5}
\end{figure}
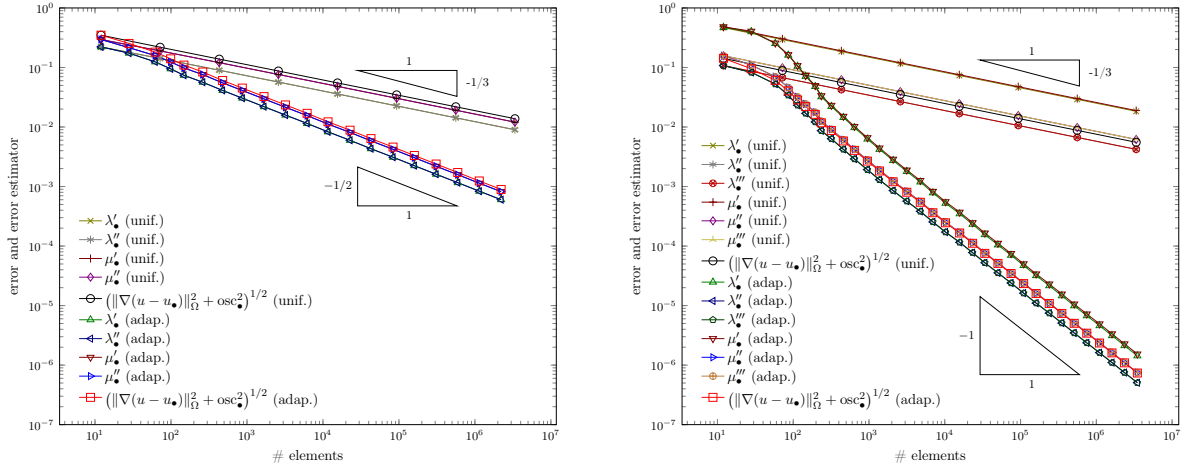
\subsection{Experiment with known solution with generic singularity}
\label{ex2}
We prescribe the exact solution in polar coordinates by
\begin{align*}
      u(x_1,x_2) = r^{2/3}\sin(2\varphi/3)
\quad\text{with } r\in[0,\infty), \varphi\in[0,2\pi).
\end{align*}
Hence, $f=0$ and the solution defines inhomogeneous
Dirichlet conditions. Furthermore, ${\rm osc}_\bullet=0$ and for ${\mathcal{S}}^1$-FEM even ${\rm res}_\bullet=0$.
This  implies that $\lambda_\bullet'=\lambda_\bullet''$ and $\mu_\bullet'=\mu_\bullet''$ for ${\mathcal{S}}^1$-FEM.
For ${\mathcal{S}}^2$-FEM, we additionally have that ${\rm apx}_\bullet=0$, which implies that $\lambda_\bullet''=\lambda_\bullet'''$ and $\mu_\bullet''=\mu_\bullet'''$.
It is well known that $u$ has a generic singularity at the reentrant corner $(0,0)$, which
leads to reduced regularity $u\in H^{1+2/3-\varepsilon}(\Omega)$ for all $\varepsilon>0$.
According to approximation theory we therefore get a reduced convergence order ${\mathcal{O}}(N^{-1/3})$ for
uniform refinement (M3), which is indeed observed in Figure~\ref{fig:mybsp2errornvb3}.
Our adaptive Algorithm~\ref{algorithm} recovers the optimal convergence rates
for ${\mathcal{S}}^1$-FEM and ${\mathcal{S}}^2$-FEM, which are plotted in Figure~\ref{fig:mybsp2errornvb3} for (M3)
and Figure~\ref{fig:mybsp2errornvb5} for (M3'). Hence, these figures also
verify Theorem~\ref{thm:abstract}.
In Figure~\ref{fig:effectivityreliabilityindexosc} and
Figure~\ref{fig:effectivityreliabilityindex}, we consider corresponding reliability and efficiency indices for adaptive refinement, which empirically confirm Remark~\ref{rem:estimator} (i) and (ii) (for inhomogeneous Dirichlet conditions).
Note that $C_{\rm son}=4$ for (M3) and $C_{\rm son}=6$ for (M3').

\begin{figure}
\centering
\subfigure[${\mathcal{S}}^1$-FEM with (M3) and $\eta_\bullet(T,\widehat u_\bullet)^2 :=\lambda_\bullet(T,\widehat u_\bullet)^2+{\rm res}_\bullet(T,\widehat u_\bullet)^2$.]{\label{subfig:mybsp3errornvb3P1}
\begin{tikzpicture}[scale=0.5]
\begin{loglogaxis}[width=0.92\textwidth,
xlabel={\small \# elements}, ylabel={\small error estimator}, font={\scriptsize},
ymin=1e-7,ymax=2e-0,
legend style={font=\small, draw=none, fill=none, cells={anchor=west}, legend pos=south west}]
\addplot [color=red!50!green, mark=x, line width=0.5pt, mark size=3pt, mark options={solid, red!50!green}]                     table [x=nrelements,y=eta1] {datanumerics/mybsp3-P1FEM-nvb3eta1T2uni.dat};
\addplot [color=gray, mark=asterisk, line width=0.5pt, mark size=3pt, mark options={solid, gray}]                              table [x=nrelements,y=eta2] {datanumerics/mybsp3-P1FEM-nvb3eta1T2uni.dat};
\addplot [color=black!50!red, mark=+, line width=0.5pt, mark size=3pt, mark options={solid, black!50!red}]                     table [x=nrelements,y=eta4] {datanumerics/mybsp3-P1FEM-nvb3eta1T2uni.dat};
\addplot [color=violet, mark=diamond, line width=0.5pt, mark size=3pt, mark options={solid, violet}]                           table [x=nrelements,y=eta5] {datanumerics/mybsp3-P1FEM-nvb3eta1T2uni.dat};
\addplot [color=black!50!green, mark=triangle, line width=0.5pt, mark size=3pt, mark options={solid, black!50!green}]          table [x=nrelements,y=eta1] {datanumerics/mybsp3-P1FEM-nvb3eta1T2ada.dat};
\addplot [color=black!50!blue, mark=triangle, line width=0.5pt, mark size=3pt, mark options={solid, rotate=90, black!50!blue}] table [x=nrelements,y=eta2] {datanumerics/mybsp3-P1FEM-nvb3eta1T2ada.dat};
\addplot [color=black!50!red, mark=triangle, line width=0.5pt, mark size=3pt, mark options={solid, rotate=180, black!50!red}]  table [x=nrelements,y=eta4] {datanumerics/mybsp3-P1FEM-nvb3eta1T2ada.dat};
\addplot [color=blue, mark=triangle, line width=0.5pt, mark size=3pt, mark options={solid, rotate=270, blue}]                  table [x=nrelements,y=eta5] {datanumerics/mybsp3-P1FEM-nvb3eta1T2ada.dat};
\logLogSlopeTriangle{0.8}{0.2}{0.78}{-1/3}{black}{\scriptsize};
\logLogSlopeTrianglelow{0.8}{0.2}{0.525}{-1/2}{black}{\scriptsize};

\legend{
$\lambda_\bullet'$ (unif.),
$\lambda_\bullet''$ (unif.),
$\mu_\bullet'$ (unif.),
$\mu_\bullet''$ (unif.),
$\lambda_\bullet'$ (adap.),
$\lambda_\bullet''$ (adap.),
$\mu_\bullet'$ (adap.),
$\mu_\bullet''$ (adap.),
}
\end{loglogaxis}
\end{tikzpicture}
}
\hspace{0.015\textwidth}
\subfigure[${\mathcal{S}}^2$-FEM with (M3) and $\eta_\bullet(T,\widehat u_\bullet)^2 :=\lambda_\bullet(T,\widehat u_\bullet)^2+{\rm apx}_\bullet(T)^2$.]{\label{subfig:mybsp3errornvb3P2}
\begin{tikzpicture}[scale=0.5]
\begin{loglogaxis}[width=0.92\textwidth,
xlabel={\small \# elements}, ylabel={\small error estimator}, font={\scriptsize},
ymin=1e-7,ymax=2e-0,
legend style={font=\small, draw=none, fill=none, cells={anchor=west}, legend pos=south west}]
\addplot [color=red!50!green, mark=x, line width=0.5pt, mark size=3pt, mark options={solid, red!50!green}]                    table [x=nrelements,y=eta1] {datanumerics/mybsp3-P2FEM-nvb3eta3T2uni.dat};
\addplot [color=gray, mark=asterisk, line width=0.5pt, mark size=3pt, mark options={solid, gray}]                             table [x=nrelements,y=eta2] {datanumerics/mybsp3-P2FEM-nvb3eta3T2uni.dat};
\addplot [color=black!25!red, mark=otimes, line width=0.5pt, mark size=2.5pt, mark options={solid, black!25!red}]             table [x=nrelements,y=eta3] {datanumerics/mybsp3-P2FEM-nvb3eta3T2uni.dat};
\addplot [color=black!50!red, mark=+, line width=0.5pt, mark size=3pt, mark options={solid, black!50!red}]                    table [x=nrelements,y=eta4] {datanumerics/mybsp3-P2FEM-nvb3eta3T2uni.dat};
\addplot [color=violet, mark=diamond, line width=0.5pt, mark size=3pt, mark options={solid, violet}]                          table [x=nrelements,y=eta5] {datanumerics/mybsp3-P2FEM-nvb3eta3T2uni.dat};
\addplot [color=black!25!yellow, mark=Mercedes star, line width=0.5pt, mark size=3pt, mark options={solid, black!25!yellow}]  table [x=nrelements,y=eta6] {datanumerics/mybsp3-P2FEM-nvb3eta3T2uni.dat};
\addplot [color=black!50!green, mark=triangle, line width=0.5pt, mark size=3pt, mark options={solid, black!50!green}]         table [x=nrelements,y=eta1] {datanumerics/mybsp3-P2FEM-nvb3eta3T2ada.dat};
\addplot [color=black!50!blue, mark=triangle, line width=0.5pt, mark size=3pt, mark options={solid, rotate=90, black!50!blue}]table [x=nrelements,y=eta2] {datanumerics/mybsp3-P2FEM-nvb3eta3T2ada.dat};
\addplot [color=black!75!green, mark=pentagon, line width=0.5pt, mark size=2.5pt, mark options={solid, black!75!green}]       table [x=nrelements,y=eta3] {datanumerics/mybsp3-P2FEM-nvb3eta3T2ada.dat};
\addplot [color=black!50!red, mark=triangle, line width=0.5pt, mark size=3pt, mark options={solid, rotate=180, black!50!red}] table [x=nrelements,y=eta4] {datanumerics/mybsp3-P2FEM-nvb3eta3T2ada.dat};
\addplot [color=blue, mark=triangle, line width=0.5pt, mark size=3pt, mark options={solid, rotate=270, blue}]                 table [x=nrelements,y=eta5] {datanumerics/mybsp3-P2FEM-nvb3eta3T2ada.dat};
\addplot [color=brown, mark=oplus, line width=0.5pt, mark size=2.5pt, mark options={solid, brown}]                            table [x=nrelements,y=eta6] {datanumerics/mybsp3-P2FEM-nvb3eta3T2ada.dat};
\logLogSlopeTriangle{0.8}{0.2}{0.75}{-1/3}{black}{\scriptsize};
\logLogSlopeTrianglelow{0.8}{0.2}{0.15}{-1}{black}{\scriptsize};

\legend{
$\lambda_\bullet'$ (unif.),
$\lambda_\bullet''$ (unif.),
$\lambda_\bullet'''$ (unif.),
$\mu_\bullet'$ (unif.),
$\mu_\bullet''$ (unif.),
$\mu_\bullet'''$ (unif.),
$\lambda_\bullet'$ (adap.),
$\lambda_\bullet''$ (adap.),
$\lambda_\bullet'''$ (adap.),
$\mu_\bullet'$ (adap.),
$\mu_\bullet''$ (adap.),
$\mu_\bullet'''$ (adap.),
}
\end{loglogaxis}
\end{tikzpicture}
}
\caption{Experiment from Section~\ref{ex3} with  unknown solution with generic singularity. We use the (minimal) refinement with (M3)
for both, uniform and adaptive refinement.
}
\label{fig:mybsp3errornvb3}
\end{figure}
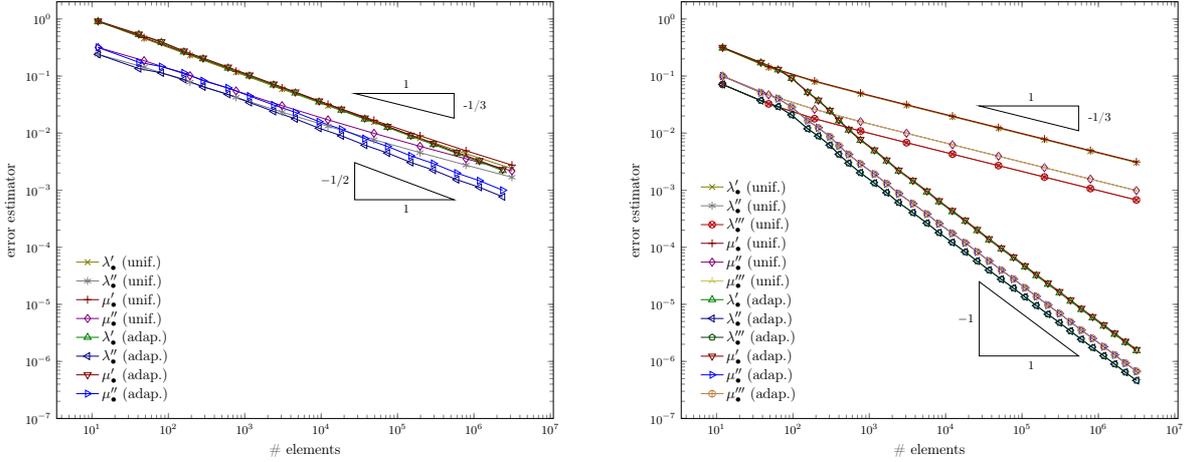

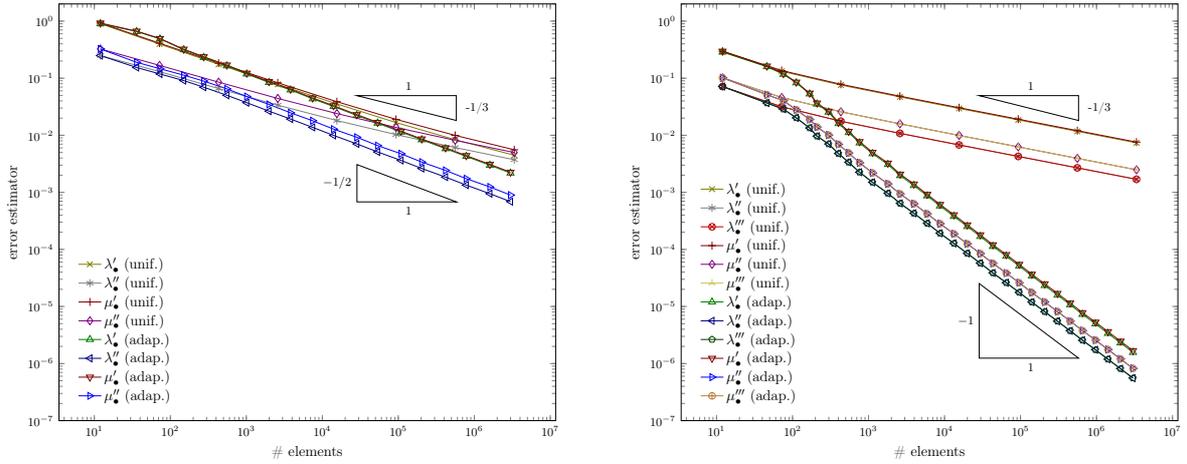
\begin{figure}
\centering
\subfigure[${\mathcal{S}}^1$-FEM with (M3') and $\eta_\bullet(T,\widehat u_\bullet)^2 :=\lambda_\bullet(T,\widehat u_\bullet)^2+{\rm osc}_\bullet(T)^2$.]{\label{subfig:mybsp3errornvb5P1}
\begin{tikzpicture}[scale=0.5]
\begin{loglogaxis}[width=0.92\textwidth,
xlabel={\small \# elements}, ylabel={\small error estimator}, font={\scriptsize},
ymin=1e-7,ymax=2e-0,
legend style={font=\small, draw=none, fill=none, cells={anchor=west}, legend pos=south west}]
\addplot [color=red!50!green, mark=x, line width=0.5pt, mark size=3pt, mark options={solid, red!50!green}]                     table [x=nrelements,y=eta1] {datanumerics/mybsp3-P1FEM-nvb5eta2T2uni.dat};
\addplot [color=gray, mark=asterisk, line width=0.5pt, mark size=3pt, mark options={solid, gray}]                              table [x=nrelements,y=eta2] {datanumerics/mybsp3-P1FEM-nvb5eta2T2uni.dat};
\addplot [color=black!50!red, mark=+, line width=0.5pt, mark size=3pt, mark options={solid, black!50!red}]                     table [x=nrelements,y=eta4] {datanumerics/mybsp3-P1FEM-nvb5eta2T2uni.dat};
\addplot [color=violet, mark=diamond, line width=0.5pt, mark size=3pt, mark options={solid, violet}]                           table [x=nrelements,y=eta5] {datanumerics/mybsp3-P1FEM-nvb5eta2T2uni.dat};
\addplot [color=black!50!green, mark=triangle, line width=0.5pt, mark size=3pt, mark options={solid, black!50!green}]          table [x=nrelements,y=eta1] {datanumerics/mybsp3-P1FEM-nvb5eta2T2ada.dat};
\addplot [color=black!50!blue, mark=triangle, line width=0.5pt, mark size=3pt, mark options={solid, rotate=90, black!50!blue}] table [x=nrelements,y=eta2] {datanumerics/mybsp3-P1FEM-nvb5eta2T2ada.dat};
\addplot [color=black!50!red, mark=triangle, line width=0.5pt, mark size=3pt, mark options={solid, rotate=180, black!50!red}]  table [x=nrelements,y=eta4] {datanumerics/mybsp3-P1FEM-nvb5eta2T2ada.dat};
\addplot [color=blue, mark=triangle, line width=0.5pt, mark size=3pt, mark options={solid, rotate=270, blue}]                  table [x=nrelements,y=eta5] {datanumerics/mybsp3-P1FEM-nvb5eta2T2ada.dat};
\logLogSlopeTriangle{0.8}{0.2}{0.78}{-1/3}{black}{\scriptsize};
\logLogSlopeTrianglelow{0.8}{0.2}{0.525}{-1/2}{black}{\scriptsize};

\legend{
$\lambda_\bullet'$ (unif.),
$\lambda_\bullet''$ (unif.),
$\mu_\bullet'$ (unif.),
$\mu_\bullet''$ (unif.),
$\lambda_\bullet'$ (adap.),
$\lambda_\bullet''$ (adap.),
$\mu_\bullet'$ (adap.),
$\mu_\bullet''$ (adap.),
}
\end{loglogaxis}
\end{tikzpicture}
}
\hspace{0.015\textwidth}
\subfigure[${\mathcal{S}}^2$-FEM with (M3') and $\eta_\bullet(T,\widehat u_\bullet)^2 :=\lambda_\bullet(T,\widehat u_\bullet)^2+{\rm osc}_\bullet(T)^2$.]{\label{subfig:mybsp3errornvb5P2}
\begin{tikzpicture}[scale=0.5]
\begin{loglogaxis}[width=0.92\textwidth,
xlabel={\small \# elements}, ylabel={\small error estimator}, font={\scriptsize},
ymin=1e-7,ymax=2e-0,
legend style={font=\small, draw=none, fill=none, cells={anchor=west}, legend pos=south west}]
\addplot [color=red!50!green, mark=x, line width=0.5pt, mark size=3pt, mark options={solid, red!50!green}]                    table [x=nrelements,y=eta1] {datanumerics/mybsp3-P2FEM-nvb5eta2T2uni.dat};
\addplot [color=gray, mark=asterisk, line width=0.5pt, mark size=3pt, mark options={solid, gray}]                             table [x=nrelements,y=eta2] {datanumerics/mybsp3-P2FEM-nvb5eta2T2uni.dat};
\addplot [color=black!25!red, mark=otimes, line width=0.5pt, mark size=2.5pt, mark options={solid, black!25!red}]             table [x=nrelements,y=eta3] {datanumerics/mybsp3-P2FEM-nvb5eta2T2uni.dat};
\addplot [color=black!50!red, mark=+, line width=0.5pt, mark size=3pt, mark options={solid, black!50!red}]                    table [x=nrelements,y=eta4] {datanumerics/mybsp3-P2FEM-nvb5eta2T2uni.dat};
\addplot [color=violet, mark=diamond, line width=0.5pt, mark size=3pt, mark options={solid, violet}]                          table [x=nrelements,y=eta5] {datanumerics/mybsp3-P2FEM-nvb5eta2T2uni.dat};
\addplot [color=black!25!yellow, mark=Mercedes star, line width=0.5pt, mark size=3pt, mark options={solid, black!25!yellow}]  table [x=nrelements,y=eta6] {datanumerics/mybsp3-P2FEM-nvb5eta2T2uni.dat};
\addplot [color=black!50!green, mark=triangle, line width=0.5pt, mark size=3pt, mark options={solid, black!50!green}]         table [x=nrelements,y=eta1] {datanumerics/mybsp3-P2FEM-nvb5eta2T2ada.dat};
\addplot [color=black!50!blue, mark=triangle, line width=0.5pt, mark size=3pt, mark options={solid, rotate=90, black!50!blue}]table [x=nrelements,y=eta2] {datanumerics/mybsp3-P2FEM-nvb5eta2T2ada.dat};
\addplot [color=black!75!green, mark=pentagon, line width=0.5pt, mark size=2.5pt, mark options={solid, black!75!green}]       table [x=nrelements,y=eta3] {datanumerics/mybsp3-P2FEM-nvb5eta2T2ada.dat};
\addplot [color=black!50!red, mark=triangle, line width=0.5pt, mark size=3pt, mark options={solid, rotate=180, black!50!red}] table [x=nrelements,y=eta4] {datanumerics/mybsp3-P2FEM-nvb5eta2T2ada.dat};
\addplot [color=blue, mark=triangle, line width=0.5pt, mark size=3pt, mark options={solid, rotate=270, blue}]                 table [x=nrelements,y=eta5] {datanumerics/mybsp3-P2FEM-nvb5eta2T2ada.dat};
\addplot [color=brown, mark=oplus, line width=0.5pt, mark size=2.5pt, mark options={solid, brown}]                            table [x=nrelements,y=eta6] {datanumerics/mybsp3-P2FEM-nvb5eta2T2ada.dat};
\logLogSlopeTriangle{0.8}{0.2}{0.78}{-1/3}{black}{\scriptsize};
\logLogSlopeTrianglelow{0.8}{0.2}{0.15}{-1}{black}{\scriptsize};

\legend{
$\lambda_\bullet'$ (unif.),
$\lambda_\bullet''$ (unif.),
$\lambda_\bullet'''$ (unif.),
$\mu_\bullet'$ (unif.),
$\mu_\bullet''$ (unif.),
$\mu_\bullet'''$ (unif.),
$\lambda_\bullet'$ (adap.),
$\lambda_\bullet''$ (adap.),
$\lambda_\bullet'''$ (adap.),
$\mu_\bullet'$ (adap.),
$\mu_\bullet''$ (adap.),
$\mu_\bullet'''$ (adap.),
}
\end{loglogaxis}
\end{tikzpicture}
}
\caption{Experiment from Section~\ref{ex3} with  unknown solution with generic singularity. We use the (minimal) refinement with (M3')
for both, uniform and adaptive refinement.
}
\label{fig:mybsp3errornvb5}
\end{figure}
\subsection{Experiment with unknown solution with generic singularity}
\label{ex3}
For this example, we define $f=1$ in $\Omega$ and $u=0$ on $\Gamma$. The solution is unknown.
Therefore, we only plot the estimators in Figure~\ref{fig:mybsp3errornvb3} for (M3)
and in Figure~\ref{fig:mybsp3errornvb5} for (M3'). All estimators are efficient and reliable. Hence, the convergence rate
of our numerical solution is observed by the asymptotics of the estimators.
As in Example~\ref{ex2}, uniform mesh refinement leads to a suboptimal convergence rate,
whereas Algorithm~\ref{algorithm} reproduces the optimal rates by adaptive mesh refinement.
For ${\mathcal{S}}^2$-FEM, we observe that $\lambda_\bullet''=\lambda_\bullet'''$ and $\mu_\bullet''=\mu_\bullet'''$ since $f$ is constant.

\textbf{Acknowledgement.}\quad
The authors are supported by the Austrian Science Fund (FWF) through the research projects \textit{Optimal isogeometric boundary element method} (grant P29096) and \textit{Optimal adaptivity for BEM and FEM-BEM coupling} (grant P27005), the doctoral school \textit{Dissipation and dispersion in nonlinear PDEs} (grant W1245), and the special research program \textit{Taming complexity in PDE systems} (grant SFB F65).

\bibliographystyle{alpha}
\bibliography{literature}

\end{document}